\documentclass{amsart}
\usepackage{mathrsfs}
\usepackage{amssymb,amsmath}
\usepackage{graphicx}
\usepackage{comment}
\usepackage{tikz}
\usepackage{tikz-cd}
\usepackage{pgfplots}
\usetikzlibrary{math}
\usepackage[pdftex]{hyperref}
\usepackage{flowchart}
\usepackage[all,cmtip]{xy}
\usetikzlibrary{matrix,calc,arrows}

\newcommand{\bC}{{\mathbb C}}
\newcommand{\bP}{{\mathbb P}}
\newcommand{\bR}{{\mathbb R}}
\newcommand{\bZ}{{\mathbb Z}}
\newcommand{\bN}{{\mathbb N}}

\newcommand{\bT}{{\mathbb T}}
\newcommand{\bF}{{\mathbb F}}
\newcommand{\bfL}{{\mathbf L}}

\newcommand{\bfH}{{\mathbf H}}

\newcommand{\fF}{\pi}

\newcommand{\cF}{\mathcal F}

\newcommand{\cD}{\mathcal D}

\newcommand{\cM}{\mathcal M}
\newcommand{\cO}{\mathcal O}

\newcommand{\cV}{\mathcal V}

\newcommand{\cR}{\mathcal R}

\newcommand{\cH}{\mathcal H}

\newcommand{\scrX}{\mathscr X}

\newcommand{\scrO}{\mathscr O}

\newcommand{\fd}{\mathfrak{d}}
\newcommand{\Mbar}{\overline{\cM}}

\newcommand{\Pin}{\operatorname{Pin}}

\newcommand{\Ob}{\operatorname{Ob}}

\newcommand{\id}{\operatorname{id}}

\newcommand{\Ch}{\operatorname{Ch}}

\newcommand{\Hom}{\operatorname{Hom}}

\newcommand{\Fuk}{\operatorname{Fuk}}

\newcommand{\val}{\operatorname{val}}

\newcommand{\ro}{\operatorname{o}}

\newcommand{\Color}{\operatorname{Col}}
\newcommand{\rmodz}{{\mathbb{R}/\mathbb{Z}}}
\newcommand{\novr}{\Lambda_{\geq 0}}
\newcommand{\novf}{\Lambda}

\DeclareMathOperator*{\Cech}{\operatorname{\check{C}ech}}
\def\co{\colon\thinspace}

\numberwithin{equation}{section}

\newtheorem{thm}{Theorem}[section]
\newtheorem{cor}[thm]{Corollary}
\newtheorem{lem}[thm]{Lemma}
\newtheorem{lemma}[thm]{Lemma}
\newtheorem{prop}[thm]{Proposition}
\newtheorem{defin}[thm]{Definition}
\newtheorem{def-lem}[thm]{Definition-Lemma}

\theoremstyle{remark}

\newtheorem{rem}[thm]{Remark}

\newtheorem{example}[thm]{Example}
\newtheorem{examples}[thm]{Examples}

\author{Mohammed Abouzaid}
\address{Department of Mathematics, Stanford University, Stanford, CA, USA}
\email{abouzaid@stanford.edu}

\author{Yoel Groman}
\address{Einstein Institute of Mathematics, The Hebrew University of Jerusalem, Jerusalem, Israel}
\email{yoel.groman@mail.huji.ac.il}

\author{Umut Varolgunes}
\address{Department of Mathematics, Koç University, Istanbul, Turkey}
\email{uvarolgunes@ku.edu.tr}
\title{Fukaya category with supports}
\date{\today}

\begin{document}

\begin{abstract}
We associate to each compact subset $K$ of a closed symplectic manifold $M$
an $A_\infty$-category $\Fuk_K$, called the Fukaya category with support on $K$. The objects are tautologically unobstructed global Lagrangians endowed with appropriate decorations. The dependence on $K$ is detected by the morphism complexes which can be computed by completed telescopes associated with acceleration data for $K$. In particular, these vanish for Lagrangians disjoint of $K$.  For $K=M$ this recovers the usual Fukaya category of tautologically unobstructed objects. 
These categories admit restriction functors under inclusions of compact
subsets. We prove a Mayer--Vietoris descent theorem
for weakly involutive covers. The construction is carried out over the Novikov ring and is established whenever classical transversality methods
apply.
\end{abstract}
\maketitle
\tableofcontents

\section{Introduction}
\label{sec:introduction}
Let $M$ be a closed symplectic manifold. The goal of this paper is to assign to each compact subset $K$ of $M$ an $A_\infty$ category $$K\mapsto \Fuk_K$$ called the Fukaya category with supports, and to establish some of its fundamental properties: the existence of restriction functors and a generalization of the descent property \cite{Varolgunes2021} (see Section~\ref{sec:def-fuk-with-support} and Theorem~\ref{thm-descent}). This structure has an immense amount of quantitative information even for the case $M=T^2$. The paper provides a complete proof in the cases where classical transversality methods apply.

The object set of $\Fuk_K$ is independent of $K$ and consists, to a first approximation,  of 
embedded Lagrangians $L\to M$ that do not bound pseudo-holomorphic discs. This is a first approximation because we in fact allow the possibility that they bound discs whose Maslov index is $2$ or more, and equip the Lagrangians with brane data (see Section~\ref{sec:lagrangian-branes}). In addition to the main innovation of considering a support condition, the paper has the following features:
\begin{enumerate}
\item We do not impose any transversality condition between Lagrangians, so that the Fukaya category that we associate to the torus consists, for example, of \emph{all} essential curves.
\item We perform all constructions over the Novikov ring, i.e. without inverting the Novikov parameter.
\end{enumerate}

These features are not completely new, as it is plausible that they could be achieved by the methods of \cite{Fukaya2021GH}, but our methods, while similar in spirit (take a limit over perturbations) are quite different in practice. In particular, our construction results in morphism spaces between two Lagrangians $(L_\ell,L_r)$ that take into account all other Lagrangians;  these are truly gigantic chain complexes, which are constructed, as in \cite{Abouzaid2023}, by the use of methods of homotopical algebra, with operads being replaced by multicategories, and the $2$-sided bar construction serving as an explicit model for (weighted) homotopy colimits. We verify that our construction has the expected properties by showing that our morphism spaces are homotopy equivalent to those introduced in the last author's thesis, as a (completed) direct limit of Floer cohomology groups over sequences of Hamiltonians that converge to an indicator function of $K$.

In the remainder of this introduction, we provide a more explicit description of the Fukaya category (continuing to omit technical details), then we provide a precise statement of our main results.

\subsection{A description of the construction}

In our quest to simplify the discussion, let us presume that we are in a situation where, for some choices of almost complex structures which we will suppress in this introduction, Floer theory assigns to a pair of Lagrangians $(L_\ell,L_r)$, and a Hamiltonian $H$ whose time-$1$ flow maps $L_\ell$ to a Lagrangian which is transverse to $L_r$, a chain complex $CF^*(L_\ell,L_r; H)$ over the Novikov ring (i.e. formal series $\sum_{0 \leq \lambda} a_\lambda T^\lambda $ with the property that the set of exponents $\lambda$ with non-vanishing coefficient $a_\lambda$ is discrete); this is a relatively standard outcome of Floer theory, once we have dealt with the possibility that $L_\ell$ and $L_r$ bound holomorphic discs, chosen regular almost complex structures, and fixed some auxiliary data such as local systems on the Lagrangians, or bulk structures on the ambient symplectic manifold. We may presume as well that, given a homotopy $H_s$ of Hamiltonians parametrized by $s \in \bR$, which is monotone in the sense that $\partial_s H_s$ is nowhere negative, and limits to Hamiltonians $H_-$ and $H_+$,  the same Floer-theoretic methods yield a \emph{continuation map}
\begin{equation} \label{eq:continuation_maps}
  CF^*(L_\ell,L_r; H_-) \to CF^*(L_\ell,L_r; H_+).
\end{equation}

In this setting, \cite{Varolgunes2021} defines the Floer cohomology $HF_K(L_\ell,L_r)$ with support on $K$ (initially called the local Floer cohomology) as the cohomology of the completion of the direct limit in Equation \eqref{eq:continuation_maps}, taken over a sequence of Hamiltonians which converge to the indicator function which vanishes on $K$ and is $+\infty$ on its complement. It is important at this point to realize that this formulation only makes sense over the Novikov ring (i.e. before inverting the Novikov parameter), although there are alternative approaches over the Novikov field, by requiring that the maps do not decrease the valuations. In either case, the construction makes sense exactly because of our assumption that the homotopy is monotone, which ensures that Equation \eqref{eq:continuation_maps} is defined prior to inverting this parameter.

There are three issues that arise:
\begin{enumerate}
\item Ensuring that we can define this complex without assuming that the Lagrangians are transverse.
\item Defining the restriction maps associated to inclusions of compact subsets $K \subset K'$.
\item Constructing $A_\infty$ operations on these chain complexes (in a way compatible with these restriction maps).
\end{enumerate}
It turns out that the first issue is completely straightforward to resolve: the natural transversality condition is not on the pair $L_\ell$ and $L_r$, but on the triple $(L_\ell, L_r, H)$, and for any pair of Lagrangians we can find sufficiently many time-dependent functions $H$ to ensure that the image of $L_\ell$ under the Hamiltonian flow of $H$ is transverse to $L_r$. In this way, the (ordinary) Floer cohomology of a pair of Lagrangians which are not transverse can be defined as the direct limit of the Floer cohomology groups of all strictly negative Hamiltonians (cf.\ Fukaya \cite{Fukaya2010}).

If the subsets $K$ and $K'$ are fixed in advance, we may arrange for an appropriate interleaving of the sequences of smooth functions approximating the indicator functions of $K$ and $K'$. Analyzing what happens for a triple of Lagrangians becomes more complicated, so it is tempting to consider the partially ordered set of all Hamiltonians which are non-positive on a given subset of $M$. The set associated to $K'$  is a strict subset of the one associated to $K$, so one can reasonably hope that a chain level map realizing the restriction $ HF_{K'}(L_\ell,L_r) \to HF_K(L_\ell,L_r)$ can be defined essentially without any choices. There is only one wrinkle in this part of the story: at the chain level, the continuation map in Equation \eqref{eq:continuation_maps} depends on the choice of the interpolation of Hamiltonians.  The approach we take, which seems inevitable once  essentially all Hamiltonians are admitted, is to record the data of all Hamiltonians, as well as all continuation maps between them, and all homotopies between such continuation maps, which, unsurprisingly, leads to considering families of continuation maps of arbitrary dimension. We choose to encode this data in the language of cubically enriched categories, as we did in \cite{Abouzaid2023} for Hamiltonian Floer theory, though alternative approaches are also possible.

This brings us to the third issue which requires a generalization, not so extensive in the grand scheme of things, of the constructions of \cite{Abouzaid2023}, to account for the presence of multiple Lagrangians, but replacing operads by multicategories. Heuristically, the goal now is to define a product operation arising from the pair of pants product
\begin{equation} \label{eq:pair_of_pants_product} CF^*(L_\ell,L_1; H_{\ell 1})  \otimes CF^*(L_1,L_r; H_{1r}) \to CF^*(L_\ell,L_r; H_{\ell r}), \end{equation}
for appropriate choices of Hamiltonians $H_{\ell 1}$, $H_{1r}$, and $H_{\ell r}$, but the fundamental issue is again that we have to record all such choices, all the ways that they can be composed, and all the possible (higher) homotopies between them. This is what the multicategory records. 






Concretely, the chain complex of morphisms from $L_\ell$ to $L_r$ in $\Fuk_M(K)$ is defined by a procedure involving discrete choices over the edges (and leaves) of a tree, and continuous choices over its vertices: for the discrete parts, recall that a levelled planar tree is the datum of a finite sequence of composable order-preserving surjective maps between totally ordered sets, consisting of at least one such map, so that the target of the last map is a singleton. In other words, we have a positive integer $k$ (the length of the sequence), a collection of positive integers $n_1,\ldots,n_k$ and order-preserving surjective maps $$[n_k] \twoheadrightarrow \ldots \twoheadrightarrow [n_1]\twoheadrightarrow [1],$$ where $[n]$ is the set $\{1,\ldots, n\}$. Such  data can be encoded by a tree, properly embedded in the plane,  with $n_k$ semi-infinite incoming edges and $1$ semi-infinite outgoing edge. We visualize the tree as going from top to bottom.

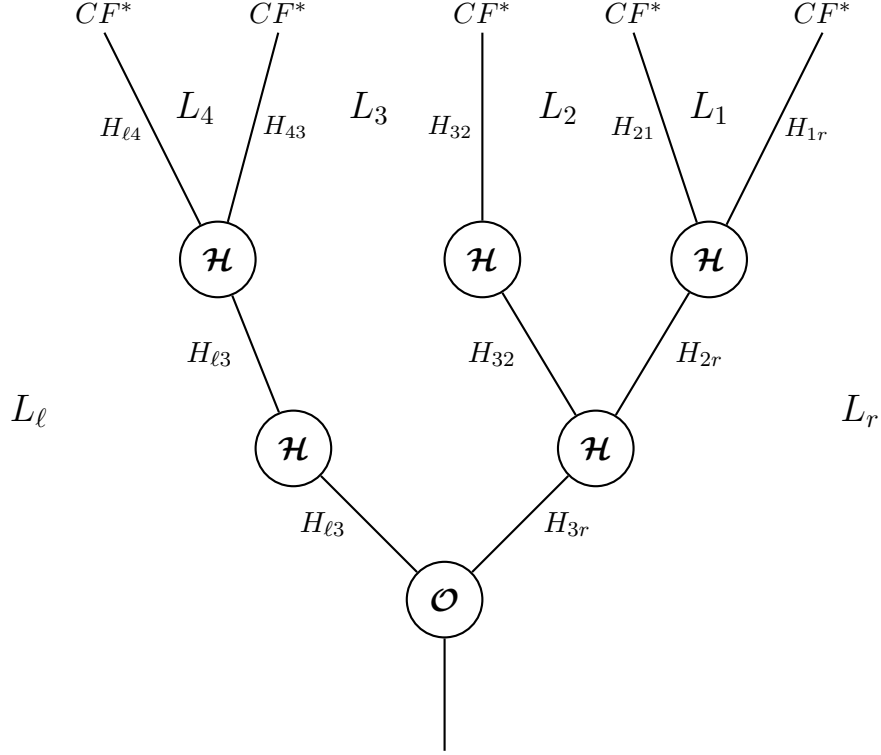
\begin{figure}
    \centering
   \begin{tikzpicture}[
    scale=1,
    every node/.style={transform shape},
    blob/.style={circle, draw, thick, minimum size=1cm, inner sep=0pt, font=\Large\boldmath},
    edge label/.style={midway, font=\large, auto}
]

    
    \node[blob] (O) at (0,0) {$\mathcal{O}$};

    \node[blob] (H_L1) at (-2, 2) {$\mathcal{H}$};
    \node[blob] (H_R1) at (2, 2) {$\mathcal{H}$};

    \node[blob] (H_L2) at (-3, 4.5) {$\mathcal{H}$};

    \node[blob] (H_M2) at (0.5, 4.5) {$\mathcal{H}$};
    \node[blob] (H_R2) at (3.5, 4.5) {$\mathcal{H}$};

    \draw[thick] (O) -- (0, -2);


    \draw[thick] (O) -- (H_L1) node[edge label, left=5pt] {$H_{\ell 3}$};
    \draw[thick] (O) -- (H_R1) node[edge label, right=5pt] {$H_{3r}$};

    \draw[thick] (H_L1) -- (H_L2) node[edge label, left=5pt] {$H_{\ell 3}$};

    \draw[thick] (H_R1) -- (H_M2) node[edge label, left=5pt] {$H_{32}$};
    \draw[thick] (H_R1) -- (H_R2) node[edge label, right=5pt] {$H_{2r}$};


    \draw[thick] (H_L2) -- ++(-1.5, 3) node[midway, left] {$H_{\ell 4}$} node[above, font=\large] {$CF^*$};
    \draw[thick] (H_L2) -- ++(0.8, 3) node[midway, right] {$H_{43}$} node[above, font=\large] {$CF^*$};

    \draw[thick] (H_M2) -- ++(0, 3) node[midway, left] {$H_{32}$} node[above, font=\large] {$CF^*$};

    \draw[thick] (H_R2) -- ++(-1.0, 3) node[midway, left] {$H_{21}$} node[above, font=\large] {$CF^*$};
    \draw[thick] (H_R2) -- ++(1.5, 3) node[midway, right] {$H_{1r}$} node[above, font=\large] {$CF^*$};

    
    \node[font=\LARGE] at (-3.3, 6.5) {$L_4$};
    
    \node[font=\LARGE] at (-1, 6.5) {$L_3$};
    
    \node[font=\LARGE] at (1.5, 6.5) {$L_2$};
    
    \node[font=\LARGE] at (3.5, 6.5) {$L_1$};
    
    \node[font=\LARGE] at (-5.5, 2.5) {$L_{\ell}$};
    \node[font=\LARGE] at (5.5, 2.5) {$L_{r}$};

\end{tikzpicture}
    \caption{A Lagrangian and Hamiltonian labelled planar tree with underlying maps $[5]\to [3]\to [2]\to [1]$..}
    \label{fig:labelledtree}
\end{figure}

Given a levelled planar tree, represented by a planar embedding $T\subset \mathbb{R}^2$, we define a Lagrangian labelling to be a labelling of the connected components of $\mathbb{R}^2\setminus T$ by Lagrangians. A Hamiltonian labelling is a labelling of each of the edges $e$ of $T$ by a function $H_e:M\times [0,1]\to \mathbb{R}$ such that the time-$1$ map of $H_e$ maps the Lagrangian to the left of $e$ to a submanifold of $M$ which is transverse to the Lagrangian to the right of $e$. This completes the description of the discrete data that goes into the definition of the multicategory of Lagrangian and Hamiltonian data.
To describe the continuous data, note that we may associate to each vertex $v$ of valence $n_v+1$ a sequence $(L_{v;0}, \cdots, L_{v;n_v})$ of Lagrangians labelling the adjacent regions, as well as Hamiltonians $H_{v;i,i+1}:M\times [0,1]\to \mathbb{R}$ for $i=0,\ldots,n_v-1$ and $H_{v;0,n_v}:M\times [0,1]\to \mathbb{R}$ satisfying the previously described transversality condition. In this paper, we shall construct a chain complex
 $$\mathcal{H}_v:=C_{-*}(L_{v;0},\ldots, L_{v;n_v};H_{v;0,1},\ldots, H_{v;n_v-1,n_v}; H_{v; 0,n_v})$$
 of all cubical families of monotone Floer data on an $(n_v+1)$-punctured disc: the $0$-chains are (linear combinations of) choices of Hamiltonians (and almost complex structures, which we are suppressing), parametrized by a (possibly broken) $(n_v+1)$-punctured disc, with boundary segments labelled by the Lagrangians $L_{v;k}$, and so that the Hamiltonian data is given, near each puncture in strip-like coordinates, by the Floer equation with Hamiltonian data $H_{v; i,i+1}$. The $1$-chains are homotopies between such data for varying modulus (appropriately defined to account for breaking), and the higher chains generalize the construction to higher dimensions.

 We pause here to remind the reader that this is essentially the data that goes in the usual construction of the Fukaya category: the difference is that the traditional construction is inductive and essentially corresponds to a family of choices of Hamiltonian data parametrized by the moduli space of punctured discs, so that we have one datum for each possible element of the moduli space. Here, instead, we allow all such choices.

 Inspecting Figure \ref{fig:labelledtree} should give an indication for how these complexes are used: for each discrete choice (tree labelled by Hamiltonians and Lagrangians) we shall take the tensor product (as complexes of abelian groups) of:
 \begin{enumerate}
 \item  the Lagrangian Floer cochain complexes $CF^*(L_i, L_{i+1}; H_{i,i+1})$ indexed by the incoming semi-infinite edges,
 \item the chain complexes $\mathcal{H}_v$ described above at each vertex $v$ which is not adjacent to the outgoing edge (we keep the notation short by not writing the relevant data).
 \item the complex $\mathcal{O}_{n_v}:=C_{-*}(Disks_{n_1,1})$,  of cubical families of $(n_1+1)-$punctured disks whenever $v$ is the unique vertex which is adjacent to the outgoing semi-infinite edge (recall that $n_1$ is the valence of this vertex).
 \end{enumerate}

 The Floer complex for the pair of Lagrangians $L_\ell$ and $L_r$ will be the (completed) direct sum, over all Lagrangian and Hamiltonian labelled planar trees of these tensor products (shifted by one less than the number of levels). The differential on this total complex is upper triangular with respect to the ordering of trees by the number of levels: the diagonal entries are given by the differential on the tensor product $ \mathcal{O} \otimes \bigotimes_{v} \mathcal{H}_v \otimes \bigotimes_{i} CF^*(L_i, L_{i+1}; H_{i,i+1})$. The off-diagonal entries are only non-zero for labelled trees which differ by collapsing a level. There are three cases to consider:
 \begin{enumerate}
 \item Collapsing the top-most level (the one carrying all the incoming leaves) corresponds to actions of the first level of Hamiltonian data on Lagrangian Floer cochains, i.e. the tensor product of the maps
   \begin{equation}
   \cH_{v} \otimes \bigotimes_{i=0}^{n_v-1} CF^*(L_{v;i}, L_{v;i+1}; H_{v;i,i+1}) \to CF^*(L_{v;0}, L_{v;n_v}; H_{v;0,n_v}).
 \end{equation}
\item Collapsing an inner level corresponds to the map on Floer data given by concatenation (not gluing!) of Riemann surfaces  $ \bigotimes_{v_i} \mathcal{H}_{v_i} \to \mathcal{H}_v$, where the $v_i$ are the vertices in the source tree that are collapsed.
  
\item Collapsing  the bottom level corresponds to applying the forgetful map to $\mathcal{O}$ (i.e. remembering only the underlying punctured disc, not the Floer data), to the last level of complexes $\cH_v$, followed by the map of cubical chains induced by concatenation of stable discs:
  \begin{equation} \bigotimes_i \cH_{v_i} \to \bigotimes_i \cO_{n_{v_i}} \to  \cO_{\sum n_{v_i}}. \end{equation}
\end{enumerate}
Taking the completion of this complex (where the Hamiltonians are again always assumed to be non-positive on $K$) yields the morphism space $\Fuk_K(L_\ell,L_r)$ between these Lagrangians in the Fukaya category with support on $K$. This assignment of a tensor factor to each edge and each vertex of a labelled tree is summarized in Figure \ref{fig:labelledtree}.

  \begin{rem}
    Describing the signs in this total complex can be quite cumbersome. Instead, we proceed in a more standard way by first defining a simplicial chain complex and then applying geometric realization.
  \end{rem}

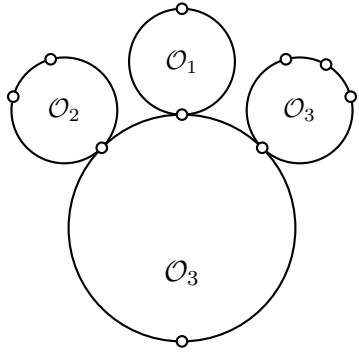
\begin{figure}
  \centering
  \begin{tikzpicture}[
    scale=1,
    puncture/.style={circle, draw, thick, fill=white, inner sep=1.4pt}
  ]
    \draw[thick] (0,0) circle (1.5);
    \node[font=\large] at (0,-0.6) {$\cO_{3}$};
    \node[puncture] at (0,-1.5) {};

    \begin{scope}[shift={(-1.556,1.556)}]
      \draw[thick] (0,0) circle (0.7);
      \node at (0,0) {$\cO_{2}$};
      \node[puncture] at (315:0.7) {};
      \node[puncture] at (105:0.7) {};
      \node[puncture] at (165:0.7) {};
    \end{scope}

    \begin{scope}[shift={(0,2.2)}]
      \draw[thick] (0,0) circle (0.7);
      \node at (0,0) {$\cO_{1}$};
      \node[puncture] at (270:0.7) {};
      \node[puncture] at (90:0.7) {};
    \end{scope}

    \begin{scope}[shift={(1.556,1.556)}]
      \draw[thick] (0,0) circle (0.7);
      \node at (0,0) {$\cO_{3}$};
      \node[puncture] at (225:0.7) {};
      \node[puncture] at (15:0.7) {};
      \node[puncture] at (60:0.7) {};
      \node[puncture] at (105:0.7) {};
    \end{scope}
  \end{tikzpicture}
  \caption{The nodal disc underlying the operadic structure map \eqref{eq:operadic_composition} for $k=3$.}
  \label{fig:nodal-disc}
\end{figure}

  In this context, it is straightforward to define the structure of an algebra over the operad $\cO$, i.e. construct maps
  \begin{equation} \Fuk_K(L_{k-1},L_{k}) \otimes \cdots \otimes  \Fuk_K(L_0,L_{1}) \otimes \cO_{k} \to   \Fuk_K(L_0,L_{k}) \end{equation}
  which are compatible with operadic composition. These maps are defined as the completion of a map defined on the summands of each factor of the left hand side which, as discussed earlier, are  indexed by labelled trees. There are essentially two ingredients to the construction. The first is to use the operadic structure maps
  \begin{equation} \label{eq:operadic_composition} \cO_{n_{v_k}}  \otimes \cdots \otimes  \cO_{n_{v_1}} \otimes \cO_{k} \to   \cO_{\sum n_{v_i}} =  \cO_{n_{v}}  \end{equation}
 on the bottom nodes of each tree. The geometry underlying these structure maps is that of a nodal disc, as drawn in Figure \ref{fig:nodal-disc}.

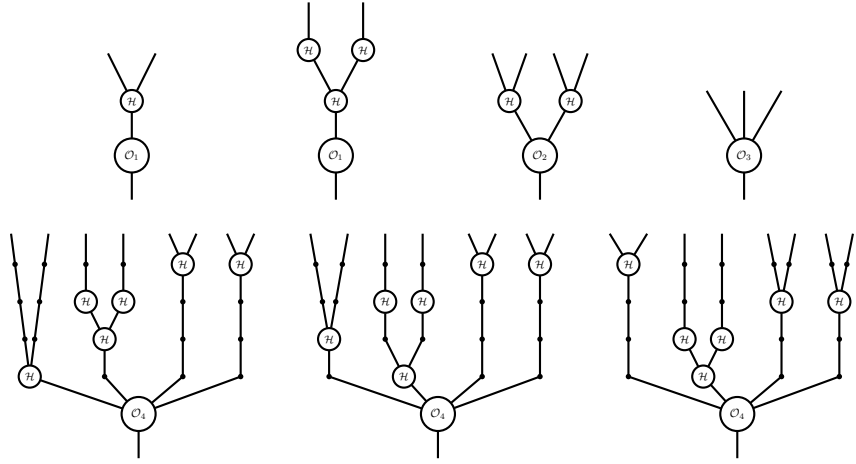
\begin{figure}
  \centering
  \begin{tikzpicture}[
    scale=0.45,
    every node/.style={transform shape},
    oblob/.style={circle, draw, thick, minimum size=1.0cm, inner sep=1pt, font=\normalsize},
    hblob/.style={circle, draw, thick, minimum size=0.65cm, inner sep=0pt, font=\small},
    idnode/.style={circle, fill, inner sep=1.6pt}
  ]
    \begin{scope}[shift={(-9,7.6)}]
      \node[oblob] (A) at (0,0) {$\cO_{1}$};
      \draw[thick] (A) -- (0,-1.3);
      \node[hblob] (AH) at (0,1.6) {$\cH$};
      \draw[thick] (A) -- (AH);
      \draw[thick] (AH) -- ++(-0.7,1.4);
      \draw[thick] (AH) -- ++(0.7,1.4);
    \end{scope}
    \begin{scope}[shift={(-3,7.6)}]
      \node[oblob] (B) at (0,0) {$\cO_{1}$};
      \draw[thick] (B) -- (0,-1.3);
      \node[hblob] (BH) at (0,1.6) {$\cH$};
      \draw[thick] (B) -- (BH);
      \node[hblob] (BH1) at (-0.8,3.1) {$\cH$};
      \node[hblob] (BH2) at (0.8,3.1) {$\cH$};
      \draw[thick] (BH) -- (BH1);
      \draw[thick] (BH) -- (BH2);
      \draw[thick] (BH1) -- ++(0,1.4);
      \draw[thick] (BH2) -- ++(0,1.4);
    \end{scope}
    \begin{scope}[shift={(3,7.6)}]
      \node[oblob] (C) at (0,0) {$\cO_{2}$};
      \draw[thick] (C) -- (0,-1.3);
      \node[hblob] (CH1) at (-0.9,1.6) {$\cH$};
      \node[hblob] (CH2) at (0.9,1.6) {$\cH$};
      \draw[thick] (C) -- (CH1);
      \draw[thick] (C) -- (CH2);
      \draw[thick] (CH1) -- ++(-0.5,1.4);
      \draw[thick] (CH1) -- ++(0.5,1.4);
      \draw[thick] (CH2) -- ++(-0.5,1.4);
      \draw[thick] (CH2) -- ++(0.5,1.4);
    \end{scope}
    \begin{scope}[shift={(9,7.6)}]
      \node[oblob] (D) at (0,0) {$\cO_{3}$};
      \draw[thick] (D) -- (0,-1.3);
      \draw[thick] (D) -- ++(-1.1,1.9);
      \draw[thick] (D) -- ++(0,1.9);
      \draw[thick] (D) -- ++(1.1,1.9);
    \end{scope}

    \begin{scope}[shift={(-8.8,0)}]
      \node[oblob] (O1) at (0,0) {$\cO_{4}$};
      \draw[thick] (O1) -- (0,-1.3);
      \node[hblob] (a1) at (-3.2,1.1) {$\cH$};
      \draw[thick] (O1) -- (a1);
      \draw[thick] (a1) -- (-3.75,5.3); \draw[thick] (a1) -- (-2.65,5.3);
      \node[idnode] at (-3.344,2.2) {}; \node[idnode] at (-3.056,2.2) {};
      \node[idnode] at (-3.488,3.3) {}; \node[idnode] at (-2.912,3.3) {};
      \node[idnode] at (-3.632,4.4) {}; \node[idnode] at (-2.768,4.4) {};
      \draw[thick] (O1) -- (-1.0,1.1);
      \node[idnode] at (-1.0,1.1) {};
      \node[hblob] (b1) at (-1.0,2.2) {$\cH$};
      \draw[thick] (-1.0,1.1) -- (b1);
      \node[hblob] (b1l) at (-1.55,3.3) {$\cH$}; \node[hblob] (b1r) at (-0.45,3.3) {$\cH$};
      \draw[thick] (b1) -- (b1l); \draw[thick] (b1) -- (b1r);
      \draw[thick] (b1l) -- (-1.55,5.3); \draw[thick] (b1r) -- (-0.45,5.3);
      \node[idnode] at (-1.55,4.4) {}; \node[idnode] at (-0.45,4.4) {};
      \draw[thick] (O1) -- (1.3,1.1);
      \node[idnode] at (1.3,1.1) {};
      \node[idnode] at (1.3,2.2) {};
      \node[idnode] at (1.3,3.3) {};
      \draw[thick] (1.3,1.1) -- (1.3,3.3);
      \node[hblob] (c1) at (1.3,4.4) {$\cH$};
      \draw[thick] (1.3,3.3) -- (c1);
      \draw[thick] (c1) -- (0.9,5.3); \draw[thick] (c1) -- (1.7,5.3);
      \draw[thick] (O1) -- (3.0,1.1);
      \node[idnode] at (3.0,1.1) {};
      \node[idnode] at (3.0,2.2) {};
      \node[idnode] at (3.0,3.3) {};
      \draw[thick] (3.0,1.1) -- (3.0,3.3);
      \node[hblob] (d1) at (3.0,4.4) {$\cH$};
      \draw[thick] (3.0,3.3) -- (d1);
      \draw[thick] (d1) -- (2.6,5.3); \draw[thick] (d1) -- (3.4,5.3);
    \end{scope}
    \begin{scope}[shift={(0.0,0)}]
      \node[oblob] (O2) at (0,0) {$\cO_{4}$};
      \draw[thick] (O2) -- (0,-1.3);
      \draw[thick] (O2) -- (-3.2,1.1);
      \node[idnode] at (-3.2,1.1) {};
      \node[hblob] (a2) at (-3.2,2.2) {$\cH$};
      \draw[thick] (-3.2,1.1) -- (a2);
      \draw[thick] (a2) -- (-3.75,5.3); \draw[thick] (a2) -- (-2.65,5.3);
      \node[idnode] at (-3.395,3.3) {}; \node[idnode] at (-3.005,3.3) {};
      \node[idnode] at (-3.59,4.4) {}; \node[idnode] at (-2.81,4.4) {};
      \node[hblob] (b2) at (-1.0,1.1) {$\cH$};
      \draw[thick] (O2) -- (b2);
      \draw[thick] (b2) -- (-1.55,2.2); \draw[thick] (b2) -- (-0.45,2.2);
      \node[idnode] at (-1.55,2.2) {}; \node[idnode] at (-0.45,2.2) {};
      \node[hblob] (b2l) at (-1.55,3.3) {$\cH$}; \node[hblob] (b2r) at (-0.45,3.3) {$\cH$};
      \draw[thick] (-1.55,2.2) -- (b2l); \draw[thick] (-0.45,2.2) -- (b2r);
      \draw[thick] (b2l) -- (-1.55,5.3); \draw[thick] (b2r) -- (-0.45,5.3);
      \node[idnode] at (-1.55,4.4) {}; \node[idnode] at (-0.45,4.4) {};
      \draw[thick] (O2) -- (1.3,1.1);
      \node[idnode] at (1.3,1.1) {};
      \node[idnode] at (1.3,2.2) {};
      \node[idnode] at (1.3,3.3) {};
      \draw[thick] (1.3,1.1) -- (1.3,3.3);
      \node[hblob] (c2) at (1.3,4.4) {$\cH$};
      \draw[thick] (1.3,3.3) -- (c2);
      \draw[thick] (c2) -- (0.9,5.3); \draw[thick] (c2) -- (1.7,5.3);
      \draw[thick] (O2) -- (3.0,1.1);
      \node[idnode] at (3.0,1.1) {};
      \node[idnode] at (3.0,2.2) {};
      \node[idnode] at (3.0,3.3) {};
      \draw[thick] (3.0,1.1) -- (3.0,3.3);
      \node[hblob] (d2) at (3.0,4.4) {$\cH$};
      \draw[thick] (3.0,3.3) -- (d2);
      \draw[thick] (d2) -- (2.6,5.3); \draw[thick] (d2) -- (3.4,5.3);
    \end{scope}
    \begin{scope}[shift={(8.8,0)}]
      \node[oblob] (O3) at (0,0) {$\cO_{4}$};
      \draw[thick] (O3) -- (0,-1.3);
      \draw[thick] (O3) -- (-3.2,1.1);
      \node[idnode] at (-3.2,1.1) {};
      \node[idnode] at (-3.2,2.2) {};
      \node[idnode] at (-3.2,3.3) {};
      \draw[thick] (-3.2,1.1) -- (-3.2,3.3);
      \node[hblob] (a3) at (-3.2,4.4) {$\cH$};
      \draw[thick] (-3.2,3.3) -- (a3);
      \draw[thick] (a3) -- (-3.75,5.3); \draw[thick] (a3) -- (-2.65,5.3);
      \node[hblob] (b3) at (-1.0,1.1) {$\cH$};
      \draw[thick] (O3) -- (b3);
      \node[hblob] (b3l) at (-1.55,2.2) {$\cH$}; \node[hblob] (b3r) at (-0.45,2.2) {$\cH$};
      \draw[thick] (b3) -- (b3l); \draw[thick] (b3) -- (b3r);
      \draw[thick] (b3l) -- (-1.55,5.3); \draw[thick] (b3r) -- (-0.45,5.3);
      \node[idnode] at (-1.55,3.3) {}; \node[idnode] at (-0.45,3.3) {};
      \node[idnode] at (-1.55,4.4) {}; \node[idnode] at (-0.45,4.4) {};
      \draw[thick] (O3) -- (1.3,1.1);
      \node[idnode] at (1.3,1.1) {};
      \node[idnode] at (1.3,2.2) {};
      \draw[thick] (1.3,1.1) -- (1.3,2.2);
      \node[hblob] (c3) at (1.3,3.3) {$\cH$};
      \draw[thick] (1.3,2.2) -- (c3);
      \draw[thick] (c3) -- (0.9,5.3); \draw[thick] (c3) -- (1.7,5.3);
      \node[idnode] at (1.08,4.4) {}; \node[idnode] at (1.52,4.4) {};
      \draw[thick] (O3) -- (3.0,1.1);
      \node[idnode] at (3.0,1.1) {};
      \node[idnode] at (3.0,2.2) {};
      \draw[thick] (3.0,1.1) -- (3.0,2.2);
      \node[hblob] (d3) at (3.0,3.3) {$\cH$};
      \draw[thick] (3.0,2.2) -- (d3);
      \draw[thick] (d3) -- (2.6,5.3); \draw[thick] (d3) -- (3.4,5.3);
      \node[idnode] at (2.78,4.4) {}; \node[idnode] at (3.22,4.4) {};
    \end{scope}
  \end{tikzpicture}
  \caption{The top row shows three Lagrangian and Hamiltonian labelled planar trees, and a corolla with three inputs. The bottom row shows three of the twelve terms of the resulting composition; the inserted identity nodes are drawn as small filled dots.}
  \label{fig:tree-complex}
\end{figure}

The other ingredient is the Eilenberg-Zilber map: we insert ``identities'' in each factor to separate the levels of the trees being composed. The detailed formulas are provided in Section~\ref{sec-review-left-Kan}, but a representative visualization is provided in Figure \ref{fig:tree-complex}: acting by $\cO_3$ on the displayed summands has twelve terms (indexed by the shuffles of the levels of the three trees), of which we display three.

    There is only one remaining step to construct an $A_\infty$ category in the usual sense. The most explicit way to do this is the following: choose an isomorphism of operads between the moduli space of discs and the space of ribbon metric trees. The latter has a standard cubulation, with coordinates given by the negative of the exponential of the lengths of edges, and this cubulation is compatible with the operadic structure maps. The sum of the cubical chains represented by the top dimensiona curves in each moduli space then represent a (relative) fundamental chain for the moduli space, which is compatible with operadic composition. The sub-operad generated by these operations is spanned, in parity $k$, by all possible compositions of the higher operations $\mu_{d}$ that can be applied to a string of $d$ inputs, with differential given by the usual quadratic formula:
    \begin{equation}
      \delta [\mu_d] = \sum_{d_- + d_+ = d +1 // 1 \leq k \leq d_- } \pm  [\mu_{d_-}] \circ_{k} [\mu_{d_+}]
    \end{equation}
    where the signs are specified by any of many possible conventions.

    The outcome is a map of operads
    \begin{equation} \label{eq:boardman_vogt}
      A_\infty \to \scrO
    \end{equation}
which makes every algebra over $\scrO$ an $A_\infty$ algebra in the usual sense.

\subsection{Statement of the main results}

Let $M$ be a closed symplectic manifold such that every spherical class of positive symplectic area has non-negative Chern number (see Definition~\ref{def:weakly-monotone_M}), possibly equipped with a bulk class as described in Section~\ref{sec:bulk-branes}, and let $\rho$ be a potential value (corresponding to fixing the value of the weighted count of Maslov $2$ holomorphic discs, as defined in Section~\ref{sec-superpotential}). The following result is proved at the end of Section~\ref{sec:def-fuk-with-support}.

\begin{thm}\label{thm-restriction-functors}
  For every pair of compact subsets $K\subset K'$ of $M$, there is a strict $A_\infty$ functor
  \begin{equation}
    \mathrm{Fuk}_{K',\rho}\to \mathrm{Fuk}_{K,\rho},
  \end{equation}
  which is the identity map on objects, and is strictly compatible with triple inclusions $K \subset K' \subset K''$. Moreover, the symplectormophism group of $M$ acts on these Fukaya categories: every symplectomorphism $\phi$ induces a functor
  \begin{equation}
    \mathrm{Fuk}_{K,\rho} \to  \mathrm{Fuk}_{\phi(K),\rho},
  \end{equation}
which is strictly compatible with compositions (the bulk and grading data on the source and target of this functor and given by pushforward under $\phi$).
\end{thm}

As discussed earlier, potential applications of our approach will rely on a comparison of our groups with those constructed earlier in the literature. The following is proved in Section~\ref{sec:acceleration-data}.

\begin{thm}\label{thm-intro-telescope}
    The morphism spaces between any two Lagrangian branes receive a canonical quasi-isomorphism from the completed telescope of an acceleration datum.
\end{thm}

Finally, we prove the following descent property in Section~\ref{sec:mayer-vietoris}.

\begin{thm}\label{thm-descent}
  If $(K_1, \ldots,  K_N)$ is a weakly involutive cover of $K$ then the canonical functor
  \begin{equation} \Fuk_{K}\to \Cech(\Fuk; K_1,\ldots,K_N)\end{equation} is a quasi-isomorphism of $A_\infty$-categories.
\end{thm}

\begin{rem}
  There are many additional properties of these categories which can be proved. In the next section, which focuses on examples and computations, we provide one such general property, namely the vanishing of every object of $\Fuk_{K}$, if its support is disjoint from $K$. A deeper follow-up property is the following: let $\phi$ be a Hamiltonian diffeomorphism, and $K \subset K'$ a nested pair so that $\phi K \subset K'$. We obtain a diagram
  \begin{equation}
    \begin{tikzcd}[column sep=large]
      & \Fuk_{K'} \arrow[dl] \arrow[dr] & \\
      \Fuk_{K} & & \Fuk_{K}, \arrow[ll, "\phi_*"']
    \end{tikzcd}
  \end{equation}
  in which the two downward arrows are the restriction functors of Theorem~\ref{thm-restriction-functors} associated to the inclusions $\phi K \subset K'$ and $K \subset K'$, and the bottom arrow is the functor $\phi_*$ induced by $\phi$, and which does not commute. The point, however, is that it commutes up to homotopy, after inverting the Novikov variable. If the Hamiltonian isotopy can be chosen so that the image of $K$ never leaves $K'$, then we do not have to invert the Novikov variable.
\end{rem}

We complete this introduction with an outline of the paper. The examples and computations announced above are collected in Section~\ref{sec-basic-computations}, which immediately follows this introduction and has two parts: we first compute, by means of the completed telescope, the Floer cohomology of a Lagrangian section of the projection of the square torus supported on an annulus, and we then show the vanishing of the Floer cohomology with support if one of the branes is disjoint from the support, and that, if instead one of the branes is contained in the support, the Floer cohomology with support is simply the global Floer cohomology.

Section~\ref{sec:prel-sympl-manif} collects the preliminaries from Lagrangian Floer theory that we shall use. The two sections which follow provide the analogue, in the Lagrangian setting, of the constructions which were implemented for Hamiltonian Floer theory in \cite{Abouzaid2023}:  an open string Hamiltonian indexing multicategory, whose objects record the Floer data of a pair of branes and whose multimorphisms are cubes of Hamiltonian data, and the Floer functor, which assigns to each object the associated Floer cochains, and to each cube of Floer data its action on the Floer chains. Section~\ref{sec:def-fuk-with-support} reviews the operadic left Kan extension and gives the complete definition of the Fukaya category with support $\Fuk_{K,\rho}^f$. This is also where the restriction functors and the symplectomorphism action of Theorem~\ref{thm-restriction-functors} are defined. Section~\ref{sec:acceleration-data} proves Theorem~\ref{thm-intro-telescope}, comparing the morphism spaces of these categories, and the maps induced on them by the restriction functors, with the completed telescope of an acceleration datum; the comparison passes through the identification of a homotopy left Kan extension with a weighted homotopy colimit, discussed in Section~\ref{sec-weighted-leftkan}. Section~\ref{sec:mayer-vietoris} is dedicated to the, rather long, proof of Theorem~\ref{thm-descent}.  The paper ends with two appendices: Appendix~\ref{sec:trees-discs} fixes conventions about ribbon trees and pre-stable punctured discs with strip-like ends, while Appendix~\ref{sec:monotonicity-inequality} discusses an integral monotonicity condition, which we use to bound the geometric energy of a Floer solution by its topological energy.

\subsection*{Acknowledgements}
M.A. was supported by NSF grants DMS-2103805 and DMS-2506145. Y.G. was supported by ISF grant No. 3605/24. U.V. was supported by TÜBİTAK grants no. 124F451 and 225N333. The authors used large language models for
proofreading, improving the exposition, and critically reviewing the arguments.

\section{Basic computations}
\label{sec-basic-computations}

In this section, we list some example computations of morphism spaces in Fukaya categories with support.

\subsection{The computation for a Lagrangian section}

Let $\bT=(\rmodz)^2$ have coordinates $(p,q)$ and symplectic form
$\omega=dp\wedge dq$, and let
\[
\pi\colon \bT\longrightarrow \rmodz,\qquad \pi(p,q)=p.
\]
Fix a proper closed interval $P=[a,b]\subset \rmodz$, and put
$K=\pi^{-1}(P)$. We take
\[
L=\rmodz\times\{0\}
\]
with its standard grading, Pin structure, and trivial rank-one local
system. Thus $L$ is a Lagrangian section of $\pi$. We assume that the
bulk is empty. Our aim is only to compute the underlying supported
Floer cohomology with support
\[
HF_K^*(L,L;\novf)
:=
H^*\!\left(
\operatorname{Fuk}_K^*(L,L)\otimes_{\novr}\novf
\right).
\]
By
Proposition~\ref{prop-tel-quasi}, it is enough to compute the $T$-adically completed
telescope associated to an acceleration datum.

Choose an acceleration datum $H_i=h_i\circ\pi$, where
$h_i\colon\rmodz\to\mathbb R$ is negative on $P$, converges to $0$ on
$P$, and converges to $+\infty$ on
$\rmodz\setminus P$. See Figure \ref{fig:shape-h}. We choose the functions $h_i$ so that, for every
fixed $N\in\mathbb Z$ and all sufficiently large $i$, the equation
\[
-h_i'(p)=N
\]
has precisely two non-degenerate solutions. Since
\[
X_{H_i}=-h_i'(p)\partial_q,
\]
the corresponding Hamiltonian chords from $L$ to itself are
\[
x(t)=(p,Nt).
\]
The connected components of the path space from $L$ to itself are
indexed by $N\in\mathbb Z$, and Floer differentials and continuation
maps preserve this index. We denote the two chords in class $N$ by
$\alpha_{i,N}$ and $\beta_{i,N}$, choosing the notation so that
\[
|\alpha_{i,N}|=0,
\qquad
|\beta_{i,N}|=1.
\]
Equivalently, they are respectively the index-$0$ and index-$1$ zeros
of the closed one-form
\[
\eta_{i,N}=dh_i+N\,dp
\]
on $\rmodz$. Thus, for fixed $N$ and sufficiently large $i$, the
entire Floer complex has the form
\[
CF^*(L,L;H_i)_N
=
\novr\,\alpha_{i,N}\oplus\novr\,\beta_{i,N}.
\]

\begin{figure}
    \centering
    \includegraphics[width=\linewidth]{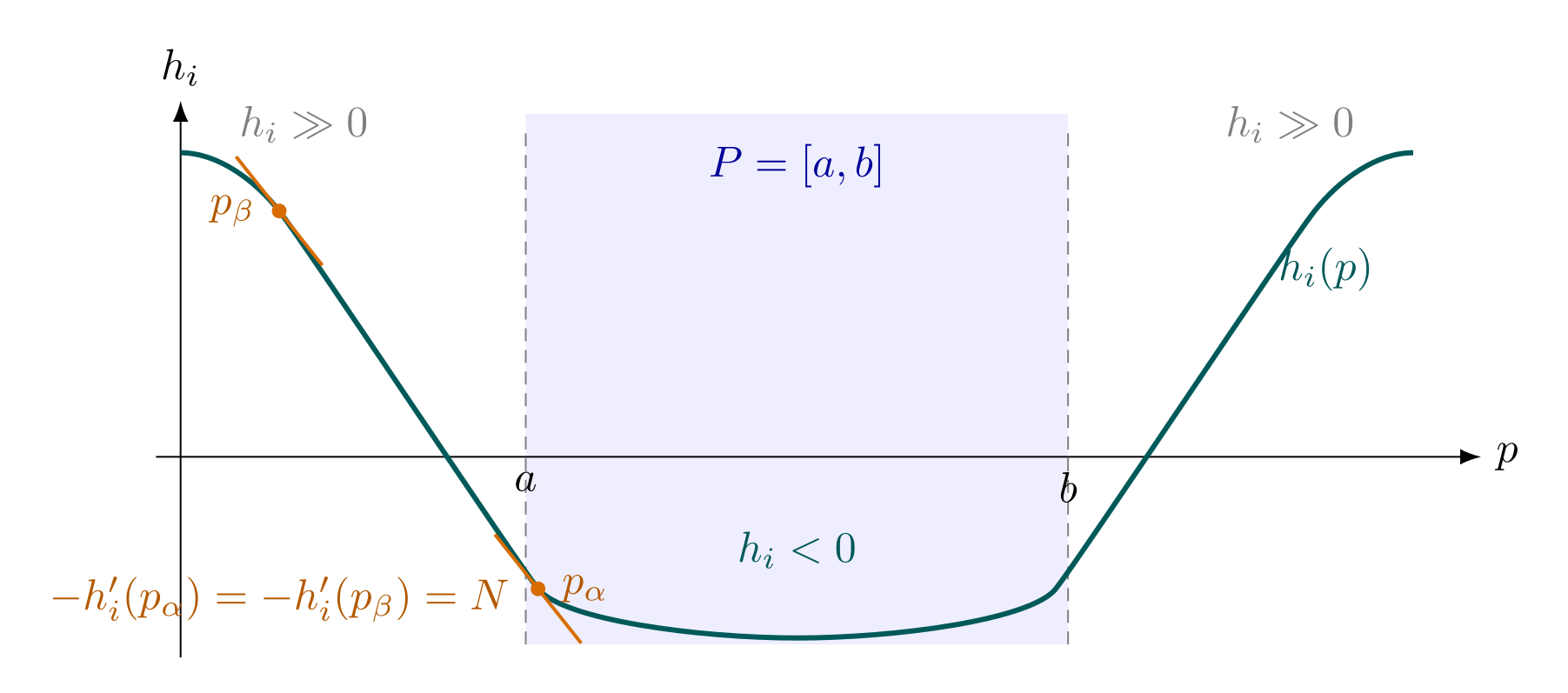}
    \caption{The shape of functions used in the computation}
    \label{fig:shape-h}
\end{figure}

\begin{prop}
    
\begin{equation}\label{eq:section-final}
HF_K^*(L,L;\novf)
\cong
\widehat{\bigoplus}_{N\in\mathbb Z}
\novf\,\alpha_N,
\qquad
|\alpha_N|=0.
\end{equation}
In particular,
\[
HF_K^j(L,L;\novf)=0
\qquad\text{for }j\neq0.
\]
\end{prop}

Here the completed direct sum in
Equation~\eqref{eq:section-final} consists of formal sums
$\sum_N a_N\alpha_N$ satisfying
\[
\#\bigl\{
N\in\mathbb Z:\operatorname{val}(a_N)<E
\bigr\}<\infty
\qquad
\text{for every }E\in\mathbb R;
\]
equivalently,
\[
\operatorname{val}(a_N)\longrightarrow+\infty
\qquad\text{as }|N|\longrightarrow\infty.
\]
\begin{proof}
The inclusion $\pi_1(L)\to\pi_1(\bT)$ is injective, and hence
\[
\pi_2(\bT,L)=0.
\]
For the Floer differential of $H_i$, the
equation on the strip reduces to a one-dimensional Morse--Novikov equation, which
allows us to determine it explicitly.

To see this reduction, lift a strip in class $N$ so that its boundary
lies on $q=0$ and $q=N$, and write
\[
q(s,t)=Nt+r(s,t).
\]
The Floer equation becomes
\[
p_s=r_t+N+h_i'(p),
\qquad
p_t=-r_s.
\]
It follows that
\[
r_{ss}+r_{tt}-h_i''(p)r_s=0.
\]
Since $r=0$ on the boundary and at the two strip-like ends, the
Hopf maximum (and minimum) principle gives $r=0$. Thus every Floer strip is a spun
gradient trajectory
\[
u(s,t)=(\gamma(s),Nt),
\qquad
\dot\gamma(s)=h_i'(\gamma(s))+N.
\]
In particular, the Floer complex in class $N$ is
exactly the Morse--Novikov complex of $\eta_{i,N}$.

The two rigid strips from $\alpha_{i,N}$ to $\beta_{i,N}$ correspond
to the two gradient trajectories on the circle, which are shown pictorially in Figure \ref{fig:tilt}. Their orientation
signs are opposite, and their topological actions differ by
\[
\left|\int_{\rmodz}\eta_{i,N}\right|=|N|.
\]
It follows, after changing a generator by a unit of $\novr$ if
necessary, that
\begin{equation}\label{eq:section-differential}
d\alpha_{i,N}
=
T^{A_{i,N}}\bigl(1-T^{|N|}\bigr)\beta_{i,N},
\qquad
d\beta_{i,N}=0,
\end{equation}
where $A_{i,N}$ is the smaller of the two topological actions. For
$N=0$, the two contributions have the same action and cancel, so
Equation~\eqref{eq:section-differential} simply says
\[
d\alpha_{i,0}=0.
\]
For $N\neq0$, the element $1-T^{|N|}$ is a unit in $\novr$, and hence
\[
CF^*(L,L;H_i)_N\otimes_{\novr}\novf
\]
is acyclic. 

\begin{figure}
    \centering
    \includegraphics[width=\linewidth]{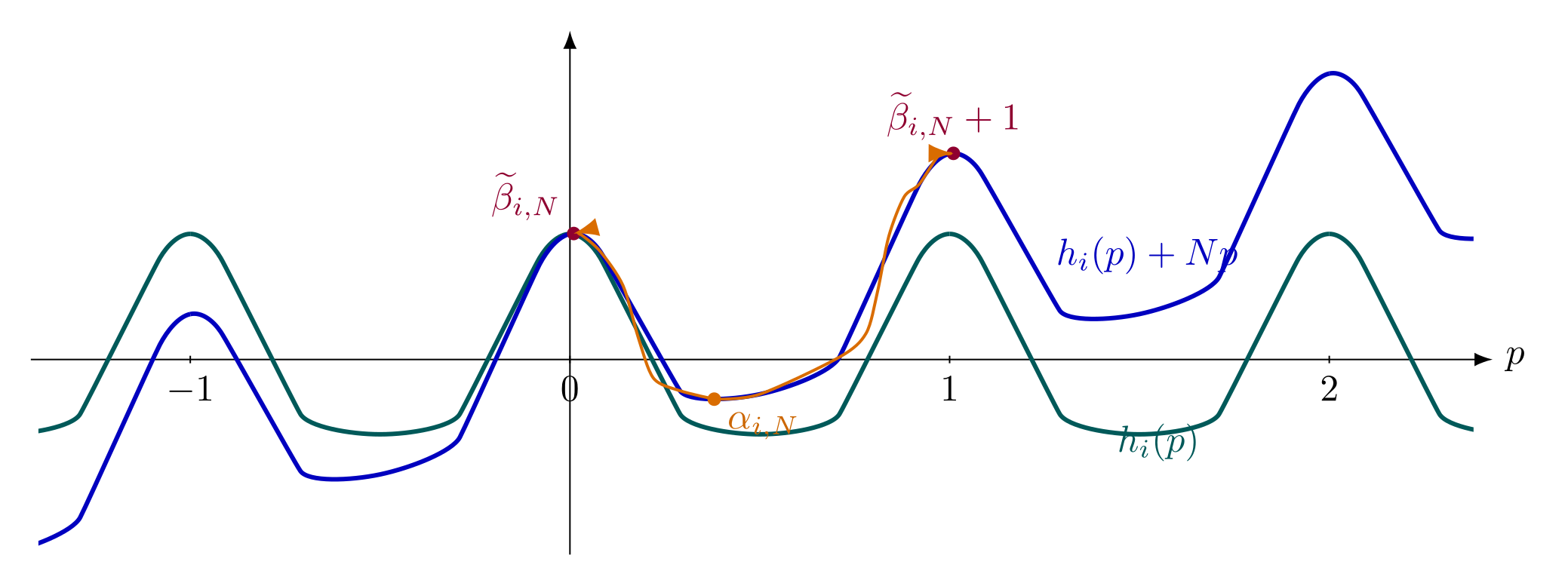}
    \caption{The gradient trajectories of the tilted function}
    \label{fig:tilt}
\end{figure}

The controlled choice of the acceleration datum gives
\begin{equation}\label{eq:section-action-gap}
A_{i,N}\longrightarrow+\infty
\qquad (i\longrightarrow\infty)
\end{equation}
for every fixed $N$.

Since continuation maps $c_{i,N}$ have degree $0$ and there is only one
generator in each degree:
\begin{equation}\label{eq:section-continuation}
\begin{aligned}
c_{i,N}(\alpha_{i,N})
&=
u_{i,N}T^{\delta_{i,N}}\alpha_{i+1,N},\\
c_{i,N}(\beta_{i,N})
&=
v_{i,N}T^{\Delta_{i,N}}\beta_{i+1,N},
\end{aligned}
\end{equation}
where $u_{i,N},v_{i,N}\in\novr^\times$ have valuation $0$. The short
continuation trajectory following the $\alpha$-chord shows that
\begin{equation}\label{eq:section-alpha-energy}
\sum_{i\geq i_0}\delta_{i,N}<\infty.
\end{equation}
On the other hand,
\begin{equation}\label{eq:section-beta-energy}
\sum_{i\geq i_0}\Delta_{i,N}=+\infty.
\end{equation}
For $N\neq0$, this last assertion follows from the
chain-map equation as follows. Taking valuations in
\[
c_{i,N}d=dc_{i,N}
\]
and using Equation~\eqref{eq:section-differential} gives
\[
\Delta_{i,N}-\delta_{i,N}
=
A_{i+1,N}-A_{i,N}.
\]
For $N=0$, Equation~\eqref{eq:section-beta-energy} follows directly
from the divergence of the values of the functions on the $\beta$-critical points.

Fix $E>0$ and write
\[
R_E=\novr/T^E\novr.
\]
Set
\[
C_L^*(N,E)
:=
\underset{i}{\operatorname{colim}}\,
\bigl(
CF^*(L,L;H_i)_N\otimes_{\novr}R_E
\bigr).
\]
Here we use the standard quasi-isomorphism from the telescope of the
truncated ray to its ordinary filtered colimit. By
Equation~\eqref{eq:section-action-gap}, after discarding finitely many
terms the differential in this ray vanishes modulo $T^E$. It
therefore splits as the direct sum of the $\alpha$-ray and the
$\beta$-ray.

After harmless unit changes of basis, the connecting maps in these
two rays are multiplication by $T^{\delta_{i,N}}$ and
$T^{\Delta_{i,N}}$, respectively.
Equation~\eqref{eq:section-beta-energy} implies that every element of
the $\beta$-ray is eventually mapped into $\novf_{\geq E}$, so its
direct limit is zero. Equation~\eqref{eq:section-alpha-energy}, with
the usual normalization in which the $\alpha$-actions approach their
limiting value strictly from one side, gives
\begin{equation}\label{eq:section-truncated}
C_L^*(N,E)
\cong
\frac{\novf_{>0}}{\novf_{> E}}\,\alpha_N,
\qquad
|\alpha_N|=0.
\end{equation}
The maps associated to decreasing $E$ are the evident truncation
maps.

At each finite stage only finitely many classes $N$ occur.
Consequently,
\[
\underset{i}{\operatorname{colim}}\,
\bigl(
CF^*(L,L;H_i)\otimes_{\novr}R_E
\bigr)
\cong
\bigoplus_{N\in\mathbb Z}C_L^*(N,E).
\]
Taking the inverse limit over $E$, which is the completed-telescope
construction of the supported morphism complex, and then extending
scalars to $\novf$, gives \eqref{eq:section-final}.
\end{proof}

$HF_K^0(L,L;\novf)$ is also equipped with an associative $\novf$-algebra structure. Let us now describe it  without proofs. 

A \emph{rigid analytic annulus} of modulus $m\geq 0$ is defined as the analytic domain defined by $0\leq val(t)\leq m$ (or equivalently, $a\leq val(t)\leq b$ for any $b-a=m)$ inside the rigid analytic affine line with coordinate $t.$ By definition, the algebra of functions on this rigid analytic annulus is the algebra $$\Lambda\langle x,y\rangle/(xy-T^m),$$ where the triangular brackets indicate that we are taking the Tate algebra with those generators: formal power series in $\novf$ with the valuation of the coefficients going to infinity. An algebra that recieves a surjective map from a Tate algebra is called an \emph{affinoid algebra.} 

Using this terminology, it can be shown that  $HF_K^0(L,L;\novf)$ is isomorphic to the algebra of functions on the non-archimedean annulus of modulus $b-a.$ This is a manifestation of the mirror duality between $T^2$ on the $A-$side and a non-archimedean elliptic curve on the $B-$side but this is outside scope of this paper. 

\begin{rem}
    Perhaps a more familiar (in mirror symmetry) representation of the algebra of functions is as the Kontsevich-Soibelman algebra of non-archimedean analytic functions convergent on $[a,b]$: $$\{\sum_{n\in\bZ} \beta_nX^{n}\mid \beta_n\in\novf\text{ and } val(\beta_n)+nr\to \infty \text{ for all } r\in [a,b]\}.$$ The convergence condition should be thought of as the condition for the formal power series to converge (i.e. define an element of $\novf$) when we substitute $T^r$ (or any other element of $\novf$) for  $X$. 
\end{rem}

\subsection{Vanishing results}
 We fix two
objects $\mathbb L_0,\mathbb L_1$ in the same $\rho$-summand, and denote
their underlying Lagrangian submanifolds by $L_0,L_1$.  In this section
we assume that the zero Hamiltonian is admissible for
$(\mathbb L_0,\mathbb L_1)$.  In particular, $L_0$ and $L_1$ intersect
transversely. We ignore the bulk, local-system, and orientation factors associated
to a rigid continuation solution since they do not affect the Novikov valuation of the corresponding
continuation maps.

\begin{prop}
\label{prop-brane-disjoint-support}
Let $K\subset M$ be compact.  Suppose that the underlying Lagrangian
submanifold of either $\mathbb L_0$ or $\mathbb L_1$ is disjoint from
$K$.  Then
\[
    \Fuk_K^*(\mathbb L_0,\mathbb L_1)
\]
is acyclic.
\end{prop}

\begin{proof}
We treat the case
\[
    L_0\cap K=\varnothing;
\]
the case of $L_1$ is identical.

Choose open sets
\[
    L_0\subset V\Subset U\Subset M\setminus K.
\]
Choose a smooth function
\[
    q:M\longrightarrow[0,\infty)
\]
such that
\[
    q^{-1}(0)=K,
    \qquad
    q|_U=1.
\]
Let $\epsilon_i>0$ decrease to zero and set
\[
    H_i=iq-\epsilon_i.
\]
Then
\[
    H_i|_K=-\epsilon_i\longrightarrow0,
\]
while, for every $x\notin K$,
\[
    H_i(x)\longrightarrow+\infty.
\]
Moreover,
\[
    H_{i+1}-H_i
    =
    q+\epsilon_i-\epsilon_{i+1}>0.
\]
Thus $(H_i)$ is an acceleration datum for $K$.

On $U$ we have
\[
    H_i=i-\epsilon_i,
\]
so $X_{H_i}=0$ there.  Since $L_0\subset U$, uniqueness for the
Hamiltonian ODE implies that every Hamiltonian trajectory of $H_i$
issuing from $L_0$ is constant.  Hence the Hamiltonian chords of $H_i$
from $L_0$ to $L_1$ are precisely the points of
\[
    L_0\cap L_1.
\]
In particular they are nondegenerate and admissible by our standing
assumption.

Choose regular monotone continuation data from $H_i$ to $H_{i+1}$,
with the Hamiltonian on $U$ given by a monotone interpolation between
the constants $i-\epsilon_i$ and $i+1-\epsilon_{i+1}$.  The stationary
continuation strips cause no additional regularity issue: subtracting
this spatially constant function from the continuation Hamiltonian does
not change its Hamiltonian vector field, the continuation equation, or
its linearization, so their regularity follows from
Lemma~\ref{lem-const-reg}.  The remaining continuation data may be
chosen regular in the usual way.

Put
\[
    a_i=(i+1-\epsilon_{i+1})-(i-\epsilon_i).
\]
Then $a_i\geq 1$.  Let $u$ be a rigid continuation strip.

If
\[
    u(\mathbb R\times[0,1])\subset U,
\]
then
the energy identity gives
\[
    E_{\mathrm{top}}(u)
    =
    E_{\mathrm{geo}}(u)+a_i
    \geq 1.
\]

Suppose instead that $u$ leaves $U$.  Choose
\[
    L_0\subset V \Subset U.
\]
Both asymptotic chords lie in $L_0\cap L_1\subset V$, while the image
of $u$ meets $M\setminus U$.  On $U$ the Hamiltonian vector field
vanishes, so there the continuation equation is the ordinary
pseudo-holomorphic strip equation.  The Lagrangian-boundary version of Sikorav's Monotonicity Lemma \cite{Sikorav1994} therefore gives a constant
\[
    \hbar>0,
\]
independent of $i$, such that every such strip which crosses from $V$
to $M\setminus U$ satisfies
\[
    E_{\mathrm{geo}}(u)\geq\hbar.
\]
Since the continuation is monotone,
\[
    E_{\mathrm{top}}(u)\geq E_{\mathrm{geo}}(u)\geq\hbar.
\]

Thus every rigid contribution to the continuation map
\[
    \kappa_i:
    CF^*(\mathbb L_0,\mathbb L_1;H_i)
    \longrightarrow
    CF^*(\mathbb L_0,\mathbb L_1;H_{i+1})
\]
has topological energy at least
\[
    \delta:=\min\{1,\hbar\}>0.
\]
Consequently
\[
    \kappa_i(CF^*(H_i))
    \subset
    T^\delta CF^*(H_{i+1})
\]
for every $i$.

Let
\[
    C=\widehat{\bigoplus_i CF^*(H_i)}
\]
and let $\kappa:C\to C$ denote the shift induced by the continuation
maps.  Then
\[
    \kappa^n(C)\subset T^{n\delta}C,
\]
so $\kappa$ is topologically nilpotent and
\[
    (1-\kappa)^{-1}
    =
    \sum_{n\geq0}\kappa^n
\]
converges $T$-adically.  Hence the completed telescope, which is the
cone of $1-\kappa$ up to the usual sign convention, is acyclic.
Proposition~\ref{prop-tel-quasi} identifies this telescope with
$\Fuk_K^*(\mathbb L_0,\mathbb L_1)$.

\end{proof}

\begin{rem}
\label{rem-disjoint-nontransverse}
The transversality hypothesis above is imposed only to keep the
proof elementary: it allows the acceleration Hamiltonians to be
spatially constant near the Lagrangian disjoint from $K$, allowing us to use Monotonicity for $J$-holomorphic curves.  The vanishing statement also holds without this hypothesis, by introducing
a small perturbing Hamiltonian near that Lagrangian. In that case we use a Lagrangian version of Hein's Lemma \cite{Hein12} for the corresponding local Floer energy estimate. 
\end{rem}

\begin{prop}
\label{prop-branes-contained-support}
Suppose
\[
    K'=K\cup K_2
\]
is a weakly involutive cover and that the underlying Lagrangian
submanifold of either $\mathbb L_0$ or $\mathbb L_1$ is disjoint from
$K_2$.  Then restriction induces a quasi-isomorphism
\[
    \Fuk_{K'}^*(\mathbb L_0,\mathbb L_1)
    \longrightarrow
    \Fuk_K^*(\mathbb L_0,\mathbb L_1).
\]

\end{prop}

\begin{proof}
By the Mayer--Vietoris theorem, the morphism complex over $K'$ is
computed by the binary Cech complex associated to the cover
$K'=K\cup K_2$.  Proposition~\ref{prop-brane-disjoint-support} gives
\[
    \Fuk_{K_2}^*(\mathbb L_0,\mathbb L_1)\simeq0
\]
and, since $K\cap K_2\subset K_2$,
\[
    \Fuk_{K\cap K_2}^*(\mathbb L_0,\mathbb L_1)\simeq0.
\]
The Cech complex therefore reduces, up to quasi-isomorphism, to
\[
    \Fuk_K^*(\mathbb L_0,\mathbb L_1),
\]
and the resulting map is the restriction map.
\end{proof}

\begin{cor}
    For an inclusion $K\Subset K'$ so that one of the Lagrangians is contained in $K$, the restriction map from $K'$ to $K$ induces a quasi isomrphism
    \[
    \Fuk_{K'}^*(\mathbb L_0,\mathbb L_1)
    \longrightarrow
    \Fuk_K^*(\mathbb L_0,\mathbb L_1).
\]
\end{cor}

\begin{rem}
In fact, it can be proven that the restriction is a quasi-isomorphism for arbitary inclusions $K\subset K'$
 whenever both underlying Lagrangians are contained in $K$  directly using appropriate
acceleration data. We do not pursue this here. 
\end{rem}

\section{Preliminaries from Lagrangian Floer Theory}
\label{sec:prel-sympl-manif}
Fix a ground ring $\Bbbk$.  We denote by $\Lambda$ the Novikov ``field'' consisting of formal sums
\begin{equation}
  \sum_{\lambda \in \bR}  a_{\lambda} T^{\lambda},
\end{equation}
with $a_{\lambda} \in \Bbbk$, which are required to satisfy the property that 
the set $$\{\lambda\in\bR:a_\lambda\neq 0\}$$  is discrete, closed and bounded below in the subspace topology inherited from the real line. The ring $\Lambda$ is a field if and only if the ring $\Bbbk$ is a field, but the term Novikov ring is reserved for the subring $\Lambda_0$ consisting of elements for which $a_\lambda$ vanishes when $\lambda < 0$. We shall also use $U_{\Lambda}$ for the unitary elements, which is the multiplicative subgroup of elements of $\Lambda_0$ for which $a_0$ is invertible. Note that, when $\Bbbk$ is a field, $U_\Lambda$ is the set of elements of norm $1$ with respect to the norm defined by the valuation 
$$\val\left( \sum_{\lambda \in \bR}  a_{\lambda} T^{\lambda}\right):=\min\{\lambda:a_\lambda\neq 0\}.$$

\subsection{Bulk and grading structures}
\label{sec:bulk-branes}
We shall introduce a notion of bulk deformation which is more general than the existing notions in the literature. Our point of view relies on the notion of a co-orientation with respect to an abelian group $G$, whose group structure we refer to as multiplication. 

 Given a vector bundle $\pi: E\to B$, the integral vertically compactly supported cohomology groups of the total space in degree $\mathrm{rk}(E)$  define a locally constant sheaf (of abelian groups) over $B$ which assigns to an open susbet $U$ the group $H_{vc}^{\mathrm{rk}{(E)}}(\pi^{-1}(U))$. The fiber of the corresponding local system over $b\in B$ is canonically identified with $H^{\mathrm{rk}(E)}_c(E_b)$. Denoting by $\underline{G}$ the constant sheaf, we can also consider the hom sheaf to $\underline{G}$, which  gives a local system whose fibers are canonically identified with $\Hom\left(H^{\mathrm{rk}(E)}_c(E_b),G\right)$.
\begin{defin}
  Let $Y$ be a smooth manifold of dimension $n$ and $f \co Z \to Y$ a smooth map from a manifold of dimension $n-d$. We define a \emph{$G$-valued co-orientation} of $Z$ in $Y$ to be a global section of the  local system of abelian groups over $Z$ whose fibre at $z \in Z$ is
  \begin{equation}
  H^{n}_c(T_{f(z)}Y) \otimes \Hom\left( H^{n-d}_c(T_zZ),G \right). 
  \end{equation}
  \end{defin}

We shall use the fact that $G$-cooriented manifolds of codimension $d$ have well-defined intersection numbers, valued in $G$, with oriented $d$-dimensional manifolds mapped into the ambient manifold. From our point of view, it is clearest to describe the situation locally: if we are given an oriented manifold $D$ of dimension $d$ in the usual sense, a smooth map $u:D\to Y$, and a pair $(p,z) \in D \times Z$ such that $f$ and $u$ are transversely intersecting at these points (in particular, $u(p)=f(z)$), we can define the local intersection number as follows. The transversality assumption implies that $du$ and $df$ are respectively injective at $p$ and $z$, and 
\begin{equation}
H^{d}_{c}((N_YZ)_{z};G) \cong  H^{n}_c(T_{f(z)}Y) \otimes \Hom\left( H^{n-d}_c(T_zZ),G \right)
\end{equation}
where $(N_YZ)_z =  T_{f(z)}Y/ df_z(T_zZ)$. On the other hand, an orientation of $D$ gives us a fundamental class $[D]_p$ in the locally finite homology group
  $$H_{d}^{lf}(T_pD;\mathbb{Z})\simeq H_{d}^{lf}((N_YZ)_{z};\mathbb{Z}),$$ where the isomorphism is given by $du_p$ composed with the quotient map. The local intersection number is then given by the pairing
  \begin{equation}
H^{d}_{c}((N_YZ)_{z};G)\otimes H_{d}^{lf}((N_YZ)_{z};\mathbb{Z})\to  G.
  \end{equation}If $u$ is transverse to $Z$ and intersects it at finitely many points, we obtain a global intersection number $u \cdot Z$ as the product of these local intersection numbers.

  In general, we obtain an intersection number after perturbation. In order for this number to be well-defined and invariant, some assumption must be imposed:
  \begin{lem}
    If $f$ is proper and the fibre product $Z \times_{Y} D $ is compact,  then the intersection number $u \cdot Z$ is invariant under compactly supported perturbations of $u$. \qed 
  \end{lem}

  We now apply the above general considerations to our specific situation: let $M$ be a closed symplectic manifold of dimension $2n$.
  \begin{defin}
    A \emph{bulk \textbf{b}} on $M$ is a (possibly empty) proper smooth map $f \co B \to  M$ whose source is a smooth manifold of dimension $2n-2$, which is equipped with a $U_\Lambda$-valued (i.e. unitary) coorientation.
  \end{defin}

\begin{rem}
  Our notion of bulk unifies that of background classes (which appear in discussions of signs \cite{Wehrheim2015}) and bulk deformations as in \cite{Fukaya2009}, and has a natural direct sum operation given by unions of manifolds. In the first case, a background class is defined to be an element of $H^2(M; \bZ/2)$, and Thom's theorem \cite{Thom1954} that every cohomology class with $\bZ/2$ coefficients is represented by a manifold, together with the canonical homomorphism $\bZ/2 \to U_{\Lambda}$ taking the generator to $-1$ shows that background classes are recovered by our framework. To recover the notion of bulk from  \cite{Fukaya2009}, we recall that they formulate bulk classes as elements of $H^2(M, \Lambda)$ of strictly positive valuations. Such elements give elements of $H^2(M; U_\Lambda)$ by exponentiation. Since they also restrict attention to characteristic $0$, the fact that integral cohomology classes are represented by co-oriented submanifolds readily leads to a lift of such classes to our setting.

  What is patently missing from our framework are bulk classes associated to elements of $U_{\Lambda}$ of odd torsion, but which do not lift to $H^2(M; \bZ)$. Such classes can be constructed from classes in $H^2(M; \bZ/p)$ on which the Bockstein homomorphism does not vanish, but cannot generally be represented by a manifold, so that a generalization of our approach is required, which we expect can be achieved by Baas-Sullivan theory.
\end{rem}

\begin{rem}
The reader may have encountered the more commonly used approach to formulating bulks in the  mirror symmetry literature that uses de Rham theory. In this alternative formalism the ground ring is the complex numbers (or any field extension of real numbers) and the bulk is represented by a closed $2$-form $\beta$ with values in $\Lambda_0$. In place of twisted local systems (introduced below in Definition \ref{def-twisted-local}), one considers principal $U_{\Lambda}$ bundles with connections on Lagrangians whose curvature is equal to the restriction of $-\beta$. The product of the holonomy around the boundary with the exponential of the integral of $\beta$ over the disk is then part of the contribution of each curve. Restricting to embedded bulks for simplicity of exposition, our formalism is related to this one by associating to a submanifold the corresponding Thom form.  We introduce our version since we want to be able to work over more general ground rings but the de Rham framework would allow us to remove some transversality assumptions in later discussions.
  \end{rem}

  We also fix the following structure on $M$ throughout the paper:
\begin{defin} \cite{Seidel2000}
A \emph{$\mathbb{Z}/2$-grading structure} is given by a smooth fiber bundle $\mathcal{G}_2\to M$ which factors through the Lagrangian Grassmannian bundle $\text{LGr}\to M$ such that over each $x\in M,$ the map on the fibers $\mathcal{G}_{2,x}\to \text{LGr}_x $ is a connected two-fold covering space.
\end{defin}
We define a compatible \emph{$\mathbb{Z}$-grading structure} to be a smooth fiber bundle $\mathcal{G}_\infty\to M$ which factors through $\mathcal{G}_2\to M$ such that over each $x\in M,$ the map on the fibers $\mathcal{G}_{\infty, x}\to \mathcal{G}_{2,x} $ is a universal covering space. If we are given a compatible $\mathbb{Z}$-grading structure, we say that we are in the $\mathbb{Z}$-graded setup.

\begin{rem}
Isomorphism classes of $\mathbb{Z}/2$-grading structures are in one-to-one correspondence with $H^1(M;\mathbb{Z}_2),$ which classifies real line bundles on $M$. Given a real line bundle $E\to M,$ we can construct $\mathcal{G}_2\to M$ as the double cover of $\text{LGr}$ with fiber over $(x,l)$ given by the isomorphisms modulo positive scaling between $\det_\mathbb{R}(l)$ and $E_x.$
\end{rem}

\subsection{Lagrangian branes}
\label{sec:lagrangian-branes}

For the next definition, we fix a (commutative) ring $R$, and a multiplicative subgroup $U_R$ of invertible elements, which will give us a special case of the abelian group $G$ considered in the previous section. We shall use the notion of a \emph{divisorial pseudo-cycle} over $Y$: a topological space stratified by manifolds of dimension at most $\dim Y -2$, none of which have dimension $\dim Y -3$:
\begin{defin}\label{def-twisted-local}
  Let $Y$ be a smooth manifold and $f \co  Z  \to Y$ a map from a codimension-$2$ manifold whose top stratum is equipped with a $U_R$-valued coorientation, whose image is the top stratum of a properly embedded divisorial pseudo-cycle. A \emph{$Z$-twisted rank-$1$ $U_R$-local system} $E$ of free $R$-modules on $Y$ is a collection of $R$-modules $E_x$ for every point $x$ in the complement of the image of $Z$, and an isomorphism
  \begin{equation}
   g_\gamma: E_{\gamma(0)} \to E_{\gamma(1)}
  \end{equation}
for every piecewise smooth path $\gamma:[0,1]\to  Y \setminus f(Z)$ such that
  \begin{enumerate}
        \item (Identity) if $\gamma$ is constant, $g_\gamma=\mathrm{Id}_{E_\gamma(0)},$
        \item (Concatenation) if $\gamma(1)=\tau(0),$ then $g_{\gamma\cdot \tau}=g_\tau \circ g_\gamma,$
          \item (Unitarity) if $\gamma$ is a loop, then $g_{\gamma}$ lies in $U_R$, and
        \item (Homotopy) if $\gamma$ and $\tau$ are homotopic in $Y$ by a homotopy whose intersection number with  $Z$ is $g \in U_R$,  then $$g_\gamma=g\cdot g_\tau.$$
    \end{enumerate}
\end{defin}
Note that the intersection number is well defined because the homotopy may be assumed to be smooth in the interior and transverse to $Z$. This is where the condition of non-existence of codimension $3$ strata arises. This intersection number may a priori itself depend on the homotopy class of the homotopy.

\begin{example}
    Let us consider the case where $Y$ is the real projective plane and $Z$ consists of finitely many points each equipped with its non-zero $\mathbb{Z}/2$-coorientation. Assume that $-1\in R$ is contained in $U_R$ and consider the group homomorphism $\mathbb{Z}/2\to U_R$, which gives us a $U_R$-coorientation. It is elementary to see that up to isomorphism, $Z$-twisted rank one $U_R$-local systems of free $R$-modules on $Y$ correspond to elements $a\in U_R$ such that $a^2=(-1)^{|Z|}.$ In particular whenever $Z$ has an even number of elements, we simply have a local system on $Y$, whereas in the odd case, we have a genuinely different structure. Let us also note that for $Y$ a two-sphere with $Z$ an odd number of points, we have another phenomenon, namely that there do not exist any rank-$1$ $Z$-twisted local systems unless the characteristic is $2$. This phenomenon was explored in detail in \cite{Rezchikov2019}.
\end{example}

Our next definition is designed to admit, in the Fukaya category, objects which are supported on Lagrangians $L$ that are more general than those which admit $\Pin$ structures. In order to do this, we shall need to identify a divisorial pseudo-cycle $W_2(L)$, representing the second Stiefel-Whitney class $w_2(L)$, whose complement has a natural $\mathrm{Pin}$ structure with the property that extending this $\mathrm{Pin}$ structure from a circle in the complement to a bounding disc is given by the intersection number with this cycle.

We start by considering the universal problem: what is the space of vector bundles on which the obstruction universally vanishes? It is well-known that the second Stiefel-Whitney class vanishes on all (real) line bundles, so that the universal bundle over $\bR \bP^\infty$ admits a $\Pin$ structure. Stabilizing by a trivial factor yields a $\Pin$ structure on the space of vector bundles which can be written as the direct sum of a trivial factor with a line bundle. Switching to a Grassmannian model for the classifying space of vector bundles, we can identify an open dense set: let $\bR^{N-n+1} \subset \bR^{N}$ denote the subspace on which the last $n-1$ coordinates vanish, and consider the subvariety $W_2$ of the Grassmannian of $n$-planes in $\bR^{N}$ consisting of subspaces meeting $\bR^{N-n+1}$ in a subspace of dimension at least $2$. The remaining subspaces meet $\bR^{N-n+1}$ transversely, i.e. in a $1$-dimensional subspace, and their projection to $\bR^{n-1}$ is surjective, so that the complement of $W_2$ precisely classifies vector bundles with a splitting into a trivial factor and a line bundle (after taking the limit $N \to +\infty$).  A classical computation \cite[Theorem~12.1]{Milnor-Stasheff} shows: \begin{lem} The cycle $W_2$ represents the second Stiefel-Whitney class of the tautological bundle. \qed \end{lem} From this viewpoint, universality of the construction immediately yields the claim that extending a $\Pin$ structure from a circle to its bounding disc is given by the intersection number with $W_2$.

In order to apply this to our situation, we choose a classifying map which is transverse to $W_2$ (this makes sense because $W_2$ is smooth away from codimension $2$, as can be seen by computing the dimension of the locus where the intersection has dimension $3$), the inverse image in $L$ is then our chosen representative of $W_2(L)$.   We denote by $P$ the distinguished  $\mathrm{Pin}:=\mathrm{Pin}^+$ structure on the complement of  $W_2(L)$ in $L$, which is obtained by pullback.

\begin{rem} In subsequent discussions, we shall require that various evaluation maps be transverse to $W_2(L)$. This should be understood in the weakest stratified sense as transverse to all of its strata. In particular, two and three dimensional manifolds mapping to $L$ are generically disjoint from the singular locus of $W_2(L)$.
\end{rem}

Let us now take a $\mathbb{Z}_2$-graded symplectic manifold $M$ and a bulk \textbf{b} with underlying map $B\to M.$

\begin{defin} \label{def:Lagrangian_brane}

 A \emph{Lagrangian brane} $\bfL$  is a compact connected Lagrangian submanifold $L \subset M$ that is transverse to $B$ (whose preimage we denote $B(L)$), equipped with the following structures:
  \begin{enumerate}
  \item if $\mathrm{char}(\Bbbk) \neq 2$, a codimension-$2$ manifold $W_2(L) \to L$ whose Poincar\'e dual is the second Stiefel-Whitney class $w_2(L)$; if $\mathrm{char}(\Bbbk) = 2$, we declare $W_2(L)$ to be empty.
  \item a lift of the canonical section of $\text{LGr}\mid_L\to L$ to a section of $\mathcal{G}_2|_L\to L$
      \item a $W_2(L) \sqcup B(L)$-twisted rank one $U_\Lambda$-local system $E_{\bfL}$ of free $\Lambda_0$-modules on $L$,
  \item an $\omega$-tame almost complex structure $J_L$. 
  \end{enumerate}
  In the $\mathbb{Z}$-graded setup, if in addition, we have a section of $\mathcal{G}_\infty|_L\to L$ lifting $\mathcal{G}_2|_L\to L,$ we say that $\bfL$ is a $\mathbb{Z}$-graded Lagrangian brane.
  \end{defin}

  Note that $B(L)$ is a codimension-$2$ $U_\Lambda$-cooriented manifold over $L$ by the transversality requirement.  We call the image of $W_2(L)$ in the definition the \emph{pinless locus} of the brane, and note that the fact that $w_2(L)$ is the obstruction to the existence of a $\mathrm{Pin}$ structure implies that the $\mathrm{Pin}$ structure on the complement satisfies the following property: for each Riemann surface with boundary mapping to $L$, whose boundary is disjoint from $W_2(L)$, the intersection number with $W_2(L)$ agrees with the count, modulo $2$, of the number of boundary components on which the $\mathrm{Pin}$ structure on $TL$ does not bound.

 \begin{rem}
   A good example to keep in mind for the above complicated setup is that of an Enriques surface $X$, which admits a symplectic structure so that there is a nodal Lagrangian torus fibration $X\to \mathbb{R}\mathbb{P}^2$ admitting a Lagrangian section. To obtain $\bZ/2$ gradings, we choose the bundle $\mathcal{G}_2\to X$ to be the non-trivial one. Note that the non-trivial real line bundle on $X$ can be obtained by pulling back the determinant line bundle of $T\mathbb{R}\mathbb{P}^2\to \mathbb{R}\mathbb{P}^2,$ for which the Lagrangian section admits the desired lift $\mathcal{G}_2\mid_L\to L$ tautologically.  If we consider the $\mathbb{Z}$-grading structure of $X$ given by the vertical Lagrangian distribution of the fibration (whose corresponding $\bZ$-grading extends over the singularities of the fibration), then there is a preferred $\mathbb{Z}$-grading on the Lagrangian section.

  For the choice of bulk, recall that $w_2(\mathbb{R}\mathbb{P}^2)$ is not zero and moreover the restriction map sends $w_2(X)$ to $w_2(\mathbb{R}\mathbb{P}^2)$. The latter holds whenever $\bR \bP^2$ is a totally real submanifold of a symplectic manifold $X$, using the fact that $w_2(\mathbb{R}\mathbb{P}^2)=w_1(\mathbb{R}\mathbb{P}^2)\cup w_1(\mathbb{R}\mathbb{P}^2)$ and the Whitney sum property of Stiefel-Whitney classes. 

      Our setup allows us to incorporate the section as an object of the Fukaya category in two essentially distinct ways:
      \begin{enumerate}
      \item We may choose the bulk to be any submanifold representing $w_2(TX)$ (with its non-zero $\mathbb{Z}/2$-coorientation). In this case twisted local systems are usual local systems.
        \item Another option is to take the empty background class, but consider twisted local systems on the Lagrangian section. In this latter case, we need $\Bbbk$ to contain a square root of $-1$ in order to have any twisted local systems, as first discovered by Rezchikov.
        \end{enumerate}
        We expect these two different choices to correspond to Brauer classes in the mirror Enriques surface.
 \end{rem}

\subsection{Weakly unobstructed Lagrangian branes}

In general, Floer-theoretic constructions involving Lagrangian branes require the use of virtual fundamental chains. However, there are certain classes that can be handled using simpler methods. In order for the notion to be reasonable to check, one should impose some condition on holomorphic spheres, and the standard way to do so is to impose the following condition:
\begin{defin} \label{def:weakly-monotone_M}
A symplectic manifold \((M,\omega)\) is spherically nonnegative if every class \(A\in\pi_2(M)\) with \(\omega(A)>0\) satisfies$$ \langle c_1(TM),A\rangle\ge 0. $$We assume throughout that \(M\) satisfies this condition.
\end{defin}

For a tame almost complex structure $J_L$ we consider the moduli space $\Mbar_{1}(L,J_L)$ of finite energy stable pseudo-holomorphic discs with $1$ boundary marked point whose boundary maps to $L$. This moduli space admits a natural evaluation map
\begin{equation} \label{eq:evaluation_one-gon-to-fibre-product}
  \Mbar_{1}(L; J_L) \to L.
\end{equation}
For the next definition, we recall that the (virtual) dimension of the union $\Mbar_{1}(L; J_L , m) $ of components of $ \Mbar_{1}(L; J_L)$ consisting of curves of Maslov index $m$ is $m+n-2$. In particular, the expected codimension of the  image in Equation \eqref{eq:evaluation_one-gon-to-fibre-product} is $2 - m $. Note as well that the next definition does not use any data about the brane other than the almost complex structure. 

\begin{defin} \label{def:weakly-unobstr-lagr}
   A Lagrangian brane  $\bfL$ is \emph{weakly tautologically unobstructed} if
     there are no elements of  $\Mbar_1(L; J_L)$ of Maslov index strictly smaller than $2$, and the union of components consisting of discs of index $2$ satisfies the following property: 
     \begin{equation} \label{eq:Biran-Cornea-conditions}
       \parbox{30em}{ there is a subset $Z(L; J_L) \subset L$ which is a countable union of manifolds of codimension $2$ so that all elements of $\Mbar_{1}(L; J_L , 2)$ whose image under evaluation lies outside this subset are regular.  }
     \end{equation}
   \end{defin}
It is important to note that our regularity assumption is stated for curves without constraint at the marked point. We will refer to the locus $Z(L; J_L) \subset L$ from Condition \eqref{eq:Biran-Cornea-conditions} as the \emph{irregular locus} of the brane.
Let us use the phrase \emph{tautologically unobstructed} if instead of Condition \eqref{eq:Biran-Cornea-conditions} we simply have no element of  $\Mbar_1(L; J_L)$ of index less than or equal to $2$.

 
\begin{prop}[Biran-Cornea] \label{prop:BS-transversality}
  If $L \subset M$ is a monotone Lagrangian submanifold with minimal Maslov index greater than or equal to $2$, then, for a generic choice of almost complex structure $J_L$, Condition \eqref{eq:Biran-Cornea-conditions} holds. \qed
\end{prop}

\begin{prop}[Fukaya]
  If $(L,J_L)$ satisfies Condition \eqref{eq:Biran-Cornea-conditions}, then any Lagrangian sufficiently $C^\infty$ close to $L$ admits an almost complex structure satisfying this condition. 
\end{prop}

\begin{examples}
  If $\dim_{\bR} M \leq 4$, $c_1(M) = 0$, and $L$ is an embedded Lagrangian vanishing Maslov class, then a generic choice of $J_L$ ensures that $\Mbar_1(L; J_L)$ is empty whenever $L$ is embedded.
\end{examples}

\begin{example}
    Let $M$ be a smooth compact Fano toric variety with complex structure $J$ and let $\omega$ be a monotone toric  K\"ahler form. Then all the torus fibers equipped with the complex structure $J$ are weakly tautologically unobstructed. Indeed, denote by $D$ the toric divisor and let $L\subset M\setminus D$ be any fiber. Then by the maximum principle there are no non-constant holomorphic discs in $M\setminus D$ with boundary on $L$. By positivity of intersections any non-constant $J$-holomorphic disc with boundary in $L$ has Maslov index $\geq 2$ since $D$ is the anti-canonical divisor. Since \(M\) is Fano, every non-constant \(J\)-holomorphic sphere has positive first Chern number. Hence every stable \(J\)-holomorphic disc with boundary on \(L\) has Maslov index at least \(2\). Moreover, a stable disc of Maslov index \(2\) has no non-constant sphere components. Indeed the only other possibility is a sphere $C$ in $M$ such that $C\cdot D=1$. This would give rise to a non-constant algebraic map $\bC\to(\bC^*)^n$ which is impossible since each coordinate would be a nowhere-vanishing polynomial on \(\mathbb C\), hence constant.  Finally, by the regularity theorem of Cho--Oh
    \cite[Theorem~6.1]{Cho2006}, all $J$-holomorphic discs with boundary on a toric fiber are Fredholm regular. Hence Condition \eqref{eq:Biran-Cornea-conditions} holds with empty irregular locus.

   Observe further that this conclusion concerning weak tautological unobstructedness is independent of the choice of toric K\"ahler form. Indeed, for any toric K\"ahler form compatible with the same complex structure $J$, the torus orbits remain Lagrangian, while the moduli spaces of $J$-holomorphic discs, their Maslov indices, and their linearized operators are unchanged. Thus $(L,J)$ remains weakly tautologically unobstructed in the sense of Definition~\ref{def:weakly-unobstr-lagr}, although such a symplectic form need not satisfy the standing assumption of Definition~\ref{def:weakly-monotone_M}.
    
\end{example}

\begin{example}
Consider the conic fibration $M=\{xy=f(u_1,\dots,u_n)\}\subset \bC^2\times(\bC^*)^n$. The authors of \cite{Abouzaid2016b} construct a Kahler form $\omega$ and Lagrangian torus fibration with base $\bR^{n+1}$ whose singular locus is the amoeba of $f=0$ inside $\{0\}\times \bR^n$. The torus fibers that do not meet the conics over the hypersurface $\{f=0\}\subset (\bC^*)^n$ are tautologically unobstructed \cite[Proposition 5.1]{Abouzaid2016b}.
The potentially obstructed fibers lie on a closed domain diffeomorphic to the product of $\bR$ with the amoeba.

The mirror dual $Y$ of this example is the Gross fibration $\pi:Y\to\bR^{n+1}$ on the complement $X\setminus D$ of an anti-canonical divisor in a  toric Calabi-Yau variety $X$ introduced in \cite{Gross2001b}. 
$Y$ carries a global holomorphic function $w_0:Y\to\bC$ and all the torus fibers that avoid the zero locus of this function are tautologically unobstructed
\cite[Lemma 8.1]{Abouzaid2016b} or \cite[Lemma 4.27]{Chan2012b}. In fact $\pi(w_0^{-1}(0))$ is a hypersurface in $\bR^{n+1}$ diffeomorphic to a hyperplane. 

In the latter example, one can also consider the torus fibers as submanifolds in $X$ \cite{Auroux2007,Chan2012b}. The fibers off the wall $\pi(w_0^{-1}(0))$ only have discs of Maslov index $\geq 2$. One can further perturb the pair $(J,w_0)$ to a pair of an almost complex structure $J'$ and a $J'$-holomorphic function $w_0$. This perturbation can be supported away from $w_0^{-1}(0)$ as well as the divisor $D$ and still achieve transversality for simple discs with boundary on torus fibers off the wall. For such $J'$ we still have that all the discs have Maslov index $\geq 2$. The argument for Condition \eqref{eq:Biran-Cornea-conditions} then proceeds as before.
\end{example}

\subsection{Superpotential values for weakly tautologically unobstructed branes}\label{sec-superpotential}

Our goal in this section is to associate an element of $\Lambda_0$ to each weakly tautologically unobstructed brane $\bfL=(L,J_L,\ldots)$, by an appropriate weighted count of $J_L$-holomorphic disks with boundary on $L$ that send the point $1\in \mathbb{D}$ to a chosen  point lying in the complement of the irregular locus $Z:= Z(L; J_L) \subset L$. Our main result is the independence of this count from all choices.

Let $p \in L \setminus Z$ denote a point in the complement of the irregular locus, which is moreover a regular value of the evaluation map, and write $\Mbar_p(\bfL)$ for the moduli space of stable $J_L$ discs, of Maslov index $2$, whose image under the evaluation map is $p$. 
By the requirements for being weakly tautologically unobstructed, this is a countable discrete set, equipped with an evaluation map
\begin{equation} \label{eq:evaluation_discs_at_p}
  \Mbar_p(\bfL) \times S^1 \to L,
\end{equation}
mapping $1 \in S^1$ to $p$. In particular, we have:
\begin{lem}
  After generic perturbation, the evaluation map in Equation \eqref{eq:evaluation_discs_at_p} is disjoint from the images of the codimension $2$ manifolds  $W_2(L)$  and $B(L)$. \qed
\end{lem}

The appropriate weight $a_u \in \Lambda_0$  of each element $u \in  \Mbar_p(\bfL)$ will be a product of the following terms: \begin{enumerate}
    \item The area of the disc $T^{\int u^*\omega}$
      \item The (weighted) intersection with $B$, which we denote  $ u\cdot B \in U_\Lambda$
   \item the sign $\pm$ arising from the Pin structure on the Lagrangian boundary conditions $L\setminus W_2(L)$ (see \cite{Seidel2008a}).
    \item the monodromy of the twisted local system $g_{\partial u}\in U_\Lambda$.
\end{enumerate}

\begin{lem} \label{lem:weight_indep_perturbation}
  The weight of each element of $ \Mbar_p(\bfL) $ is independent of the choice of perturbation making the boundary disjoint from   $W_2(L)$  and $B(L)$.
\end{lem}
\begin{proof}
  Under a generic homotopy of perturbations, the intersection with  $W_2(L)$  and $B(L)$ will change at finitely many times. In the first case, the monodromy of the twisted local system and the sign contribution will both change by the intersection number with $W_2(L)$, and in the second case, the monodromy will change by the same element of $U_{\Lambda}$ as the interior intersection with $B$. We conclude that the product of the four contributions does not change.  
\end{proof}

\begin{defin}
    The superpotential value of a weakly tautologically unobstructed brane $\bfL$ is the sum $\sum_{u\in \Mbar_p}a_u$ (which makes sense by Gromov compactness) for any regular value $p$ of the evaluation map, lying in $L \setminus Z.$
\end{defin}

\begin{thm}
    The superpotential value is well-defined, in the sense  that it is independent of the point in $L$.
\end{thm}
\begin{proof}
  Let $p$ and $p'$ be such regular values. Since $Z$ has codimension $2$, and we assumed that $L$ is connected, we may choose a generic path $\gamma$ connecting them in the complement of $Z$, so that the moduli space of Maslov $2$ discs passing through elements of $\gamma$, which we denote $\Mbar_{\gamma}$ is a $1$-dimensional manifold with boundary. This manifold is equipped with an evaluation map
  \begin{equation} 
  \Mbar_{\gamma}\times S^1 \to L.
\end{equation}
Let us decompose $  \Mbar_{\gamma}$ into components: these consist either of circles (which do not contribute to our argument) or of paths which either connect $\Mbar_p$  and $\Mbar_{p'}$ or have both endpoints in the same moduli space. In any case, the standard application of Stokes's theorem shows that the area contribution is constant on each component. 

To understand the remaining contributions, choose a perturbation of the evaluation map which is transverse to   $W_2(L)$  and $B(L)$. The same argument as in Lemma \ref{lem:weight_indep_perturbation} now shows that (i) two points in $\Mbar_p$  and $\Mbar_{p'}$ which are connected by a path have the same contribution and (ii) two points in $\Mbar_p$ (or $\Mbar_{p'}$) which are connected by a path have contributions that differ by a sign, and therefore exactly cancel each other out in the algebraic sum. 
\end{proof}


\section{The open string Hamiltonian indexing category}
 Fix a $\mathbb{Z}_2$-graded symplectic manifold $M$ and a bulk \textbf{b}.
\subsection{Regular Floer data: strips}
\label{sec:regular-floer-data-1}

Consider a pair of branes $\bfL_0=(L_0,J_0,\ldots)$ and $ \bfL_1=(L_1,J_1,\ldots)$:
\begin{defin} \label{def:admissible_Hamiltonian}
A Hamiltonian $H_{01} \co M \times [0,1] \to \bR$ is \emph{admissible} for $(\bfL_0, \bfL_1)$ if its time-$1$ map maps $L_0$ transversely to $L_1$ and each time-$1$ chord $\gamma:[0,1]\to M$ satisfies the following: \begin{itemize}
\item the image of $\gamma$ is disjoint from the bulk $B \subset M$. 
\item $\gamma(i)$ is disjoint from the pinless locus $W_2(L_i)$, from the irregular locus $Z({L_i}; J_i)$ for $i = 0, 1$, and from the critical values of the evaluation map of the moduli space of Maslov $2$ discs with boundary on $L_i$.
\end{itemize}
\end{defin}

To preview how the first property will be used, let $E^i$ be the local system on $L_i$, twisted with respect to $B(L_i) \sqcup W_2(L_i)$, and $u \co  \bR\times [0,1]  \to M$ a strip with boundary conditions on $L_0$ and $L_1$, and asymptotic conditions $x$ and $y$ given by time-$1$ chords of $H_{01}$. Generically perturbing $u$ within the space of maps with the given asymptotic and Lagrangian boundary conditions, we obtain a map $u'$ which is transverse to $B$ and whose boundary is disjoint from $W_2(L_i)$. The product of the intersection number of $u'$ with $B$ and the parallel transport maps along the boundary defines a homomorphism:
\begin{equation} \label{eq:map_associated_to_strip}
  \Hom_{\Lambda_0}(E^{0}_{x(0)}, E^{1}_{x(1)}) \to   \Hom_{\Lambda_0}(E^{0}_{y(0)}, E^{1}_{y(1)}).
\end{equation}

Unlike the usual situation, in which the homomorphism in Equation \eqref{eq:map_associated_to_strip} is independent of perturbations, the signed contribution in our case may vary as the boundary of the strip crosses the divisor $B(L_i) \sqcup W_2(L_i)$. This will be accounted for by the signed action on the orientation line.

Consider as well a family $J_{01}$ of almost complex structures on $M$ parametrized by the interval, which agrees at the ends with the complex structures $J_0$ and $J_1$ associated to the branes $\bfL_0$ and $\bfL_1$. This determines a Floer equation
\begin{equation} \label{eq:Floer-equation}
J_{01} \circ \left( du - X_{H_{01}} \otimes dt \right) = \left( du - X_{H_{01}} \otimes dt \right) \circ  j  
\end{equation}
on the strip $\mathbb{R}\times [0,1]$.

For time $1$-chords $x,y$ of $H_{01}$ and a relative homology class $\beta$ we denote by $ \Mbar(x,y,\beta) $ the moduli space of stable (unparametrized) Floer strips from $x$ to $y$. That is, stable configurations consisting of a Floer strip with spheres attached in the interior and discs attached to the boundary (modulo translations of the strip). When $x=y$ this includes the constant strip with discs attached.
Let us now assume that the branes are in addition weakly tautologically unobstructed.
\begin{defin}
We say that the stable moduli spaces $ \Mbar(x,y,\beta) $ are regular (transversely cut) if the following conditions hold:
\begin{enumerate}
    \item The moduli spaces of Floer strips of virtual dimension $< 2$\footnote{note that the dimension of the moduli space of parametrized Floer strips is one more than this} are cut out transversely.
    \item The moduli space $\Mbar_{1}(J_{01})$ of simple $J_{01}(t)$-holomorphic spheres is cut out transversely by the section defined by the $[0,1]$-parametrized family of $J_{01}(t)$-induced Cauchy-Riemann operators over an appropriate functional completion of maps $S^2\to M$ crossed with $[0,1]$.
    \item The subset of the moduli spaces $\Mbar_{1}(L_i;J_{L_i})$, consisting of stable Maslov $2$ discs mapping away from the irregular locus,  is cut out transversely. 
    \item The intersections of strips with the discs and spheres are transverse. 
\end{enumerate}

The pair $(H_{01},J_{01})$ is \emph{regular} if for each relative homology class $\beta$ of strips,  the stable moduli spaces $ \Mbar(x,y,\beta) $ are regular (transversely cut) whenever the virtual dimension is less than $2.$ 
 \end{defin}

\begin{lem}
Given weakly-tautologically unobstructed branes $(\bfL_0, \bfL_1)$ and an admissible Hamiltonian $H_{01}$, there is a path of almost complex structures $J_{01}$ so that $(H_{01},J_{01})$ is regular.
\end{lem}
\begin{proof}
  This is a standard argument, proceeding as follows: Given a path $J_{01}$, consider the moduli space $\Mbar_{1}(J_{01})$ of $J_{01}(t)$ spheres, equipped with its evaluation map
  \begin{equation}
    \Mbar_{1}(J_{01}) \to [0,1] \times M.
  \end{equation}
  By Definition \ref{def:weakly-monotone_M}, the inverse image of the interior \((0,1)\times M\) contains no spheres of negative Chern number. For a generic path, the parametrized moduli space of one-marked simple Chern-zero spheres has dimension $\dim M-3$, which is four less than the dimension of the evaluation target $M\times[0,1]$. We then impose a further genericity condition by requiring that the union of the moduli spaces $\cM(x,y,\beta)$ of strips (rather than stable strips) with virtual dimension less than $2$ be regular, and that the evaluation map to $[0,1] \times M$ be disjoint from the moduli space of Chern $0$ spheres. There is thus no interior bubbling for moduli spaces of this virtual dimension. 

  The assumption of tautological unobstructedness rules out any disc bubbles of index $<2$, which are the only ones that can contribute to a stable configuration of virtual dimension $\leq 1$ in which the strip and sphere components all have non-negative index. 
\end{proof}

Given the weakly tautologically unobstructed branes $(\bfL_0, \bfL_1)$, admissible  Hamiltonian $H_{01}$ and a regular path of almost complex structures $J_{01}$, we define a free $\Lambda_0$ module
$$CF^*(\bfL_0,\bfL_1; H_{01},J_{01}),$$
as the direct sum, over all time-$1$ $H_{01}$-chords $x$ starting at $L_0$ and ending on $L_1$ of the rank-$1$ $\bZ/2$-graded free $\Lambda_0$-modules
\begin{equation}
  \Hom_{\Lambda_0}(E^{0}_{x(0)}, E^{1}_{x(1)})\otimes_{\bZ}  \ro_x,
\end{equation}
where we recall that $\ro_x$ is the graded orientation line associated to the lifts of $L_0$ and $L_1$ to our chosen $2$-fold cover of the Grassmannian of Lagrangians. 

Exactly as in the definition of the potential value, a generic perturbation of a rigid element $u$  of $  \Mbar(x,y,\beta) $ determines a map
  \begin{equation}
   d_u \co   \Hom_{\Lambda_0}(E^{0}_{x(0)}, E^{1}_{x(1)}) \otimes_{\bZ} \ro_x \to  \Hom_{\Lambda_0}(E^{0}_{y(0)}, E^{1}_{y(1)}) \otimes_{\bZ} \ro_y,
   \end{equation}
   which is the product of
   \begin{enumerate}
    \item the topological action of the strip $T^{\int_\Sigma \left( u^*\omega + d(u^*H_{01})\wedge dt \right)}$
    \item the weighted intersection with $B$,
          \item the parallel transport map on the local systems $E^0$ and $E^1$, and
   \item the morphism $\ro_x \to \ro_y$ arising from the Pin structure on the Lagrangian boundary conditions $L_i \setminus W_2(L_i)$ and the grading on $L_i$ (see \cite{Seidel2008a}).
\end{enumerate}

\begin{lem}
The homomorphism $d_u$ is independent of the choice of perturbation. \qed
\end{lem}

The sum of these operators defines a map
\begin{equation}
d \co  CF^*(\bfL_0,\bfL_1; H_{01},J_{01}) \to CF^*(\bfL_0,\bfL_1; H_{01},J_{01}).
\end{equation}
Letting $\rho_0$ and $\rho_1$ denote the superpotential values of $\bfL_0$ and $\bfL_1$, we have:
\begin{thm}
  The differential $d$ is a matrix factorization for $\rho_1 - \rho_0$ in the sense that
  \begin{equation}
    d^2 x = (\rho_1 - \rho_0) x.
  \end{equation}
\end{thm}
\begin{proof}
  Compactness of the one-dimensional moduli space follows from the disjointness from the irregular loci and basic index considerations. Since the boundary consists of broken strips and bubbling index-2 disks along $s$-constant strips, the result follows.
\end{proof}
\begin{cor}\label{cor-same-superpotential}
  If $\bfL_0$ and $\bfL_1$ have the same superpotential value, then $d^2 = 0$. \qed
\end{cor}

\subsection{The (genus-$0$) Lagrangian Floer multicategory}
\label{sec:genus-0-lagrangian}

Our aim is now to define a multicategory which we denote $\mathrm{Floer}$, enriched in symmetric cubical sets, associated to a $\mathbb{Z}_2$- or $\mathbb{Z}$-graded symplectic manifold $M$ and a bulk $\textbf{b}$. 

The objects of $\mathrm{Floer}$ are triples $\bfH = (\bfL_0, (J_{01},H_{01}), \bfL_1)$ where 
\begin{itemize}
\item $\bfL_i$ are Lagrangian branes, 
\item $H_{01}$ is admissible for $(\bfL_0,\bfL_1)$ 
\item The datum $(J_{01},H_{01})$ is regular. 
\end{itemize}
We will refer to these objects also as \emph{boundary conditions}.

In the rest of this section, we construct the $0$-dimensional part of the multimorphisms (Definition \ref{def-final-mult-dim-0}) and their multicompositions (Definition \ref{def:mult-dimens-0}) in this multicategory.

The multimorphism cubical sets are empty except possibly in the following situation: we are given a sequence $\{ \bfL_i\}_{i=0}^{k}$ of branes, and the input sequence is $\vec{\bfH} = \{\bfH_{i,i+1}\}_{i=0}^{k-1}$ where the
underlying pair of Lagrangian branes for the object $\bfH_{i,i+1}$ is $(\bfL_i, \bfL_{i+1}) $, while the output object $\bfH_{0k}$ has $(\bfL_0, \bfL_k)$ as its underlying pair of branes. We call $\vec{\bfH} = \{\bfH_{i,i+1}\}_{i=0}^{k-1}$ a \emph{composable sequence}, and refer to $\bfH_{0k}$ as a \emph{horizontal composition} of $\vec{\bfH}$. A horizontal composition is not unique, as the data $(J_{0k}, H_{0k}) $ is not determined by the inputs.

Let $T$ be a pre-stable ribbon tree with $k$ inputs and one output as in Definition \ref{defRibbonTree}. An \emph{admissible $\mathrm{Floer}$  labeling} of $T$ is a labeling of the edges $e$ of $T$ by objects $\mathbf{H}_e$ so that the following condition holds. Let the input edges be labeled by a composable sequence  $\vec{\bfH} = \{\bfH_{i,i+1}\}_{i=0}^{k-1}$. Then we require that for any edge $e$ emanating from a vertex $v$ the object $\mathbf{H}_e$ is a  horizontal composition of the incoming objects. Imposing this requirement inductively we see that the incoming objects at each vertex indeed form a composable sequence. Similarly the output edge is a horizontal composition of $\vec{\bfH}$. We will refer to $T$ with such a labeling as a $\mathrm{Floer}$ labeled tree with input $\vec{\bfH}$ and output $\mathbf{H}_{0k}$. 

Given a pre-stable disc $\Sigma$ with $k$ inputs and one output (see Definition \ref{def-pre-stable-disc}) we shall refer to its underlying tree by $T_{\Sigma}$. We remind the reader that the data of a prestable disc with $k$ inputs and one output includes, for every vertex $v$ of $T_\Sigma$, a choice of strip-like ends $\epsilon^{\pm}$ for the component $\Sigma^v$ of $\Sigma$ (see Definition \ref{def-strip-like}). Given an {admissible $\mathrm{Floer}$  labeling} of $T_\Sigma$ with input $\vec{\bfH}$ and output $\mathbf{H}_{0k}$, each input of each component $\Sigma^v$ of $\Sigma$ is indexed by an incoming edge $e$ at $v$, and is labelled the object $\bfH_e$. The output is labeled by a horizontal composition of these inputs. In particular, considering $\Sigma^v$ as a disc with boundary punctures, we obtain an induced labeling of the components of $\partial\Sigma^v$ by Lagrangian branes.


We then have:

\begin{defin}\label{def:pre-multimorphism-H}
A \emph{pre-multimorphism}  with input $\vec{\bfH} = \{\bfH_{i,i+1}\}_{i=0}^{k-1}$ and output $\mathbf{H}_{0k}$ consists of 

\begin{enumerate}
\item a pre-stable disc $\Sigma$ with $k$ inputs and one output
  \item an admissible $\mathrm{Floer}$ labeling of $T:=T_{\Sigma}$ with input  $\vec{\bfH}$ and output $\mathbf{H}_{0k}$, and
    \item for each vertex $v$ of $T$, \begin{itemize}
  
    \item a  smooth function  $H^v \co \Sigma^v \times M \to \bR$, 
    \item an almost complex structure $J^v$ on $M$, smoothly parametrized by $\Sigma^v$, and
    \item a closed $1$-form $\alpha^v$ on $\Sigma^v$.
    \end{itemize}
\end{enumerate}
These data are required to satisfy the following properties for each vertex $v$ of $T$:
\begin{enumerate} 
    \item the restriction of $\alpha^v$ to each boundary component of $\Sigma^v$ vanishes,
    \item $J^v$  agrees with $J_{i}$ at any point of a boundary component of $\Sigma^v$ labeled by  $L_i$,
    \item $H^v$ satisfies the monotonicity inequality (see Appendix \ref{sec:monotonicity-inequality})   
    \begin{equation}
\label{eq:strict-monotonicity}
    \int_{\Sigma^v} \min_{x\in M} \left(d_{\Sigma^v} (H^v(\cdot, x)) \wedge \alpha^v\right) > 0,
\end{equation}
    \item for each edge $e$ of $T$ adjacent to $v$ and the corresponding puncture $p$ of $\Sigma^v$, the almost complex structure $J^v$ pulls back under the strip-like end $\epsilon_p$ to the time dependent  almost complex structure $J^e$ encoded in $\bfH_e$. Moreover, there is a positive real number $w_p$ called the \emph{weight}, such that we have
  \begin{align}
\label{eqTransInvDat0}
     \epsilon_p^{*}\alpha^v& =w_pdt \\
\label{eqTransInvDat}
        \epsilon_p^{*}(H^v\alpha^v) & =H^{e}dt,        
  \end{align}
  \item for each internal edge of $T$ with endpoints $v_-$ and $v_+$, the weights of $\alpha^{v_-}$ and $\alpha^{v_+}$ at the punctures associated to $(v_-,e)$ and $(v_+,e)$ agree.

    \item if $p_{out}$ is the output of $\Sigma^v$ and $p$ is an input of $\Sigma^v$:     \begin{equation}\label{eq:weight_bound}\frac{w_{p}}{w_{p_{out}}}\geq \frac{1}{2^{|E_{in}(v)|-1}}.\end{equation}
\end{enumerate}
\end{defin}

See \cite[Remark 2.2]{Abouzaid2023} for motivation of various parts of this definition.

 \begin{defin} Two pre-multimorphisms are \emph{equivalent} if there is an isomorphism of the underlying pre-stable discs as in Definition \ref{def-isom-disc} that intertwines the data of $\bfH_e$, and the data of $\{(H^v,J^v,\alpha^v)\}$ up to the equivalence relation
 \begin{equation}
 \{(e^{r}H^v,J^v,e^{-r}\alpha^v)\}\sim \{(H^v, J^v,\alpha^v)\}\quad\forall r\in\bR.
\end{equation}\end{defin}

This leads to the following notion:
\begin{defin}\label{def-final-mult-dim-0}
  A \emph{multimorphism of dimension $0$} is an equivalence class of pre-multimorphisms. We denote by $\mathrm{Floer}_0(\vec{\bfH},\bfH_{0k})$ the set of all multimorphisms of dimension $0$. For  a pre-stable ribbon tree $T$ with $m$ inputs and one output we denote by $\mathrm{Floer}_{0,T}(\vec{\bfH},\bfH_{0m})$ the subset of multimorphisms defined on pre-stable discs modeled on $T$.
\end{defin}


For the definition of multicomposition, we consider composable sequences $\vec{\bfH}_1\in\mathrm{Floer}^{k_1}$ and $\vec{\bfH}_2\in\mathrm{Floer}^{k_2}$, where the notation indicates the cartesian product on the object sets. We further assume that the $i$th component of $\vec{\mathbf{H}}_1$ is a horizontal composition of $\mathbf{\vec{H}}_2$.  We then write
\begin{equation} \label{eq:multi-composition-Ham}
  \vec{\bfH}_1 \circ_i \vec{\bfH}_2  
\end{equation}
for the composable sequence with $k_1 + k_2 - 1$ elements obtained by replacing the $i$th component of $\vec{\bfH}_1$ with the vector $\vec{\bfH}_2$. 

Given two admissible $\mathrm{Floer}$ labeled trees $T_1$ and $T_2$ such that $i$th input label of $T_1$ is the same as the output label of $T_2$, we can construct a new admissible $\mathrm{Floer}$ labeled tree $T_1 \circ_i T_2$  by taking the disjoint union of $T_1$ and $T_2$ and identifying the output edge of $T_2$ with the $i$th input edge of $T_1.$ 

This operation lifts to a multicomposition operation on multimorphisms. 
\begin{defin} \label{def:mult-dimens-0} 
  The multicomposition
\begin{equation}
\circ_i:\mathrm{Floer}_0(\vec{\bfH}_2;(\bfH_1)_{i-1,i})\times \mathrm{Floer}_0(\vec{\bfH}_1;(\bfH_1)_{0k_1})\to\mathrm{Floer}_0(\vec{\bfH}_1 \circ_i \vec{\bfH}_2 ;(\bfH_1)_{0k_1})
\end{equation} takes $(\fd_2,\fd_1)$ to the multimorphism defined by attaching the output of $\fd_2$ to the $i$th input of $\fd_1$.
\end{defin}
The compatibility condition  in Definition \ref{def:pre-multimorphism-H} between the two data at the edge of $T_1 \circ_i T_2$ along which the trees are attached is ensured by choosing representatives of $\fd_1$ and $\fd_2$ such that the output weight of $\fd_2$ agrees with the $i$th input weight of $\fd_1$. To avoid any confusion, we reiterate that this definition does not involve any gluing of Riemann surfaces (or Hamiltonian data).

\subsection{Higher cubes of multimorphisms} 
\label{sec:high-cubes-mult}

As in \cite{Abouzaid2023}, rather than the naive strategy of defining a smooth structure on the set of multimorphisms of dimension $0$ and considering the smooth singular cubes of this target we shall directly construct positive dimensional cubes of multimorphisms as cube families of multimorphisms of dimension $0$ which are obtained by gluing near the boundary strata.

We first introduce some notation for the gluing operation that will be used in the definition of higher dimensional cubes. We consider an element $\fd\in\mathrm{Floer}_0(\vec{\bfH},\bfH_{0n})$ with underlying $\mathrm{Floer}$  labeled tree $T$. 

\begin{defin} The gluing $\Gamma_{\vec{r}}(\fd)$ of $\fd$ with parameters $\vec{r}\in [0,1)^{E_{int}(T)}$ is the multimorphism obtained by gluing the underlying  pre-stable disc, equipped with the restriction of the data $(H^{v_-},J^{v_-},\alpha^{v_-})$ and $(H^{v_+},J^{v_+},\alpha^{v_+})$ for each $e\in E_{int}(T)$ connecting vertices $v_-$ and $v_+$ with $r_e>0$ (see Definition \ref{def-gluing-general}). 
\end{defin}

Note that the compatibility condition of weights along the two sides of the node, which we imposed in Definition \ref{def:pre-multimorphism-H},  ensures that the Hamiltonian data on the glued Riemann surface are well-defined after gluing. In addition, the equality $2^{n-1}\cdot 2^{m-1}=2^{(n+m-1)-1}$ shows that Condition \eqref{eq:weight_bound} is satisfied.

Given a possibly broken $\fd\in \mathrm{Floer}_0(\vec{\bfH},\bfH_{0n})$ we call the triple $(T,\fd',\vec{r})$ of a $\mathrm{Floer}$ labeled pre-stable ribbon tree $T$, $\fd'\in \mathrm{Floer}_{0,T}(\vec{\bfH},\bfH_{0n})$ and  $\vec{r}:{|E_{int}(T)|}\to [0,1)$ such that \begin{equation}
\fd=\Gamma_{\vec{r}}(\fd') 
\end{equation}\emph{a gluing decomposition} of $\fd$. A given element $\fd$ of $ \mathrm{Floer}_0(\vec{\bfH},\bfH_{0n})$ may have more than one gluing decomposition.

\begin{lemma}\label{lmGluingMonotonicity} Any two gluing decompositions of $\fd$ with the same underlying ribbon tree $T$ and 
Hamiltonian datum $\fd'\in \mathrm{Floer}_{0,T}(\vec{\bfH},\bfH) $ have equal gluing parameters. 
\end{lemma}

\begin{proof}
  Consider the subset $\cF(\fd')\subset \mathrm{Floer}_{0}(\vec{\bfH},\bfH)$ consisting of multimorphisms obtained from $\fd'$ by gluing. We will construct for each edge $e$ a function $p_e:\cF(\fd')\to \bR$. We will then  show that $p_e\circ\Gamma_-(\fd'):[0,1)^{E_{int}(T)}\to \mathbb{R}$ depends only on $r_e$ and is strictly monotone in $r_e$. This will prove the desired result. 
  
Note that an edge $e$ collapsed to $v$ gives a canonical embedding $$\epsilon_e: [0,1]\to \Sigma^v$$ where the boundaries of the two pieces are glued together as in Definition \ref{dfGluing0}.  

If $e$ is not collapsed, we define $p_e(\fd)=0$. If it is collapsed to vertex $v,$ we define $p_e(\fd)$ as the infimum 
of $p \in (0,1)$ such that there is a holomorphic embedding $E:(\log(p),-\log(p))\times [0,1]\to \Sigma^v$, sending boundary to boundary and satisfying the following properties
\begin{align*}
E(0,\cdot) &=\epsilon_e \\
  E^* \alpha^v & = c dt \textrm{ for some }    c>0 \\
  E^*\left(  H^{v}_{x} \alpha^v \right) & =  H^e dt.
\end{align*}
Since the gluing parameter is less than $1$, such embeddings always exist.

Take such embeddings for two distinct edges $e$ and $f.$ We claim that images of the embeddings must be entirely disjoint. Assume the contrary, namely that the embeddings $E_e:(\log(p_e),-\log(p_e))\times [0,1]\to \Sigma^v$ and $E_f:(\log(p_f),-\log(p_f))\times [0,1]\to \Sigma^v$ intersect. Here we used the fact that if there is intersection, $e$ and $f$ must collapse to the same vertex.

Take the intersection of the images of the two embeddings. We obtain a biholomorphism $\phi$ from an open subset $U\subset (\log(p_e),-\log(p_e))\times [0,1]$ to an open subset $U'\subset (\log(p_f),-\log(p_f))\times [0,1]$ which sends level sets of $t$ to level sets of $t$ (because $\phi^*dt$ is a positive real multiple of $dt$), and more importantly, the sends the level sets of $s$ to level sets of $s$ (because $\phi^*ds=\phi^*(j^*dt)=j^*(\phi^*dt)$ is a positive real multiple of $ds=j^*dt$). Moreover, if $x_1,x_2,\ldots$ is a sequence in $U$, which converges in $(\log(p_e),-\log(p_e))\times [0,1]$, but not to a point in $U$, then $\phi(x_1),\phi(x_2),\ldots$ must not be convergent in $(\log(p_f),-\log(p_f))\times [0,1]$. The same statement holds for $\phi^{-1}$ as well. Otherwise, the two limit points would also map to the same point in $\Sigma^v$ and hence would have to be inside $U$ to begin with. We deduce that for both $U$ and $U'$, if they intersect a level set of $s$, then they must contain the entire level set. In particular, we obtain that the free homotopy classes of $e$ and $f$ must be the same. This proves that $\Sigma^v$ is an infinite strip (one input and one output) and without loss of generality, we can assume that $f$ is closer to the root of $T$ than $e$.

Moreover, by integrating $\alpha^v$ in this free homotopy class, we see that in fact $\phi^*dt=dt$ and again by holomorphicity $\phi^*ds=ds$. Therefore, $\phi$ must be a translation in $s$ and therefore $H^e=H^f$ since $\phi^*(H^fdt)=H^edt.$ By the monotonicity inequality, this is a contradiction. We have proved the desired disjointness. 

Let $v_\pm$ be the vertices adjacent to an interior edge $e$ of $T$. The disjointness we proved implies that the image of such an $E_e$ is contained in the regions obtained from the Riemann surfaces $\Sigma^{v_-}$ and $\Sigma^{v_+}$ (after removing gluing regions associated to the other punctures). It follows that  $p_e\circ \Gamma$  depends only on the gluing parameter $r_e$.

In fact, in the proof of the disjointness statement, if we took embeddings $E_e$ and $E_e'$ both for the same edge $e$, we would prove in the same way that the inverse image of locus of intersection under either map contains every level set of $s$ which it intersects. This implies that $\phi$ must be a translation in $s$, and because they both extend $\epsilon_e$, the map $\phi$ is actually the identity map. Therefore, for any given such embedding $E_e$, taking $E_e'$ to be a gluing annulus extension of $\epsilon_e$ whose image lies inside the image of $E_e$, we see that $E_e$ extends $E_e'$. Therefore, we can conclude that  $ p_e\circ\Gamma$  is a monotone function of $r_e$.

\end{proof}




Before defining the higher dimensional cubes of multimorphims, we need to describe their local models which determine the relevant data near all faces of the cube:  \begin{defin}\label{dfLocModMult}
  A \emph{local model for a codimension $k$ corner of a family of multimorphisms} with input $\vec{\bfH}$ and output $\bfH$ consists of the following data:

\begin{enumerate}
\item (Domain of maximal breaking) a smooth manifold $C$, 
\item  (Collar neighbourhood) an open neighbourhood  $U$ of the origin in $[0,1]^k$ (in the case $k=0$ we write $U=\{0\}$).
\item (Choice of broken curves) A $\mathrm{Floer}$ labeled tree $T$, and a smooth map
  \begin{equation} \label{eq:broken_curve_family}
        b \co C\times U \to\mathrm{Floer}_{0,T}(\vec{\bfH},\bfH).
  \end{equation}
\item (Gluing data) A smooth map
  \begin{equation} \label{eq:gluing_data}
   g:C\times U\to[0,1)^{|E_{int}(T)|}     
  \end{equation}
whose components vanish identically on $C\times\{0\}$.  Moreover, for each  face $\sigma$ of $U$, any component $g_i$ of $g$ either vanishes identically or is nowhere $0$ on the interior of $C\times\sigma$.
\end{enumerate}
\end{defin}

To define cubes of multimorphisms in $\mathrm{Floer}$, we require one additional choice, as well as some notation:  we fix a positive number $\epsilon_0 < 1/2$. For each $n$ denote by $F_n$ the set of faces of the $n$-cube $[0,1]^n$ (we include the top stratum among them, so that  $F_n$ has $3^n$ elements). For each $f\in F_n$, denote by $f^o$, the interior of $f$.  For each $f\in F_n$, let $W_f$ be the image in the $n$-cube of the canonical affine embedding of 
\begin{equation} \label{eq:decomposition_W_f}
V_f\times U_f \;\cong\; f^o\times[0,\epsilon_0)^{\mathrm{codim}(f)}.
\end{equation}
Denote by  $\cV_n$ the open cover of $[0,1]^n$  given by $\cV_n=\{W_f\}_{f\in F_n}$. Note that there are natural identifications $F_n=(F_1)^n$ and $\cV_n=(\cV_1)^n$ where an $n$-tuple $(W_{f_1},\dots,W_{f_n})\in(\cV_1)^n$ is identified with the Cartesian product $(W_{f_1}\times \dots\times W_{f_n})=W_{f_1\times \dots\times f_n}$. Thus the cover $\cV_n$ is compatible with intersection with faces of the $n$-cube and is equivariant with respect to the action of the symmetric group $S_n$ by transposition of the coordinates of the cube. Finally, in the next definition, we exclude the case where the input and output agree, this will correspond to the formal units we discuss at the end of this section.

\begin{defin}\label{dfCuMultMor}
 An $n$-cube $\fd$ in $\mathrm{Floer}(\vec{\bfH},\bfH)$, whenever $\vec{\bfH}$ is not a singleton which agrees with $\bfH$,  consists of the choice $(T_f, b_f,g_f)$ of a local model for each face $f \in F_n$ of the cube, with domain the set $W_f$, equipped with the decomposition from Equation \eqref{eq:decomposition_W_f} so that the following property holds: given an inclusion $f_0\subset f_1$, let $E$ be the set of edges of $T_{f_0}$ that are collapsed under the map $T_{f_0}\to  T_{f_1}$. Let $g^E_{f_0}:W_{f_0}\to[0,1)^{E}$ be the composition of $g_{f_0}$ with the projection $[0,1)^{|E_{int}(T)|}\to [0,1)^{E}$. Then we have
\begin{equation}\label{eqGluingCompatibility}
b_{f_1}|_{W_{f_0}}=\Gamma_{g^E_{f_0}}(b_{f_0}).
\end{equation}
\end{defin}
The reader may want to  note that from this data we get, for each inclusion $f_1\subset f_2$ of faces, a surjective map $\sigma_{12}:V(T_{f_1})\to V(T_{f_2})$. If $v_1$ and $v_2$ are adjacent vertices of $T_{f_1}$, then whenever the corresponding gluing parameter does not vanish (so that the components are glued), we have $\sigma_{12}(v_1)=\sigma_{12}(v_2)$. Conversely,  if the gluing parameter of the edge connecting them vanishes, these vertices remain distinct in $T_{f_2}$.

An $n$-cube determines a map of sets
  \begin{equation}\label{eqUnderlyingMap}
        b:[0,1]^n\to\mathrm{Floer}_0(\vec{\bfH},\bfH),
      \end{equation}
      which is given by the formula
\begin{equation}\label{eqLocModMult}
b(p)=\Gamma_{g_f(x,y)}(b_f(x,y)). 
\end{equation}
whenever $p=(x,y)\in V_f\times U_f=W_f$. 

We verify that this expression is well-defined. It suffices to compare
the expressions associated to a pair of faces $f_0\subset f_1$. Fix
$p\in W_{f_0}\cap W_{f_1}$, and let $f_p$ be the unique face whose
relative interior contains $p$. Thus
\[
f_0\subset f_1\subset f_p.
\]
For $i<j$, denote by $E_{ij}$ the set of edges collapsed under the
map $T_{f_i}\to T_{f_j}$. Applying Equation~\eqref{eqGluingCompatibility} to these three
inclusions gives
\[
 b_{f_p}(p)
 =
 \Gamma_{g_{f_0}^{E_{0p}}(p)}(b_{f_0}(p))
 =
 \Gamma_{g_{f_1}^{E_{1p}}(p)}(b_{f_1}(p))
 =
 \Gamma_{\left(g_{f_0}^{E_{01}}(p),
                    g_{f_1}^{E_{1p}}(p)\right)}(b_{f_0}(p)),
\]
where the last equality uses the independence of the order of
gluing. The first and last expressions give gluing decompositions
with the same source $b_{f_0}(p)$. Lemma~\ref{lmGluingMonotonicity} therefore implies that
their gluing parameters agree. The components corresponding to
edges which survive in $T_{f_p}$ vanish at $p$. It follows that,
after identifying the surviving edges of $T_{f_0}$ and $T_{f_1}$,
the full gluing parameters agree, and hence
\[
 \Gamma_{g_{f_0}(p)}(b_{f_0}(p))
 =
 \Gamma_{g_{f_1}(p)}(b_{f_1}(p)).
\]
Thus Equation~\eqref{eqLocModMult} is independent of the chosen gluing chart.

\begin{rem}\label{remNonUniqueCube}
  If $b$ contains broken pre-stable discs, it is possible to construct distinct $n$-cubes whose underlying map is $b$. The simplest example is the following: we will set $n=1$, and fix a pair of Lagrangian branes $\bfL_0$ and $\bfL_1$. At the endpoint $\{0\}$ of the interval, say that we are given a triple of Hamiltonian data $(H_0,J_0)$, $(H_1, J_1)$ and  $(H_2, J_2)$, and a pair of continuation maps $(H^{-}_{s,t}, J^{-}_{s,t})$ and  $(H^{+}_{s,t}, J^{+}_{s,t})$ on the infinite strip $\Sigma = [0,1] \times \bR$, which for specificity are constant away from $[0,1] \times [-1,1]$. In our definition of Floer data, we require that the strips be equipped with strip-like ends that are disjoint from the region where the data is not constant; we assume that the images are $[0,1] \times (-\infty, -2]$ and $[0,1] \times [2,\infty)$.

  There is a standard way to extend this choice from the endpoint to the interior: set the data $b_{\{0\}}$ to be constant, and the gluing data $g_{\{0\}}$ to be the inclusion  $[0,\epsilon_0) \subset [0,1)$. There are many different charts $(T_{0}, b'_{\{0\}} ,g'_{\{0\}} )$ for a gluing atlas that results in the same global map $b$; one way to analyse this is to associate to each continuation map two modular parameters: the modulus $r$ of the complement of the strip-like ends (given by a rectangle $[0,1] \times (0,r)$) and the modulus $\rho$ of its intersection with the largest possible negative strip-like end (a subrectangle $[0,1] \times (0,\rho)$).
  The datum of $ b'_{\{0\}}$ thus has four modular parameters, and the gluing parameter $g'_{\{0\}} $ provides one more. Requiring that this chart yield $b$ after gluing imposes three equations; one is the modular parameter of the largest region in which the data agrees with $(H_1, J_1)$, and the other two are as above (complements of the strip-like ends, and intersection with the largest possible negative strip-like ends). We thus expect that there are two free parameters that recover the same data after gluing.

For an explicit example, over a point $r \in [0,\epsilon_0)$, we let $b'_{\{0\}}$ consist of the same data as $b_{\{0\}}$, except that the positive strip-like end on the curve carrying  $(H^{-}_{s,t}, J^{-}_{s,t})$ has image $[0,1] \times [2 +r,\infty)$, and the negative strip-like end on the curve carrying  $(H^{+}_{s,t}, J^{+}_{s,t})$ has image $[0,1] \times (-\infty, 2 + r]$. Setting $g'_{\{0\}}= g_{\{0\}} $, we then obtain the same curve after gluing (with the same strip-like ends).

\begin{figure}[h]
  \centering
  \begin{tikzpicture}[scale=0.85]
    \definecolor{Hmcol}{RGB}{40,90,200}      
    \definecolor{Hpcol}{RGB}{170,40,150}     
    \definecolor{glucol}{RGB}{210,130,40}    

    \def\logr{1.2}        
    \def\rval{0.30}       
    \def\U{0.45}          
    \def\fL{1.0}          
    \def\halfsep{3.5}    

    \pgfmathsetmacro{\sH}{\U}                          
    \pgfmathsetmacro{\tw}{\U}                          
    \pgfmathsetmacro{\stripend}{\sH + \tw}             
    \pgfmathsetmacro{\glend}{\stripend + \logr}        
    \pgfmathsetmacro{\dashend}{\glend + \fL}           
    \pgfmathsetmacro{\fendL}{\stripend + \fL}          


    \begin{scope}[yshift=2.4cm]
      \begin{scope}[shift={(-\halfsep,0)}]
        \draw[line width=1.6pt] (-\fendL,0) -- (-\stripend,0);
        \draw[thin] (-\stripend,0) -- (-\sH,0);
        \draw[Hmcol, line width=2.8pt] (-\sH,0) -- (\sH,0);
        \draw[thin] (\sH,0) -- (\stripend,0);
        \draw[line width=1.6pt] (\stripend,0) -- (\glend,0);
        \draw[line width=1.6pt, dashed] (\glend,0) -- (\dashend,0);
        \foreach \x/\lab/\xs in {-\stripend/{$-2$}/-2pt, -\sH/{$-1$}/0pt, \sH/{$1$}/0pt, \stripend/{$2$}/0pt}
          \draw (\x,-0.1) -- (\x,0.1) node[above, font=\tiny, xshift=\xs] {\lab};
        \draw (\glend,-0.1) -- (\glend,0.1);
        \node[below, font=\tiny] at (\glend,-0.18) {$2{+}\log r$};
        \node[Hmcol, font=\small, above=4pt] at (0,0.45) {$H^{-}$};
      \end{scope}

      \begin{scope}[shift={(\halfsep,0)}]
        \draw[line width=1.6pt, dashed] (-\dashend,0) -- (-\glend,0);
        \draw[line width=1.6pt] (-\glend,0) -- (-\stripend,0);
        \draw[thin] (-\stripend,0) -- (-\sH,0);
        \draw[Hpcol, line width=2.8pt] (-\sH,0) -- (\sH,0);
        \draw[thin] (\sH,0) -- (\stripend,0);
        \draw[line width=1.6pt] (\stripend,0) -- (\fendL,0);
        \foreach \x/\lab/\xs in {-\stripend/{$-2$}/-2pt, -\sH/{$-1$}/0pt, \sH/{$1$}/0pt, \stripend/{$2$}/0pt}
          \draw (\x,-0.1) -- (\x,0.1) node[above, font=\tiny, xshift=\xs] {\lab};
        \draw (-\glend,-0.1) -- (-\glend,0.1);
        \node[below, font=\tiny] at (-\glend,-0.18) {${-}2{-}\log r$};
        \node[Hpcol, font=\small, above=4pt] at (0,0.45) {$H^{+}$};
      \end{scope}

    \end{scope}

    \begin{scope}[yshift=0cm]
      \pgfmathsetmacro{\nh}{\logr}                                
      \pgfmathsetmacro{\suppI}{\logr + \tw}                        
      \pgfmathsetmacro{\suppO}{\logr + \tw + 2*\sH}                 
      \pgfmathsetmacro{\transO}{\logr + 2*\tw + 2*\sH}              
      \pgfmathsetmacro{\fendO}{\logr + 2*\tw + 2*\sH + \fL}         

      \draw[line width=1.6pt] (-\fendO,0) -- (-\transO,0);
      \draw[thin] (-\transO,0) -- (-\suppO,0);
      \draw[Hmcol, line width=2.8pt] (-\suppO,0) -- (-\suppI,0);
      \draw[thin] (-\suppI,0) -- (-\nh,0);
      \draw[line width=1.6pt] (-\nh,0) -- (\nh,0);
      \draw[thin] (\nh,0) -- (\suppI,0);
      \draw[Hpcol, line width=2.8pt] (\suppI,0) -- (\suppO,0);
      \draw[thin] (\suppO,0) -- (\transO,0);
      \draw[line width=1.6pt] (\transO,0) -- (\fendO,0);

      \draw (-\nh,-0.1) -- (-\nh,0.1) node[below, font=\tiny] {$-\log r$};
      \draw (0,-0.1) -- (0,0.1) node[above, font=\tiny] {$0$};
      \draw (\nh,-0.1) -- (\nh,0.1) node[below, font=\tiny] {$\log r$};
      \pgfmathsetmacro{\suppCL}{-(\suppI + \suppO)/2}
      \pgfmathsetmacro{\suppCR}{(\suppI + \suppO)/2}
      \node[Hmcol, font=\small, above=4pt] at (\suppCL,0.45) {$H^{-}$};
      \node[Hpcol, font=\small, above=4pt] at (\suppCR,0.45) {$H^{+}$};
      \draw[glucol, decorate, decoration={brace, amplitude=4pt}]
        (-\nh,0.85) -- (\nh,0.85)
        node[midway, glucol, font=\tiny, above=4pt] {neck of length $2\log r$};
    \end{scope}

    \begin{scope}[yshift=-2.4cm]
      \begin{scope}[shift={(-\halfsep,0)}]
        \pgfmathsetmacro{\stripendC}{\stripend + \rval}
        \pgfmathsetmacro{\glendC}{\stripendC + \logr}
        \pgfmathsetmacro{\dashendC}{\glendC + \fL}

        \draw[line width=1.6pt] (-\fendL,0) -- (-\stripend,0);
        \draw[thin] (-\stripend,0) -- (-\sH,0);
        \draw[Hmcol, line width=2.8pt] (-\sH,0) -- (\sH,0);
        \draw[thin] (\sH,0) -- (\stripendC,0);                   
        \draw[line width=1.6pt] (\stripendC,0) -- (\glendC,0);
        \draw[line width=1.6pt, dashed] (\glendC,0) -- (\dashendC,0);
        \foreach \x/\lab/\xs in {-\stripend/{$-2$}/-2pt, -\sH/{$-1$}/0pt, \sH/{$1$}/0pt, \stripend/{$2$}/0pt}
          \draw (\x,-0.1) -- (\x,0.1) node[above, font=\tiny, xshift=\xs] {\lab};
        \draw (\stripendC,-0.1) -- (\stripendC,0.1);
        \node[below, font=\tiny] at (\stripendC,-0.18) {$2{+}r$};
        \draw (\glendC,-0.1) -- (\glendC,0.1);
        \node[below, font=\tiny] at (\glendC,-0.40) {$2{+}r{+}\log r$};
        \node[Hmcol, font=\small, above=4pt] at (0,0.45) {$H^{-}$};
      \end{scope}

      \begin{scope}[shift={(\halfsep,0)}]
        \pgfmathsetmacro{\sleEnd}{-\stripend + \rval}            
        \pgfmathsetmacro{\glStart}{\sleEnd - \logr}               
        \pgfmathsetmacro{\dashStart}{\glStart - \fL}              

        \draw[line width=1.6pt, dashed] (\dashStart,0) -- (\glStart,0);
        \draw[line width=1.6pt] (\glStart,0) -- (\sleEnd,0);
        \draw[thin] (\sleEnd,0) -- (-\sH,0);                      
        \draw[Hpcol, line width=2.8pt] (-\sH,0) -- (\sH,0);
        \draw[thin] (\sH,0) -- (\stripend,0);
        \draw[line width=1.6pt] (\stripend,0) -- (\fendL,0);
        \foreach \x/\lab/\xs in {-\stripend/{$-2$}/-2pt, -\sH/{$-1$}/0pt, \sH/{$1$}/0pt, \stripend/{$2$}/0pt}
          \draw (\x,-0.1) -- (\x,0.1) node[above, font=\tiny, xshift=\xs] {\lab};
        \draw (\sleEnd,-0.1) -- (\sleEnd,0.1);
        \node[below, font=\tiny] at (\sleEnd,-0.18) {${-}2{+}r$};
        \draw (\glStart,-0.1) -- (\glStart,0.1);
        \node[below, font=\tiny] at (\glStart,-0.40) {${-}2{+}r{-}\log r$};
        \node[Hpcol, font=\small, above=4pt] at (0,0.45) {$H^{+}$};
      \end{scope}

    \end{scope}
  \end{tikzpicture}
  \caption{Two distinct charts $b$ and $b'$ giving the same glued curve.  Each continuation strip $[0,1]\times\bR$ is shown as a horizontal line; the support $[-1,1]$ of $H^{-}$ is drawn in blue and that of $H^{+}$ in magenta, while the strip-like ends are thick black.  Gluing removes the two dashed regions from the infinite ends, so that the gluing region $[2,2+\log r]$ on the positive end of $H^{-}$ and the gluing region $[-2-\log r,-2]$ on the negative end of $H^{+}$ are identified along their boundaries to form a neck of length exactly $2\log r$ and centered at the coordinate $0$.  In the alternative chart $b'$, the boundary of the positive strip-like end of $H^{-}$ is shifted to $2+r$ and that of the negative strip-like end of $H^{+}$ is shifted to $-2+r$; the gluing function is unchanged, so the resulting glued curve is the same.}
  \label{fig:gluing-non-uniqueness}
\end{figure}
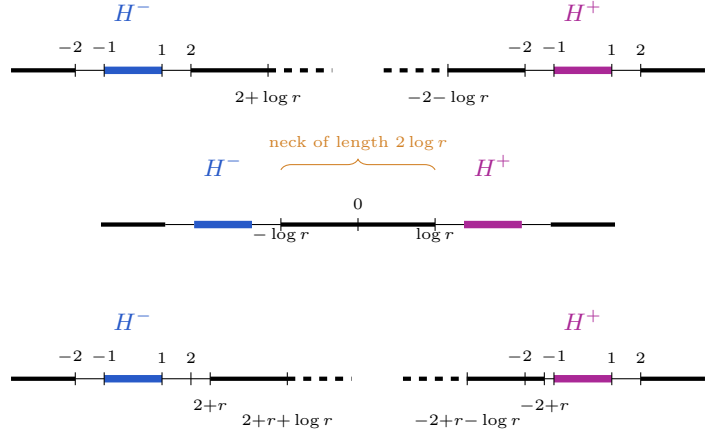
\end{rem}

We ommit the proof of the next result which is exactly the same as that of \cite[Lemma 2.13]{Abouzaid2023}:
\begin{lem}
  The collections of sets $\{\mathrm{Floer}_n(\vec{\bfH},\bfH)\}$ are the underlying sets of $n$-cubes of a symmetric cubical set. \qed
\end{lem}

The next definition is the generalization of Definition \ref{def:mult-dimens-0} to higher dimensional cubes. In its statement, we use the fact that $F_{n_1 +  n_2} = F_{n_1} \times F_{n_2} $, i.e. that every face of an $(n_1 + n_2)$-cube consists of a product of faces of the $n_1$-cube and an $n_2$-cube. 
\begin{defin}
For each integer $i$ between $1$ and the length of a sequence $\vec{\bfH}_2$ of Hamiltonians, define the multicomposition map 
\begin{align}\label{eqHMultComp}
  \circ_i:\mathrm{Floer}_{n_1}(\vec{\bfH}_2;\bfH_{1,i})\times \mathrm{Floer}_{n_2}(\vec{\bfH}_1;\bfH) & \to\mathrm{Floer}_{n_1+n_2}(  \vec{\bfH}_1 \circ_i \vec{\bfH}_2  ;\bfH) \\
  (\fd_2,\fd_1) & \mapsto \fd_1 \circ_i \fd_2
\end{align}
as follows: writing $(T_{f_1}, b_{f_1},g_{f_1})$ to be the data associated by $\fd_1$ to the faces of the $n_1$-cube, and similarly for $\fd_2$, we set the data which determines $\fd_1 \circ_i \fd_2 $ to be:
\begin{itemize}
\item The tree $T_{f_1 \times f_2}$ is the result of attaching the root of $T_{f_2}$ to the $i$th leaf of $T_{f_1}$. Thus, the (interior) edges of $T_{f_1 \times f_2} $ are the union of the edges $E_1$ and $E_2$ of $T_{f_1}$ and $T_{f_2}$ with an additional edge arising from identifying the root of $T_{f_2}$ with the $i$th leaf of $T_{f_1}$.
  \item The (family) $b_{f_1 \times f_2}$  of broken curves are given by the composition
    \begin{equation}
      \label{eq:composition_Floer_data_family}
      \begin{tikzcd}[column sep=large]
        W_{f_1^\circ \times f_2^\circ} \ar[r, "\sim"] & W_{f_2^\circ} \times W_{f_1^\circ} \ar[d, "b_{f_2} \times b_{f_1}"] \\
        \mathrm{Floer}_{0,T_1 \circ_i T_2}(\vec{\bfH}_1 \circ_i \vec{\bfH}_2 ,\bfH) & \mathrm{Floer}_{0,T_2}(\vec{\bfH}_2,\bfH_{1,i}) \times \mathrm{Floer}_{0,T_1}(\vec{\bfH}_1,\bfH) \ar[l, "\circ_i"']
      \end{tikzcd}
    \end{equation}
where the composition in the bottom row is given as in Definition \ref{def:mult-dimens-0}.
\item The gluing map $g_{f_1 \times f_2}$ vanishes on the edge of $T_{f_1 \times f_2}$ along which gluing takes place, and it otherwise assigns to each edge coming from $T_i$ the value of $g_{f_i}$.
\end{itemize}
\end{defin}

It is straightforward to check that  boundaries and degeneracies are defined in the same way on the two sides in terms of inclusions and projections of cubes. We conclude:
\begin{lem} 
  The map given in Equation \eqref{eqHMultComp} determines a map of  symmetric cubical sets:
  \begin{equation} \label{eq:H_multicomposition_cubical_sets}
       \circ_i:\mathrm{Floer}_{\bullet}(\vec{\bfH}_2;\bfH_{1,i})\otimes \mathrm{Floer}_{\bullet}(\vec{\bfH}_1;\bfH)\to\mathrm{Floer}_{\bullet}(  \vec{\bfH}_1 \circ_i \vec{\bfH}_2  ;\bfH) .
  \end{equation} \qed
\end{lem}

\begin{rem}
   We warn the reader that the set of $n$-cubes of the tensor product $\mathrm{Floer}_{\bullet}(\vec{\bfH}_2;\bfH_{1,i})\otimes \mathrm{Floer}_{\bullet}(\vec{\bfH}_1;\bfH) $ is not the naive one given by the union over $n_1 + n_2 = n$ of the left hand sides of Equation \eqref{eqHMultComp}, i.e. by pairs $(\fd_2,\fd_1)$ of $n_1$ and $n_2$ cubes; we must allow, in addition, all possible shufflings of such product cubes.  See \cite[Remark 2.16]{Abouzaid2023} for more precise details how Equation \eqref{eqHMultComp} induces the asserted map.
\end{rem}

At this stage, one can easily check that the compositions we have just defined satisfy the associativity relations of a multicategory. See \cite[Definition B.8]{Abouzaid2023}.

To complete the construction of the multicategory $\mathrm{Floer}$, recall that we imposed, in Definition \ref{def:pre-multimorphism-H}, the monotonicity condition which implies that the input and output of each multimorphism with domain a strip must be different. This implies that our construction so far has the property that morphisms from a Floer datum to itself would be empty, had we not excluded this case from Definition~\ref{dfCuMultMor}. The next definition includes a description of this special case, which ensures that the multicategory we obtain is unital:

\begin{defin}\label{def-Floer-no-reg}
The multicategory $\mathrm{Floer}$ is the following multicategory enriched in symmetric cubical sets.  
\begin{itemize}
\item The objects are boundary conditions $\bfH = (\bfL_0, (J_{01},H_{01}), \bfL_1)$ as in Section \ref{sec:regular-floer-data-1}.
\item The multimorphisms are given by the symmetric cubical sets $\mathrm{Floer}_\bullet(\vec{\bfH},\bfH)$ whenever $\vec{\bfH} $ is not a singleton that agrees with $\bfH$, and defining $\mathrm{Floer}_\bullet(\bfH,\bfH) $ to be a singleton.
    \item The multicompositions are given by Equation \eqref{eq:H_multicomposition_cubical_sets}, and the singleton in $ \mathrm{Floer}_\bullet(\bfH,\bfH)$ acts as a strict unit. 
    \end{itemize}
\end{defin}

The above definition makes sense because of the strict monotonicity condition: before adding the unit, all possible compositions which could have output in $ \mathrm{Floer}_\bullet(\bfH,\bfH) $ have empty domain. After adding the unit, the only such composition involves composing the unit with itself.

\subsection{Multimorphisms and the forgetful map}
Denote by $\overline{\cR}_{m+1}$ the moduli space of stable discs with $m+1$ boundary marked points, see Section \ref{sec:stab-moduli-discs} for details. Denote by $\cR_{m+1}$ its interior, consisting of smooth discs. Let $\mathbf{\Box}_\bullet(\overline{\cR}_{m+1})$ denote the set of cubes in $\overline{\cR}_{m+1}$, considered as a symmetric cubical set. Given any composable sequence of boundary conditions $\vec{\bfH}=(\bfH_{01},\ldots,\bfH_{m-1,m})$ and output boundary condition $\bfH_{0m}$, there is a natural map
\begin{equation}\label{eq-forget-mult}
\fF:\mathrm{Floer}_\bullet(\vec{\bfH},\bfH_{0m})\to \mathbf{\Box}_\bullet(\overline{\cR}_{m+1})
\end{equation}
which forgets the Floer data, the 1-form and the strip like ends, and collapses the unstable components of the underlying disc. 


The map $\fF$ is not in general a homotopy equivalence. The following definition identifies a condition which will ensure it:
\begin{defin}\label{defHamOrdering}
Let $\vec{\bfH}=(\bfH_{01},\dots,\bfH_{m-1,m})$ be a composable sequence of boundary conditions and let $\bfH_{0m}$ be a horizontal composition of these boundary conditions. We say that $\bfH_{0m}>\vec{\bfH}$ if for every $x \in M$, and every $1\leq i\leq m$, we have
\begin{equation}
\min_{t\in[0,1]} H_{0m}(t,x)
>
\max\left\{
2^{m-1}\max_{t\in[0,1]} H_{i-1,i}(t,x),
\max_{t\in[0,1]} H_{i-1,i}(t,x)
\right\}.
\end{equation}
\end{defin}
For the next result, we consider a punctured disc $\Sigma$ with $m$ inputs and $1$ output:
\begin{lemma}\label{lmExMonH}
If $\bfH_{0m}>\vec{\bfH}$, then for any closed $1$-form $\alpha$ on $\Sigma$ vanishing on $\partial\Sigma$ which agrees with $w_idt$ near the $i$th input and with $dt$ near the output and such that $w_i\geq 2^{1-m}$, there exists a function
\begin{equation}
 H_\Sigma:\Sigma\times M\to\bR 
\end{equation}
so that $H_{\Sigma}\alpha$ pulls back to $H_{i-1,i}dt$ with respect to the strip-like end for the $i$th input ($1\leq i\leq m$), and $H_{\Sigma}\alpha$ pulls back to $H_{0m}dt$ with respect to the strip-like end near the output, and so that $H_{\Sigma}$ satisfies the monotonicity inequality \eqref{eq:strict-monotonicity}.
\end{lemma}

\begin{proof}
  Extend each negative strip-like end to a smooth embedding of $ (-\infty,\delta] \times [0,1]$ for a small positive constant $\delta$, and the positive strip-like end to an embedding of $[-\delta, \infty) \times [0,1]$, so that the images remain disjoint, and so that the pullback of $\alpha$ agrees with $w_i dt$. We call the images of these larger domains, the \emph{extended strip-like ends}.

  Since $\alpha$ is closed and vanishes on $\partial\Sigma$, Stokes'
theorem gives $\sum_{i=1}^{m} w_i=1$, and hence
$1\le w_i^{-1}\le 2^{m-1}$. Let $H: M\to\bR$ be a smooth function so that for any $x\in M$, and for all $1\leq i\leq m$, we have
\begin{equation}
\begin{aligned}
\min_{t\in[0,1]} H_{0m}(t,x)
&> H(x) \\
&> \max\left\{
2^{m-1}\max_{t\in[0,1]} H_{i-1,i}(t,x),
\max_{t\in[0,1]} H_{i-1,i}(t,x)
\right\} \\
&\ge \max_{t\in[0,1]} \frac{H_{i-1,i}(t,x)}{w_i}.
\end{aligned}
\end{equation}
We define the function $H_{\Sigma}$ piecewise as follows:
\begin{enumerate}
\item In the image of the strip-like ends, $H_{\Sigma}$ agrees with $\frac{H_{i-1,i}}{w_i}$ at the $i$th input ($1\leq i\leq m$) and with $H_{0m}$ at the output.
\item Away from the images of the extended strip-like ends, $H_{\Sigma}\equiv H$.
\item In the strips $[0,\delta] \times [0,1]$ (or $[-\delta, 0] \times [0,1]$), we define $H_{\Sigma}$ by monotonically interpolating, along the $s$ direction, between $\frac{H_{i-1,i}}{w_i}$ and $H$ at the $i$th input ($1\leq i\leq m$), and between $\frac{H_{0m}}{w_0}$ and $H$ at the output, where $w_0=1$.
\end{enumerate}
See Figure~\ref{fig:HSigma-construction}.

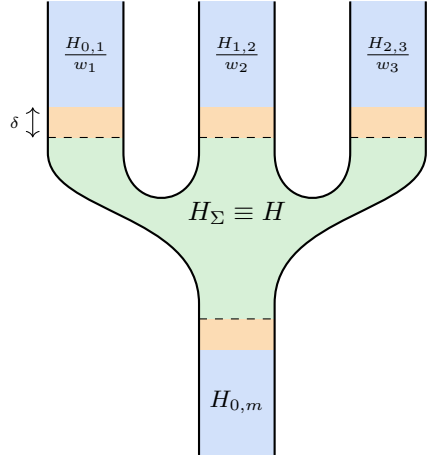
\begin{figure}[h]
  \centering
  \begin{tikzpicture}[scale=1.0]
    \definecolor{stripcol}{RGB}{210,225,250}
    \definecolor{trancol}{RGB}{252,220,180}
    \definecolor{intcol}{RGB}{215,240,215}

    \def\stripshift{0.6}
    \def\semishift{0.2}
    \def\deltaval{0.4}

    \pgfmathsetmacro{\inlo}{2 + \stripshift}        
    \pgfmathsetmacro{\intop}{\inlo - \deltaval}     
    \pgfmathsetmacro{\outhi}{-\stripshift}          
    \pgfmathsetmacro{\outbot}{\outhi + \deltaval}   
    \pgfmathsetmacro{\dipouter}{1 + \semishift}     
    \pgfmathsetmacro{\dipinner}{1 + \semishift}     

    \fill[intcol]
      (-0.5, 0)
        .. controls (-0.5, \dipouter) and (-2.5, \dipouter) .. (-2.5, 2)
        -- (-2.5, 4) -- (-1.5, 4) -- (-1.5, 2)
        .. controls (-1.5, \dipinner) and (-0.5, \dipinner) .. (-0.5, 2)
        -- (-0.5, 4) -- (0.5, 4) -- (0.5, 2)
        .. controls (0.5, \dipinner) and (1.5, \dipinner) .. (1.5, 2)
        -- (1.5, 4) -- (2.5, 4) -- (2.5, 2)
        .. controls (2.5, \dipouter) and (0.5, \dipouter) .. (0.5, 0)
        -- (0.5, -2) -- (-0.5, -2) -- (-0.5, 0)
      -- cycle;

    \fill[stripcol] (-2.5, \inlo) rectangle (-1.5, 4);
    \fill[stripcol] (-0.5, \inlo) rectangle (0.5, 4);
    \fill[stripcol] (1.5, \inlo) rectangle (2.5, 4);
    \fill[stripcol] (-0.5, -2) rectangle (0.5, \outhi);

    \fill[trancol] (-2.5, \intop) rectangle (-1.5, \inlo);
    \fill[trancol] (-0.5, \intop) rectangle (0.5, \inlo);
    \fill[trancol] (1.5, \intop) rectangle (2.5, \inlo);
    \fill[trancol] (-0.5, \outhi) rectangle (0.5, \outbot);

    \draw[thick] (-2.5, 4) -- (-2.5, 2)
        .. controls (-2.5, \dipouter) and (-0.5, \dipouter) .. (-0.5, 0)
        -- (-0.5, -2);
    \draw[thick] (0.5, -2) -- (0.5, 0)
        .. controls (0.5, \dipouter) and (2.5, \dipouter) .. (2.5, 2)
        -- (2.5, 4);
    \draw[thick] (-1.5, 4) -- (-1.5, 2)
        .. controls (-1.5, \dipinner) and (-0.5, \dipinner) .. (-0.5, 2)
        -- (-0.5, 4);
    \draw[thick] (0.5, 4) -- (0.5, 2)
        .. controls (0.5, \dipinner) and (1.5, \dipinner) .. (1.5, 2)
        -- (1.5, 4);

    \draw[dashed] (-2.5, \intop) -- (-1.5, \intop);
    \draw[dashed] (-0.5, \intop) -- (0.5, \intop);
    \draw[dashed] (1.5, \intop) -- (2.5, \intop);
    \draw[dashed] (-0.5, \outbot) -- (0.5, \outbot);

    \node[font=\small] at (-2, 3.3) {$\frac{H_{0,1}}{w_1}$};
    \node[font=\small] at (0, 3.3) {$\frac{H_{1,2}}{w_2}$};
    \node[font=\small] at (2, 3.3) {$\frac{H_{2,3}}{w_3}$};
    \node[font=\small] at (0, -1.3) {$H_{0,m}$};

    \node at (0, \dipouter) {$H_\Sigma\equiv H$};

    \draw[<->, thin] (-2.7, \intop) -- (-2.7, \inlo);
    \node[font=\tiny, anchor=east] at (-2.75, {(\intop+\inlo)/2}) {$\delta$};
  \end{tikzpicture}
  \caption{The disc $\Sigma$ with $m=3$ inputs and one output, illustrating the piecewise definition of $H_\Sigma$ in the proof of Lemma~\ref{lmExMonH}. On the strip-like ends (blue, item~(1)), $H_\Sigma$ agrees with $H_{i-1,i}/w_i$ at the $i$th input ($1\leq i\leq m$) and with $H_{0,m}$ at the output. On the transition zones (orange, item~(3))---the depth-$\delta$ extensions of the strip-like ends into the disc body, parametrized by $[0,\delta]\times[0,1]$ at each input and $[-\delta,0]\times[0,1]$ at the output---$H_\Sigma$ interpolates monotonically along the $s$-direction. On the rest of the disc body (green, item~(2)), $H_\Sigma\equiv H$.}
  \label{fig:HSigma-construction}
\end{figure}

Note that the following condition holds for each $x\in M$:
\begin{equation} \label{eq:monotonicity_HSigma}
dH_{\Sigma}(x)\wedge\alpha\geq 0,
\end{equation} with strict inequality at some points in the strips from (3).
This implies \eqref{eq:strict-monotonicity}.

\end{proof}
The following is the main result of this section. 
\begin{prop}\label{lmFrgtHtpyEq}
If $\bfH_{0m}>\vec{\bfH}=(\bfH_{01},\dots,\bfH_{m-1,m})$, then the map
\begin{equation}
\fF:\mathrm{Floer}_{\bullet}(\vec{\bfH},\bfH_{0m})\to \mathbf{\Box}_\bullet(\overline{\cR}_{m+1})
\end{equation}
is a homotopy equivalence of symmetric cubical sets.
\end{prop}
The proof will be given at the end of this section, after a few preparatory lemmas. Denote by $\mathrm{Floer}^{sm}_{0}(\vec{\bfH},\bfH_{0m})$ the topological space of $0$-dimensional multimorphisms whose underlying disc is smooth. 
\begin{lemma}\label{lmForSmHoEq}
The map
\begin{equation}
\fF:\mathrm{Floer}^{sm}_{0}(\vec{\bfH},\bfH_{0m})\to \cR_{m+1}
\end{equation}
is a homotopy equivalence of topological spaces.
\end{lemma}

\begin{proof} 

We can find a unique element of the form $({\Sigma},e^rH_\Sigma,e^{-r}\alpha,\{\epsilon^{\pm}_p\})$ in any equivalence class $[({\Sigma},H_\Sigma,J,\alpha,\{\epsilon^{\pm}_p\})]\in \mathrm{Floer}^{sm}_{0}(\vec{\bfH},\bfH_{0m})$ such that the output weight is $1$. Let us assume that the representatives that we use below satisfy this property.

The map $\fF$ factors as 
\begin{equation} \label{eq:factor-forgetful-map}
[({\Sigma},H_\Sigma, J,\alpha,\{\epsilon^{\pm}_p\})]\mapsto [({\Sigma},\alpha,\{\epsilon^{\pm}_p\})]\mapsto [({\Sigma},\{\epsilon^{\pm}_p\})]\mapsto [\Sigma].
\end{equation}
It is easy to see that all of these maps are fiber bundles.

The fibres of the first two maps are convex, so it suffices to show that they are non-empty: for the first map, this  follows from Lemma \ref{lmExMonH}, while for the middle map, this is a consequence of the de-Rham isomorphism for the relative cohomology of $(\Sigma,\partial\Sigma)$. We also used the fact that $w_i\geq 2^{1-m}$ is a convex condition and that there are solutions of $w_1+\ldots+w_m=1$ satisfying this property for all $1\leq i\leq m$, e.g $w_i=1/m$.

It remains to show that the  map forgetting the strip-like ends has contractible fibres. Arguing as in \cite[Lemma 2.21]{Abouzaid2023}, one shows that the space of strip-like ends at a given puncture deformation retracts to the space of holomorphic embeddings of a fixed neighborhood of the origin in the upper half plane into itself which preserve the origin and have non-vanishing derivative at the origin. We can then define a contracting flow on these embeddings whose fixed points are the positive real scalings. This finishes the proof.

\end{proof}

For the next result, we denote by $\Box^{sm}_{\bullet} M$ the smooth singular cubes of a differentiable manifold $M$.

\begin{lemma}\label{lmsmincHoEq}
 
  The inclusion   \begin{equation} \label{eq:include_smooth_gluable_cubes}
   \iota_*: \Box^{sm}_{\bullet}(\mathrm{Floer}^{sm}_{0}(\vec{\bfH},\bfH_{0m}) )  \hookrightarrow \mathrm{Floer}_{\bullet}(\vec{\bfH},\bfH_{0m}).
 \end{equation}
  is a homotopy equivalence of symmetric cubical sets.
\end{lemma}
\begin{proof}

The proof is the same as in the closed-string case
\cite[Lemma~2.22]{Abouzaid2023}. The almost complex structure, which
is included here as part of the Floer data although suppressed from
the notation, is simply carried along and does not affect the
argument. We spell out one compatibility point which is left implicit
in the cited proof.

Given an $n$-cube with gluing atlas
$(T_f,b_f,g_f)_{f\in F_n}$, the cited proof constructs a pointwise
homotopy by modifying the gluing parameters using a partition of
unity subordinate to the collar neighbourhoods $W_f$. At the initial
endpoint this recovers the given cube, while at the terminal endpoint
all nodes have been glued, yielding a smooth family. The construction
is constant on cubes which are already smooth.

The main point is to verify that this pointwise homotopy is itself an
$(n+1)$-cube. In the notation of the cited proof,
$\widetilde g_{f,e}(t,x)$ is the modified gluing parameter associated
to an edge $e$ in the chart indexed by $f$, obtained from
$g_{f,e}(x)$ by the partition-of-unity interpolation of
\cite[Equation~(2.40)]{Abouzaid2023}. The local gluing parameters near
$\{0\}\times f$ are then defined by
\[
   \widehat g_{f,e}(t,x)
   =
   \widetilde g_{f(x),e}(t,x),
\]
where $f(x)$ is the smallest face whose collar contains $x$; the
corresponding broken data are denoted by $\widehat b_f$.

When $f(x)$ changes, the same surviving edge may be described using
the gluing coordinates of two different charts. The well-definedness
argument following Equation~\eqref{eqLocModMult}, which uses
Lemma~\ref{lmGluingMonotonicity}, shows that these coordinates agree.
The terms associated to edges whose collapse pattern changes are
multiplied by partition-of-unity functions which vanish near the
corresponding transition. Together with the compatibility of the
broken data in Equation~\eqref{eqGluingCompatibility}, this shows that $\widehat g_f$ and
$\widehat b_f$ are smooth and satisfy the compatibility conditions
for a gluing atlas.

The remaining verifications that the construction is compatible with
faces, degeneracies and symmetries, and that it satisfies the endpoint
conditions for a deformation retraction, are unchanged from the
cited proof.

\end{proof}

We now provide the proof of the main result in this section:
\begin{proof}[Proof of Proposition \ref{lmFrgtHtpyEq}]
Consider the diagram
\begin{equation}
  \begin{tikzcd}
    \Box^{}_{\bullet}(\mathrm{Floer}^{sm}_{0}(\vec{\bfH},\bfH_{0m}))\ar[d] &\Box^{sm}_{\bullet}(\mathrm{Floer}^{sm}_{0}(\vec{\bfH},\bfH_{0m}))\ar[d]\ar[r] \ar[l] &\mathrm{Floer}_{\bullet}(\vec{\bfH},\bfH_{0m})\ar[d,"\fF"]\\
        \Box^{}_{\bullet}(\cR_{m+1}) &\Box^{sm}_{\bullet}(\cR_{m+1})\ar[l] \ar[r]& \mathbf{\Box}_\bullet(\overline{\cR}_{m+1})
  \end{tikzcd}
\end{equation}
That $\fF$ is a homotopy equivalence will follow when we show that all other arrows in the diagram are homotopy equivalences. Indeed, the upper right horizontal arrow is a homotopy equivalence by Lemma \ref{lmsmincHoEq}. The left vertical one is a homotopy equivalence by Lemma \ref{lmForSmHoEq}, and the fact that a homotopy equivalence of topological spaces induces a homotopy of the associated symmetric cubical sets. The smooth approximation argument for families of functions and $1$-forms parametrized by a cube then implies that the two horizontal maps on the left are homotopy equivalences, hence so is the middle vertical map (the use of such approximation arguments to show that the inclusion of smooth chains into continuous chains is a homotopy equivalence goes back to Eilenberg \cite{Eilenberg1947}).

The bottom right horizontal map factors as 
\begin{equation}
\Box^{sm}_{\bullet}(\cR_{m+1})\to\Box_{\bullet}(\cR_{m+1})\to \mathbf{\Box}_\bullet(\overline{\cR}_{m+1}).
\end{equation}
The map on the right of the last equation is induced from the inclusion of the interior of $\overline{\cR}_{m+1}$. Since $\overline{\cR}_{m+1}$ is a manifold with corners, this inclusion is a homotopy equivalence. The map on the left is a homotopy equivalence, again by the cubical analogue of \cite{Eilenberg1947}.

\end{proof}

We now put all these forgetful maps together into a forgetful multifunctor. We start by defining a multicategory, enriched in symmetric cubical sets, associated to a $\mathbb{Z}_2$- or $\mathbb{Z}$-graded symplectic manifold $M$ and a bulk $\textbf{b}$:

\begin{defin} \label{def:mult-forg-map}
The objects of $\mathrm{Lag}$ are pairs of Lagrangian branes $(\bfL_0, \bfL_1)$. The multimorphism cubical sets are empty except possibly in the following situation: we are given a sequence  $(\bfL_i, \bfL_{i+1})$ of input objects, $i=0,\ldots,k-1$ with $k\geq 1$, and the output object is $(\bfL_0, \bfL_k)$. In that case, we define the multimorphism space to be $\mathbf{\Box}_\bullet(\overline{\cR}_{k+1})$. The multi-composition structure is inherited from the operad structure of $\{\mathbf{\Box}_\bullet(\overline{\cR}_{k+1})\}_{k=1,\ldots}$.  
\end{defin}

We note that the above definition makes sense because we set $\overline{\cR}_{2}$ to be a point, despite the fact that there is no stable disc with two marked points.   In this way, there is a natural forgetful functor
\begin{equation}
\fF:\mathrm{Floer}\to \mathrm{Lag},
\end{equation}which sends each boundary condition to its underlying pair of Lagrangian branes on objects and uses Equation \eqref{eq-forget-mult} on multimorphism spaces. In Section~\ref{sec:def-fuk-with-support} we shall use this to define Fukaya categories.


\section{The Floer functor}

\subsection{Moduli spaces}
\label{sec:moduli_spaces}

Associated to each cube $\fd$ in $ \mathrm{Floer}_{n}(\vec{\bfH},\bfH_{0k})$ is a topological space
\begin{equation}
  \label{eq:moduli_space}
  \Mbar(\fd)
\end{equation}
whose elements are equivalence classes of stable solutions to the Floer equation with datum $b_{\fd}(x)$ associated by $\fd$ to a point $x \in [0,1]^n$. Concretely, if $x$ lies in a neighbourhood of a stratum $f$, then $b_{\fd}(x)$, which is valued in $\mathrm{Floer}_{0}(\vec{\bfH},\bfH_{0k}) $, is obtained by gluing the data $b_f$ according to the gluing parameters associated to this point by $x$. Part of the datum of an element of this set thus consists of a pre-stable disc $\Sigma_x$ with $k+1$ boundary marked points. The domains of the solutions that we consider are pre-stable discs $\Sigma$ with $k+1$ boundary marked points equipped with an embedding
\begin{equation}
\Sigma_x \to \Sigma
\end{equation}
whose complement is unstable, and which is compatible with the labelling of the marked points in the sense that the marked points labelled by $i$ in $\Sigma_x$ and $\Sigma$ are connected by chains of components in the complement of $\Sigma_x$.  We partition the components of $\Sigma \setminus \Sigma_x$ into \emph{Floer strips}, which separate the marked point labelled by $i$ in $\Sigma_x$ from $\Sigma$, or the two sides of a node in $\Sigma_x$, and \emph{bubbles} which may be either discs or spheres. We thus consider maps
\begin{equation}
  u \co \Sigma \to M
\end{equation}
satisfying the inhomogeneous Cauchy-Riemann equation
\begin{equation}
  \left(du - X_{H_\Sigma} \otimes \alpha_{\Sigma}\right)^{0,1} = 0
\end{equation}
where the function $H_{\Sigma}$, the $1$-form $\alpha_\Sigma$, and the almost complex structure are determined as follows:
\begin{itemize}
\item On a component of $\Sigma_x$, we use the data determined by $\fd$
\item On a Floer strip, we use the data determined by the associated node of $\Sigma_x$ (i.e. with respect to the identification with $\bR \times [0,1]$ which is canonical up to translation, we set $H_{\Sigma} = H_e$ for the edge $e$ labelling this node, the $1$-form $\alpha_\Sigma$ to be given by $dt$, and the almost complex structure by $J_e$).
  \item On a bubble, the function $H_{\Sigma}$ identically vanishes (as does $\alpha_\Sigma$), and the almost complex structure is given by the almost complex structure on $M$ associated by $\fd$ to the point in $\Sigma_x$ to which this bubble is attached in the case of a sphere bubble, or by the almost complex structure $J_{L_i}$ for a disc bubble attached to the $i$-th componenent of the boundary of $\Sigma_x$.
  \end{itemize}
  We associate to each solution of this equation the topological energy (c.f. \cite[Lemma 8.1.6]{McDuff2012})
\begin{align}
  E(u) & = \int u^* \omega + \tilde{u}^* dH \wedge \alpha \\
  & = \int |du - X_H \otimes \alpha|^2 d \mathrm{vol}_{\Sigma} + d H_{u(z)} \wedge \alpha.
\end{align}
  
The moduli space $  \Mbar(\fd) $ is then defined to be the space of such finite energy maps, with the property that the asymptotic conditions on the two sides of each node agree, modulo the equivalence relation which identifies two maps that are intertwined by a biholomorphism over $M$ preserving the Floer data. We have the following standard consequence of Gromov compactness whose proof we omit:
\begin{lem}
 The energy functional is proper on $  \Mbar(\fd) $.\qed
\end{lem}

\subsection{Regular Floer data}
\label{sec:regular_data}
We have already discussed the notion of regularity for objects of $\mathrm{Floer}$ in Section \ref{sec:prel-sympl-manif} and made this a requirement in Definition \ref{def-Floer-no-reg}. 
We now extend this notion to multimorphisms in $\mathrm{Floer}$. The linearization of the Cauchy-Riemann operator satisfied by each point in $ \Mbar(\fd) $ is a Fredholm map
\begin{equation} \label{eq:linearisation_CR}
C^{\infty}(\Sigma; u^* TM)/ \mathfrak{aut} \, \Sigma \to \Omega^{0,1}(\Sigma; u^* TM),
\end{equation}
where $ \mathfrak{aut}\, \Sigma$ is the tangent space of the space of biholomorphisms of $\Sigma$ which preserve the Floer data on $\Sigma$ (i.e. preserve the equivalence class of this point in $ \Mbar(\fd) $). Assuming that the corresponding point in $[0,1]^n$ lies in the interior of a face $[0,1]^k$, the deformation of the operator associated to moving the point within this stratum defines a map
\begin{equation} \label{eq:normal_deformation_CR}
 T [0,1]^k \to \Omega^{0,1}(\Sigma; u^* TM).
\end{equation}
We call the sum of the operators in Equation \ref{eq:linearisation_CR} and \ref{eq:normal_deformation_CR} the \emph{extended linearization} operator.

We associate to each cube $\fd$ a pair of evaluation maps
\begin{equation}
\Mbar_{1,0}(\fd) \to   M  \leftarrow \Mbar_{1}(J_{\fd})
\end{equation}
where the left hand side is the moduli space of solutions to Floer's equation, with $1$ interior marked point, and the right hand side is the moduli space of pseudo-holomorphic spheres of Chern number $0$, parametrized by the universal curve over $[0,1]^n$ underlying $\fd$, also with one marked point. 


\begin{defin}
  The family of Floer data $\fd$ is \emph{regular} if the following properties hold:
  \begin{enumerate}
  \item  the extended linearization operator is surjective for all elements of $\Mbar_1(J_{\fd})$ which are simple pseudo-holomorphic spheres of Chern class $0$,
  \item  the extended linearization operator is surjective for all elements of $\Mbar_{1,0}(\fd) $, whose underlying curve does not contain any Floer strip, disk bubble or sphere bubble, and for which this index is strictly smaller than $2$, and
    \item the evaluation maps from these two parametrized spaces to the product of $M$ with the universal curve over $[0,1]^n$ are transverse (and hence disjoint).
    
  \end{enumerate}
\end{defin}

\begin{lem}
  If the family of Floer data $\fd$ are regular, then the extended linearization operator is surjective for each element of $ \Mbar(\fd) $ whose virtual dimension is strictly less than $2$. \qed
\end{lem}

\begin{proof}
  Our conditions for regularity imply that the solutions to continuation equations of virtual dimension strictly smaller than $2$ are given as follows: those of virtual dimension $0$ have domain a strip, and those of virtual dimension $1$ have domain either a strip or a concatenation of a pair of virtual dimension $0$ strips (one of which is a solution to Floer's equation). By Definition \ref{def:weakly-monotone_M}, every sphere component has non-negative Chern number. A positive-Chern sphere would force the principal component to have negative virtual dimension and is therefore excluded by regularity, while a Chern-zero sphere is excluded by the evaluation-transversality condition by the evaluation-transversality condition in the definition right above. Crucially, because of our minimum Maslov index $2$ assumption there are no disk bubbles in this virtual dimension range.

  The same index and transversality argument applies to general multimorphisms, with principal polygon components in place of strips and domain degenerations included among the boundary strata.
\end{proof}

We note that regularity is closed under taking boundaries, degeneracies, and symmetries, as well as multicompositions (products). This leads to the following:
\begin{defin}
  The multicategory of regular Floer data $\mathrm{Floer}^{\mathrm{reg}}$ is the sub-multicategory of $\mathrm{Floer}$ with the same objects, and whose multimorphisms are all multimorphisms given by regular Floer data.
\end{defin}

The following result shows that the homotopy type of the multimorphism spaces does not change when we restrict to regular Floer data.

\begin{lem}\label{lem-reg-correct}
  The inclusion of the multicategory of regular Floer data $\mathrm{Floer}^{\mathrm{reg}}$ into $\mathrm{Floer}$ induces a homotopy equivalence of multimorphism spaces.
\end{lem}
\begin{proof}
We construct a deformation retraction: for each cube $\fd$ in $ \mathrm{Floer}  $, we choose a regular cube $\fd'$ with the same underlying  data at the inputs and output, so that $\fd = \fd'$ if $\fd$ is regular, and a homotopy between them, compatibly with face and degeneracy maps. The construction is inductive on the dimension of the cube, and relies essentially on genericity of regular data, so that the data $b'_f$ underlying $\fd'$ can be chosen to be an arbitrarily small perturbation of the data $b_f$ underlying $\fd$, yielding a canonical path between them up to contractible choice, which provides the desired deformation retraction.
\end{proof}

\subsection{The Floer functor}

We now indicate how to use standard techniques in Floer theory to define a dg multifunctor \begin{equation}\label{eq-Floer-functor}
    \mathrm{CF}_\rho : C_*(\mathrm{Floer}^{\mathrm{reg}}_\rho)\to \mathrm{Ch}_{\Lambda_{0}},
\end{equation}
     for each $\rho\in \Lambda_{0}$. Here $\mathrm{Floer}^{\mathrm{reg}}_\rho$ is the full subcategory of $\mathrm{Floer}^{\mathrm{reg}}$ with objects $\bfH = (\bfL_0, (J_{01},H_{01}), \bfL_1)$ such that the superpotential values of both $\bfL_0$ and $\bfL_1$ as defined in Section \ref{sec-superpotential} are equal to $\rho$, and $C_*$ is the symmetric normalized cubical chains functor as described in \cite[Appendix B.1.3]{Abouzaid2023}, with value the dg-multicategory of chain complexes $\mathrm{Ch}_{\Lambda_{0}}$ ($\mathbb{Z}/2$ or $\mathbb{Z}$-graded, depending on the assumptions).

For an object $\bfH = (\bfL_0, (J_{01},H_{01}), \bfL_1)$ of $\mathrm{Floer}^{\mathrm{reg}}_\rho$, we define $$\mathrm{CF}_\rho(\bfH)=CF^*(\bfL_0,\bfL_1; H_{01},J_{01}),$$ which is a chain complex (rather than a matrix factorization as would happen if the potential values are different) due to Corollary \ref{cor-same-superpotential}.

For multimorphism spaces with more than one input, we associate to each $n$-cube $\fd$ a map of degree $-n$, obtained from the components of $\Mbar(\fd)$ of dimension $0$ (which are regular by our assumptions). Finally, considering those moduli spaces $\Mbar(\fd)$ of dimension $1$, we conclude the compatibility with differentials by observing that the boundaries of these moduli spaces are either associated to the faces of $\fd$ or to breaking of Floer trajectories. The only point worth noting is that, by our assumptions, none of these moduli spaces contain curves with disk and sphere bubbles, since we have assumed that there is more than one input (the case of $1$-dimensional moduli spaces of strips is the reason for Corollary \ref{cor-same-superpotential}).

 \section{The definition of Fukaya category with support}\label{sec:def-fuk-with-support}

\subsection{A review of the operadic left Kan extension}\label{sec-review-left-Kan}

We fix a commutative ring $R$ and the following data: \begin{itemize} \item a differential graded multicategory $\cM$
\item a differential graded multicategory $\cO$
\item a differential graded multi-functor $\pi: \cM\to \cO$
\item an $\cM$-algebra $A: \cM\to \Ch_R $
\end{itemize} 
We call this the \emph{abstract setup} in this section. 

Given a multicategory $\cM$,  the \emph{colors of $\cM$}, denoted $\Color(\cM)$, is the multicategory with the same objects as $\cM$, where the only multimorphisms are the endomorphism spaces which are all assumed to be the unit object of $\Ch_R$ (which is the ground ring considered as a complex concentrated in degree $0$). 

We have a forgetful functor ($U$ stands for underlying) \begin{equation}U: Alg_{\cM}\to Alg_{\Color \cM},\end{equation}which admits a left adjoint: \emph{the free algebra functor} $\mathbb{F}_{\cM}: Alg_{\Color \cM}\to Alg_{\cM},$ which we now construct.

Let $C=\{C_x\}_{x\in Ob(\cM)}$ be an object of $Alg_{\Color \cM}.$ For every $y\in Ob(\cM)$, we define
\begin{equation} \label{eq:formula_free_functor}
  \mathbb{F}_{\cM}(C)(y):= \bigoplus_{n=1}^\infty \bigoplus_{\vec{x}\in Ob(\cM)^n} C_{x_1}\otimes\ldots \otimes C_{x_n} \otimes \cM(\vec{x}; y).
\end{equation}

Using the multicompositions in $\cM$ one can define chain maps \begin{equation}\cM(\vec{x};y)\to Hom_{\Ch_R}(\mathbb{F}_{\cM}(C)(x_1)\otimes \ldots \otimes \mathbb{F}_{\cM}(C)(x_n), \mathbb{F}_{\cM}(C)(y)),\end{equation} which are compatible with multicompositions. There is a natural isomorphism \begin{equation}Alg_\cM(\bF_\cM(C),A)\simeq \prod_{x\in Ob(\cM)} Hom_{\Ch_R}(C_x, UA_x)\end{equation}certifying that indeed the free algebra functor is left adjoint to the forgetful functor.

Let us denote the category of simplicial objects in $Alg_\cM$ by $Alg_\cM^\Delta$.  Note that an $\cM$-algebra $A$ can be thought of as a simplicial $\cM$-algebra whose $n$-simplices are given by $A$ for all simplicial degrees $n\geq 0$ and all face and degeneracy maps being the identity. For the next statement, we also recall that the geometric realization $C_*:\Ch_R^\Delta\to \Ch_R$
 maps every simplicial chain complex to the direct sum of the (shifted) underlying complexes, equipped with the alternating sum of the face maps. 

\begin{lem}
  The Eilenberg-Zilber shuffle map determines a lift
  \begin{equation}
       C_*' : Alg_\cM^\Delta\to Alg_\cM 
     \end{equation}
     of the geometric realization functor to the category of algebras over $\cM$.

\end{lem}

Finally, we construct the functor $\mathbb{L}\pi_*: Alg_\cM\to Alg_\cO$ in the abstract setup. Let us first define a functor on algebras over multicategories of colors  $\Color(\pi)_*: Alg_{\Color \cM} \to Alg_{\Color \cO}$ which maps an algebra $x\mapsto C_x$ to the algebra 
\begin{equation} y\mapsto \bigoplus_{x\in \pi^{-1}(y)}C_x.\end{equation}

We define a simplicial object $(\Delta\pi)_*A_\bullet$ in $Alg_\cO$ by the formula
\begin{equation} \label{eq:left_Kan_extension_Formula}
  (\Delta\pi)_*A_n= (\bF_\cO\circ \Color(\pi)_*\circ U)\circ (\bF_{\cM}U)^{n}(A).\end{equation}
It is clear how all degeneracy maps and all but one of the face maps (for each $n\geq 1$) are defined by the functoriality of $\bF_\cO\circ \Color(\pi)_*$ using the unit and the counit of the free-forgetful adjunction. For those last maps we construct an algebra map \begin{equation} \bF_\cO\circ \Color(\pi)_*\circ U\circ \bF_{\cM}(D)\to  \bF_\cO\circ \Color(\pi)_*(D),\end{equation}where $D$ is an $Ob(\cM)$-indexed collection, using $\pi$ to turn multimorphisms in $\cM$ to multimorphisms in $\cO$ and then using multicompositions in $\cO$.

\begin{defin}
  The extension of $A$ along $\pi$ is the geometric realization
  \begin{equation}
        \mathbb{L}\pi_*A : = C_*'  (\Delta\pi)_*A_\bullet.
  \end{equation}
\end{defin}

\subsection{The complete definition}

For any superpotential value $\rho\in \Lambda_0$, we have constructed the diagram

\begin{equation} 
  \begin{tikzcd}
    C_*(\mathrm{Floer}_\rho^{\mathrm{reg}}) \ar[r,"\mathrm{CF}"] \ar[d,"\pi"] & \Ch_{\Lambda_{0}} \\
    C_*(\mathrm{Lag}_\rho) . & 
  \end{tikzcd}
\end{equation} Here $\mathrm{Lag}_\rho$ denotes the full subcategory of $\mathrm{Lag}$ with objects consisting of pairs of Lagrangian branes which both have superpotential value $\rho$. As before, $C_*$ refers to symmetric normalized cubical chains.

Given $K\subset M$ compact, we consider the full subcategory $ \mathrm{Floer}_{K,\rho}^{\mathrm{reg}}$ of $ \mathrm{Floer}_{\rho}^{\mathrm{reg}}$ consisting of objects  $\bfH = (\bfL_0, (J_{01},H_{01}), \bfL_1)$ such that $H_{01}<0$ on $K.$ By simply restricting the Floer functor, we obtain the diagram 

\begin{equation}\label{diagram-left-Kan}
  \begin{tikzcd}
    C_*(\mathrm{Floer}_{K,\rho}^{\mathrm{reg}}) \ar[r,"\mathrm{CF}_K"] \ar[d,"\pi"] & \Ch_{\Lambda_{0}} \\
    C_*(\mathrm{Lag}_\rho) . & 
  \end{tikzcd}
\end{equation}

We obtain from the discussion in the previous section the operadic left Kan extension dg-multifunctor \begin{equation}
    \mathbb{L}\pi_*\mathrm{CF}_K:  C_*(\mathrm{Lag}_\rho)\to \Ch_{\Lambda_{0}}.
\end{equation}

Then, crucially for applications, we compose this with the $T$-adic completion functor $\Ch_{\Lambda_{0}}\to \Ch_{\Lambda_{0}}$ to obtain the dg-multifunctor:

\begin{equation}
  \mathrm{Fuk}_{K,\rho}:  C_*(\mathrm{Lag}_\rho)\to \Ch_{\Lambda_{0}}.
\end{equation}

We are now exactly at the stage described in the introduction, just prior to Equation \eqref{eq:boardman_vogt} (which we refer to as a Boardman-Vogt model below): the operad $C_*(\mathrm{Lag}_\rho)$ consists of copies of the operad of chains on the stable moduli spaces of discs, which we denoted there by $\scrO$, with only additional discrete data to record composability of operations. 



\begin{defin}
  For  compact subset $K$ of $M$, and each potential value $\rho\in \Lambda_0,$ a choice of Boardman-Vogt model $f: \mathcal{A}\to \mathcal{O},$ determines the $A_\infty$-category $\mathrm{Fuk}_{K,\rho}^f$as follows:
  \begin{enumerate}
  \item Its objects are objects Lagrangian branes with superpotential value $\rho$,
  \item Its morphisms space are  $\mathrm{Fuk}_{K,\rho}(\mathbf{L}_0,\mathbf{L}_1)$,
   \item  for each $k\geq 2$ the operation  $$\mu_k:=\mathrm{Fuk}_{K,\rho}(f([\mu_k])):\mathrm{Fuk}_{K,\rho}(\mathbf{L}_{k-1},\mathbf{L}_{k})\otimes \ldots\otimes\mathrm{Fuk}_{K,\rho}(\mathbf{L}_{0},\mathbf{L}_1)\to \mathrm{Fuk}_{K,\rho}(\mathbf{L}_0,\mathbf{L}_k).$$
  \end{enumerate}
\end{defin}
\begin{rem}
    The objects and morphism spaces of the  choice of  $A_\infty$-category $\mathrm{Fuk}_{K,\rho}^f$ is by construction independent of the Boardman-Vogt model $f: \mathcal{A}\to \mathcal{O}$. We expect, but do not prove, that a homotopy equivalence between different choices leads to quasi-isomorphic $A_\infty$-categories. Note that the same choice of $f$ can be used for all $K\subset M$ pairs.
\end{rem}

\begin{proof}[Proof of Theorem \ref{thm-restriction-functors}]
  A Hamiltonian which is negative on $K'$ is is particular negative on $K$, so that $\mathrm{Floer}_{K',\rho}^{\mathrm{reg}}$ is a full subcategory of $\mathrm{Floer}_{K,\rho}^{\mathrm{reg}}$, and $\mathrm{CF}_{K'}$ is the restriction of $\mathrm{CF}_{K}$ under this inclusion, which defines a map of simplicial $C_*(\mathrm{Lag}_\rho)$-algebras
  \begin{equation}
    (\Delta\pi)_*\mathrm{CF}_{K',\bullet} \to (\Delta\pi)_*\mathrm{CF}_{K,\bullet}.
  \end{equation}
  Applying geometric realization and $T$-adic completion, we obtain a strict map of $C_*(\mathrm{Lag}_\rho)$-algebras $\mathrm{Fuk}_{K',\rho}\to\mathrm{Fuk}_{K,\rho}$. Since the operations of $\mathrm{Fuk}_{K,\rho}^f$ and $\mathrm{Fuk}_{K',\rho}^f$ are obtained by evaluating these algebras on the same chains $f([\mu_k])$, the functoriality of restriction maps immediately follows.

  The argument for equivariance is exactly the same: let $\phi$ be a symplectomorphism of $M$, which acts on bulk grading structures mapping $\mathbf{b}$ to $\phi \mathbf{b}$.   The pushforward of a Lagrangian brane $\bfL = (L, J_L, \ldots)$ for $\mathbf{b}$ is a brane
  \begin{equation}
    \phi(\bfL) := (\phi(L), \phi_* J_L, \ldots),
  \end{equation}
for $\phi \mathbf{b}$.  Composing every pseudo-holomorphic map with $\phi$ identifies the moduli spaces of discs defining the superpotential, preserving areas, Maslov indices, and weights, so that $\phi(\bfL)$ is weakly tautologically unobstructed whenever $\bfL$ is, and has the same superpotential value. Similarly, $\phi$ maps induces an isomorphism of multicategories
  \begin{equation}
    \mathrm{Floer}^{\mathrm{reg}}_{K,\rho} \to \mathrm{Floer}^{\mathrm{reg}}_{\phi(K),\rho},
  \end{equation}
  under which the functors $\mathrm{CF}_{K}$ and $\mathrm{CF}_{\phi(K)}$ are identified.  This isomorphism induces a strict isomorphism of $C_*(\mathrm{Lag}_\rho)$-algebras $\mathrm{Fuk}_{K,\rho}\to\mathrm{Fuk}_{\phi(K),\rho}$, and hence, evaluating both sides on the chains $f([\mu_k])$, a strict $A_\infty$ functor
  \begin{equation}
    \mathrm{Fuk}_{K,\rho}^f \to \mathrm{Fuk}_{\phi(K),\rho}^f.
  \end{equation}
Since we did not make any choice to define this functor, it is compatible with composition, and it strictly commutes with restriction functors.
\end{proof}

\section{Computing the morphism spaces using acceleration data}\label{sec:acceleration-data}
In the previous section we have constructed $A_\infty$-categories $\mathrm{Fuk}_{K,\rho}^f$ and restriction functors $\mathrm{Fuk}_{K',\rho}^f\to \mathrm{Fuk}_{K,\rho}^f$. The goal of this section is to show that for a fixed pair of arbitrary objects $\mathbb{L}_0,\mathbb{L}_1$, we can compute the morphism spaces and the map induced on these by the restriction functor, up to homotopy, using acceleration data. Let us fix the data of $\rho$ and $f$ throughout the section and omit them from the notation.

To state our goal more precisely, recall that an acceleration datum for a subset $K\subset M$ consists of a sequence $H_1< H_2<\ldots$ of admissible (for $(\mathbb{L}_0,\mathbb{L}_1)$) Hamiltonians $H_i: [0,1]\times M\to \bR$ satisfying $H_i|_{[0,1]\times K}<0$ and so that $$
H_i(t,x)\xrightarrow[i\to\infty]{}\begin{cases}
0,& x\in K,\\
+\infty,& x\notin K
\end{cases}
$$ for every $(t,x)\in [0,1] \times M$, together with a choice of monotone interpolations between $H_i$ and $H_{i+1}$ for all $i\geq 1$. 

For $R$ a commutative ring, let $\mathbb{N}_{R}$ be the category with objects the natural numbers $1,2,\ldots$ and morphisms \begin{equation}\operatorname{Hom}(i,j)=\begin{cases}  
R & i\leq j\\ 0  & i>j,
 \end{cases}\end{equation} and let $\star_R$ be the category with one object and morphism spaces equal to $R$. We can therefore denote the Floer theoretic diagram (considering the Lagrangian Floer cochains on the pair $(\mathbb{L}_0,\mathbb{L}_1)$) induced by an acceleration datum $H_K$ for $K$ as a functor $\mathcal{C}_{H_K}:\mathbb{N}_{\Lambda_0}\to \Ch_{\Lambda_0}$.

Below, we recall the telescope model for the homotopy colimit of such a diagram which admits canonical chain maps
\begin{equation}\label{tell}
\widehat{\operatorname{tel}}(\mathcal{C}_{H_K})\to \Fuk_{K}^*(\mathbb{L}_0,\mathbb{L}_1).
\end{equation} 

The main goal of this section is to prove the following two results.

\begin{prop}\label{prop-tel-quasi}
    The comparison maps from Equation \eqref{tell} are quasi-isomorphisms.
\end{prop}

\begin{prop}\label{prop-tel-compare-main}Assume that $K'\supset K$ are compact subsets of $M$ and let $H_K$ and $H_{K'}$ be acceleration data such that $H'_i\leq H_i$ for all $i$. Then, we can construct a chain homotopy $\widehat{\operatorname{tel}}(\mathcal{C}_{H_{K'}}) \to \Fuk^*_{K}(\mathbb{L}_0,\mathbb{L}_1)[1] $ for the diagram:
   \begin{equation}\label{tel-square-comp}
        \begin{tikzcd}
\widehat{\operatorname{tel}}(\mathcal{C}_{H_{K'}}) \arrow{rr}{}\arrow{d}{} && \widehat{\operatorname{tel}}(\mathcal{C}_{H_{K}}) \arrow{d}{} \\
\Fuk^*_{K'}(\mathbb{L}_0,\mathbb{L}_1)  \arrow{rr}{} && \Fuk^*_{K}(\mathbb{L}_0,\mathbb{L}_1)
\end{tikzcd}
   \end{equation} where the top row is constructed as in \cite[Section 3.4]{Varolgunes2021} and the vertical arrows are the canonical comparison maps from Equation \eqref{tell}.

Moreover, given nested compact subsets $K''\supset K'\supset K$ and homotopy commutative squares as in Diagram \eqref{tel-square-comp}, then we can construct a chain homotopy  $\widehat{\operatorname{tel}}(\mathcal{C}_{H_{K''}}) \to \Fuk^*_{K}(\mathbb{L}_0,\mathbb{L}_1)[2] $ for the prismatic diagram 
\begin{center}
    \begin{tikzcd}[row sep=3em, column sep=1.5em]
\widehat{\operatorname{tel}}(\mathcal{C}_{H_{K''}}) \arrow[rr] \arrow[dr] \arrow[dd] 
  & & \widehat{\operatorname{tel}}(\mathcal{C}_{H_{K'}}) \arrow[dl] \arrow[dd] \\
  & \widehat{\operatorname{tel}}(\mathcal{C}_{H_{K}}) & \\
\Fuk^*_{K''}(\mathbb{L}_0,\mathbb{L}_1) \arrow[rr] \arrow[dr] 
  & & \Fuk^*_{K'}(\mathbb{L}_0,\mathbb{L}_1) \arrow[dl] \\
  & \Fuk^*_{K}(\mathbb{L}_0,\mathbb{L}_1) \arrow[uu, leftarrow, crossing over], &
\end{tikzcd}
\end{center}
 where the chain homotopy $\widehat{\operatorname{tel}}(\mathcal{C}_{H_{K''}}) \to\widehat{\operatorname{tel}}(\mathcal{C}_{H_{K}})[1]$ in the top triangle is constructed as in \cite[Section 3.4]{Varolgunes2021} and the one at the bottom triangle is the zero map. 

\qed
\end{prop}

\subsection{Background on homotopy colimits}
Throughout, all constructions are enriched over $\Ch_R$ for some arbitrary commutative ring $R$, which is omitted from the notation. In particular, we write $\star$ for $\star_R$, and $\bN$ for $\bN_R$. We assume that the morphism complexes and the values of all dg modules and weights under consideration are $K$-flat over $R$, so that the displayed bar constructions compute derived tensor products.
\subsubsection{Telescopes}
Given a diagram $F:\mathbb{N}\to \Ch$, or equivalently, $$F(1)\overset{\kappa_1}{\to} F(2)\overset{\kappa_2}{\to}\ldots,$$ where $\kappa_i$ is the image of the unit element, we can define the telescope chain complex $$\operatorname{tel}(F):= \mathrm{cone}\left(\bigoplus_{i=1}^\infty F(i)\overset{\kappa-\mathrm{id}}{\longrightarrow} \bigoplus_{i=1}^\infty F(i)\right),$$ where $\kappa$ is the linear map that sends the summand $F(i)$ to $\bigoplus_{i=1}^\infty F(i)$ by composing $\kappa_i$ with the canonical inclusion.

\subsubsection{Weighted homotopy colimits}

Let $\mathcal{A}$ be a dg-category with a dg-functor $w:\mathcal{A}^{op}\to \Ch.$ The homotopy colimit chain complex $\mathrm{hocolim}(F,w)$ of a dg-module $F:\mathcal{A}\to \Ch$ weighted by $w$ is the geometric realization of the semi-simplicial chain complex whose $n$-simplices are given by
\begin{equation}
  \bigoplus_{x_0, \ldots, x_n \in \Ob \mathcal{A} } F(x_0) \otimes \mathcal{A}(x_0, x_1) \otimes \cdots \otimes \mathcal{A}(x_{n-1}, x_n)\otimes w(x_n)     
\end{equation}
and whose face maps are given for $i=0$ by the action of $\mathcal{A}$ on $F$, for $0 < i < n$ by composition in $\mathcal{A}$, and for $i=n$ by the action of $\mathcal{A}$ on $w$.

An important special case is when $w$ is given as follows. Take a dg-functor $\pi:\mathcal{A}\to \mathcal{B}$ and $b\in \Ob\mathcal{B}$. Then, $w_{\pi,b}(a):=\mathcal{B}(\pi(a),b)$ defines a weight. For example, we can take $\mathcal{A}=\mathbb{N}$ and consider $\pi:\mathbb{N}\to \star$ which acts on each non-trivial morphism space by the identity map, then we obtain a canonical map  \begin{equation}\label{eq-tel-hocolim}
    \operatorname{tel}(F)\to \mathrm{hocolim}(F,w_{\pi,\star}),
\end{equation} which is well-known to be a quasi-isomorphism (see, e.g. \cite{Riehl2014} or \cite{Shulman2006}).

\begin{rem}
    It is important to keep in mind that, in the enriched context, there is no homotopy colimit that is defined without choosing a weight. Since the unit object in the symmetric monoidal category of chain complexes is not a final object, this seems particularly confusing in our discussion. In contrast, for the symmetric monoidal categories of topological spaces, simplicial sets, and symmetric cubical sets, the unit object is a final object, i.e. every topological space has a unique continuous map to a singleton. Therefore, categories enriched over these have a terminal weight obtained by their map down to a category with a single object and the unit object as its morphisms, and the corresponding weighted homotopy colimits are simply called homotopy colimits. Thus, if the dg-category at hand is obtained by taking chains on a topologically enriched category (or a variant such as simplicial or cubical complexes), then it comes equipped with a preferred weight as well. For example, $\mathcal{A}=\mathbb{N}$ above is obtained by taking chains on a category enriched over topological spaces, where all non-zero morphism spaces are a singleton.
\end{rem}

\subsubsection{Cofinality}
Let $\mathcal{A}$ be a dg-category with a weight of the form $w=w_{\pi,b}:\mathcal{A}^{op}\to \Ch$ for some dg-functor $\pi: \mathcal{A}\to \mathcal{B}$ and $b\in \Ob \mathcal{B}.$ Consider a dg-module $F:\mathcal{A}\to \Ch$. Given $a:\mathbb{N}\to \mathcal{A}$ with the property that $\pi(a(i))=b$ and $\pi(a(i\to i+1))=id_b$ for all $i\in \mathbb{N},$ we can immediately construct a map $$\iota_a: \operatorname{tel}(a^*F)\to \mathrm{hocolim}(F,w)$$ that sends the unshifted copy of $\bigoplus_{i=1}^\infty F(a(i))$ to the $0$-simplices by the canonical map  (where $e_b$ is the unit of $b$):  $$\bigoplus_{i=1}^\infty F(a(i))\otimes e_b\to \bigoplus_{x_0 \in \mathcal{A} } F(x_0)\otimes \mathcal{B}(\pi(x_0),b)$$ and the shifted copy to the 1-simplices by $$\bigoplus_{i=1}^\infty F(a(i))\otimes \mathcal{A}(a(i), a(i+1))\otimes e_b\to  \bigoplus_{x_0,x_1\in \Ob \mathcal{A} } F(x_0) \otimes \mathcal{A}(x_0, x_1)\otimes\mathcal{B}(\pi(x_1),b) .$$

We now give a well-known criterion for when $\iota_a$ is a quasi-isomorphism that is a variant of Quillen's Theorem A. 
\begin{lemma}\label{lem-quillen-thma}
    Assume that for every $x\in \Ob\mathcal{A},$ there exists a natural number  $n(x)$ such that the map induced by $\pi$: $$\mathcal{A}(x,a(i))\to \mathcal{B}(\pi(x),b)$$ is a quasi-isomorphism for every $i> n(x).$ Then, $\iota_a$ is a quasi-isomorphism.
\end{lemma}
\begin{proof}
We use the notation $F \otimes_{\mathcal{A}}^{\mathbb{L}} w$ for $\operatorname{hocolim}(F, w)$ to make the following argument appear more familiar. We have the following sequence of quasi-isomorphisms:$$\begin{aligned} \operatorname{tel}(a^*F) &\simeq a^*F \otimes_{\mathbb{N}}^{\mathbb{L}} w_{\pi, \star} \\ &\simeq \left( F \otimes_{\mathcal{A}}^{\mathbb{L}} \mathcal{A}(-, a(-)) \right) \otimes_{\mathbb{N}}^{\mathbb{L}} w_{\pi, \star} \\ &\simeq F \otimes_{\mathcal{A}}^{\mathbb{L}} \left( \mathcal{A}(-, a(-)) \otimes_{\mathbb{N}}^{\mathbb{L}} w_{\pi, \star} \right). \end{aligned}$$ Here the first isomorphism was alluded to above in Equation \eqref{eq-tel-hocolim} and the last two are well-known properties of the derived tensor product.

By our hypothesis, for each object $x \in \operatorname{Ob}(\mathcal{A})$, the sequence of chain complexes $\mathcal{A}(x, a(i))$ is eventually quasi-isomorphic to $\mathcal{B}(\pi(x), b) = w_{\pi,b}(x)$. Therefore, the canonical map$$F \otimes_{\mathcal{A}}^{\mathbb{L}} \left( \mathcal{A}(-, a(-)) \otimes_{\mathbb{N}}^{\mathbb{L}} w_{\pi, \star} \right) \simeq F \otimes_{\mathcal{A}}^{\mathbb{L}} w_{\pi,b}$$is a quasi-isomorphism.

One can check that composing all these maps yields $\iota_a$, and hence $\iota_a$ is a quasi-isomorphism.
\end{proof}

\begin{rem}
    It would suffice to assume that the canonical map from the homotopy colimit of $F_x: \mathbb{N}\to \Ch$ defined by $F_x(i)= \mathcal{A}(x,a(i))$ (and in the obvious way for morphisms) with the preferred weight coming from $\mathbb{N}\to \star$ as above (or we can consider the telescope) to $\mathcal{B}(\pi(x),b)$ is a quasi-isomorphism for every $x.$
\end{rem}

\subsubsection{Homotopy left Kan extension evaluated at an object is a weighted homotopy colimit}\label{sec-weighted-leftkan}
Let $\mathcal{A}$ be a dg-category and $F:\mathcal{A}\to \Ch$ be a dg-module. Consider also a dg-functor $\pi: \mathcal{A}\to \mathcal{B}$. We can think of $\mathcal{A}$ and $\mathcal{B}$ as multicategories with all multiple input multimorphism spaces taken to be the zero chain complex, which makes $\pi$ a multi-functor and $F$ an $\mathcal{A}$-algebra. Using the discussion from Section \ref{sec-review-left-Kan}, we can define the homotopy left Kan extension $\mathbb{L}\pi_*F: \mathcal{B}\to \Ch$.

Unravelling the definitions, we see that for $b\in \Ob \mathcal{B},$ we have \begin{equation}\label{eq-weighted-leftkan}
    \mathbb{L}\pi_*F(b)=\mathrm{hocolim}(F,w_{\pi,b}).
\end{equation}

\subsection{Proof of Proposition \ref{prop-tel-quasi}}
\begin{defin}
Given a multicategory $\mathcal{M}$, define $P(\mathcal{M})$ to have
\begin{enumerate} \item Objects $\operatorname{Ob}(P(\mathcal{M}))$ given by (finite) ordered sequences of objects of $\mathcal{M}$. 

\item Morphism complexes between objects $\vec{a}$ and $\vec{b}$ of $P(\mathcal{M})$ are given by: \begin{equation}P(\mathcal{M})(\vec{a},\vec{b}):= \bigoplus_{f:[n]\to [m]}\bigotimes_{j=1}^m \mathcal{M}(\vec{a}|_{f^{-1}(j)}, b_j),\end{equation} where the direct sum is over all order preserving surjections $f: [n]\to [m]$.
\item The composition structure is defined using the composition in the multicategory.
\end{enumerate}
\end{defin}

In fact, this operation is functorial (sending multifunctors to functors). We can therefore apply it to Diagram \eqref{diagram-left-Kan} and obtain
\begin{equation}
  \begin{tikzcd}
   P C_*(\mathrm{Floer}_{K,\rho}^{\mathrm{reg}}) \ar[r,"\mathrm{CF}_K"] \ar[d,"P\pi"] & \Ch_{\Lambda_{0}} \\
   P C_*(\mathrm{Lag}_\rho) . & 
  \end{tikzcd}
\end{equation} For the horizontal arrow, we have also used the monoidal structure of $\Ch.$

The definition of the left Kan extension shows that 
\begin{equation}
\mathbb{L}\pi_*\mathrm{CF}_K(\mathbb{L}_0,\mathbb{L}_1)=\mathbb{L}P\pi_*\mathrm{CF}_K(\mathbb{L}_0,\mathbb{L}_1),
\end{equation} where the right hand side is to be understood as in Section \ref{sec-weighted-leftkan} and $(\mathbb{L}_0,\mathbb{L}_1)$ is seen as a length $1$ sequence object of $P C_*(\mathrm{Lag}_\rho)$.

Finally, we finish the proof by combining Equation \eqref{eq-weighted-leftkan}, Lemma \ref{lem-quillen-thma} and most importantly Proposition \ref{lmFrgtHtpyEq}, noting that $T$-adic completion preserves quasi-isomorphisms between free $\Lambda_0$-modules (see \cite[Corollary 2.3.6]{Varolgunes2021}).

\subsection{Proof of Proposition \ref{prop-tel-compare-main}}

\tikzset{
  floernode/.style={circle, draw, solid, thick, inner sep=1pt, minimum size=6mm},
  floeredge/.style={thick}
}
\newcommand{\lintree}[2][]{%
\begin{tikzpicture}[x=0.9cm, baseline=(current bounding box.center)]
  \node (n0) at (0,0) {$u$};
  \xdef\lintreelast{0}
  \foreach \lab [count=\i] in {#2} {
    \node[floernode] (n\i) at (\i,0) {$\lab$};
    \pgfmathtruncatemacro{\j}{\i-1}
    \draw[floeredge] (n\j) -- (n\i);
    \xdef\lintreelast{\i}
  }
  \pgfmathtruncatemacro{\lintreenext}{\lintreelast+1}
  \node (nend) at (\lintreenext,0) {};
  \draw[floeredge] (n\lintreelast) -- (nend);
  \ifx\\#1\\\else\node at ({\lintreenext/2},-0.6) {$#1$};\fi
\end{tikzpicture}}
We first insert a homotopy layer, using the continuation maps and homotopies of \cite[Section~3.4]{Varolgunes2021}, replacing the telescope part of the diagram by one with strict Hamiltonian inequalities. This is achieved by shifting the target acceleration sequence by one index. The vertical arrows between the original and replacement completed telescopes are quasi-isomorphisms by the sandwiching argument of \cite[Section~3.3.2]{Varolgunes2021}. We retain the original notation for the strict replacement; concatenating with this first layer recovers the required homotopy for the original diagram.

Let us define the dg-categories $\mathcal{A}_K = P C_*(\mathrm{Floer}^{\mathrm{reg}}_{K, \rho})$ and $\mathcal{A}_{K'} = P C_*(\mathrm{Floer}^{\mathrm{reg}}_{K', \rho})$. Because any admissible Hamiltonian strictly negative on $K'$ is also strictly negative on $K$, there is a fully faithful inclusion dg-functor $I \colon \mathcal{A}_{K'} \hookrightarrow \mathcal{A}_K$; we suppress the functor $I$ in the tree diagrams below for readability. The canonical restriction map $\Fuk_{K'}^*(\mathbb{L}_0, \mathbb{L}_1) \to \Fuk_{K}^*(\mathbb{L}_0, \mathbb{L}_1)$ is defined exactly by this inclusion via the strict functoriality of the operadic left Kan extension. 

Following Section \ref{sec-weighted-leftkan}, we compute the morphism complexes $\Fuk_K^*(\mathbb{L}_0, \mathbb{L}_1)$ using their respective weighted homotopy colimits. Explicitly, prior to the $T$-adic completion, the morphism complex is modeled by $\operatorname{hocolim}(F_K, w_K)$, where $F_K = \mathrm{CF}_K \colon \mathcal{A}_K \to \Ch_{\Lambda_0}$ is the Floer chain complex dg-module, and the weight is $w_K(x) = \mathcal{B}(\pi(x), b)$, with $\mathcal{B} = P C_*(\mathrm{Lag}_\rho)$, $\pi = P\pi$, and $b = (\mathbb{L}_0, \mathbb{L}_1)$. 

Let $a_{K'} \colon \mathbb{N} \to \mathcal{A}_{K'}$ and $a_K \colon \mathbb{N} \to \mathcal{A}_K$ be the dg-functors corresponding to the acceleration data $H_{K'}$ and $H_K$. Note that these functors intertwine the two Floer functors we have been considering, so that $\mathcal{C}_{H_K} = a_K^* F_K$. Let $\kappa'_i$ and $\kappa_i$ denote their respective internal continuation maps. The pointwise inequalities $H_{K'} < H_K$ provide regular continuation maps $c_i \colon I(a_{K'}(i)) \to a_K(i)$. Choosing regular 1-cubes of Floer data interpolating between these continuations yields degree $-1$ homotopies $S_i \colon I(a_{K'}(i)) \to a_K(i+1)$ satisfying the internal differential relation $d_{int}(S_i) = \kappa_i \circ c_i - c_{i+1} \circ I(\kappa'_i)$.

To explicitly track the terms in the weighted homotopy colimits, we introduce a graphical language. We build on the framework used in Figure \ref{fig:labelledtree}, but ``linear'' trees will suffice for the purposes of this section. Note that the left and right branes are fixed throughout the discussion, so we omit them from the pictures. 

A $k$-simplex chain $u \otimes m_1 \otimes m_2 \otimes \dots \otimes m_k \otimes e_b$ is represented as a linear tree with $k+2$ nodes, read from left to right:
\begin{center}
\begin{tikzpicture}[x=1.6cm, baseline=(current bounding box.center)]
  \node (u)  at (0,0) {$u$};
  \node[floernode] (m1) at (1,0) {$m_1$};
  \node[floernode] (m2) at (2,0) {$m_2$};
  \node (dots) at (3,0) {$\cdots$};
  \node[floernode] (mk) at (4,0) {$m_k$};
  \node[floernode] (e)  at (5,0) {$e_b$};
  \node (out) at (6,0) {};
  \draw[floeredge] (u) -- (m1);
  \draw[floeredge] (m1) -- (m2);
  \draw[floeredge] (m2) -- (dots);
  \draw[floeredge] (dots) -- (mk);
  \draw[floeredge] (mk) -- (e);
  \draw[floeredge] (e) -- (out);
\end{tikzpicture}
\end{center}

\begin{itemize}
    \item The leftmost node carries the Floer cochain $u \in F_{K}(a_K(i))$.
    \item The next $k$ nodes carry the morphisms $m_1, \ldots, m_k$ of $\mathcal{A}_K$, in order.
    \item The rightmost node carries the unit $e_b$. 
\end{itemize}

The semi-simplicial boundary operator $d_\Delta = \sum (-1)^j d_j$ evaluates on these trees as an alternating sum of three types of node-collapse operations:
\begin{enumerate}
    \item $d_0$: The node $m_1$ acts on $u$, removing the leftmost node, and replacing the label at the start of the tree by $u \cdot m_1$ (with Koszul signs).
    \item $d_j$ for $0<j<k$: The adjacent nodes $m_j$ and $m_{j+1}$ are merged into a node carrying their composition $m_{j+1} \circ m_j$.
    \item $d_k$: The node $m_k$ is merged with the node $e_b$ via the map to $\mathcal{O}$. Crucially, the forgetful functor maps regular continuation maps ($0$-cubes) to the operadic unit $e_b$, so $m_k \cdot e_b = e_b$ in that case. In contrast, if $m_k$ is a higher dimensional cube, $m_k \cdot e_b = 0$.
\end{enumerate}

The unshifted elements $u \in F_{K'}(a_{K'}(i))$ and shifted elements $su \in F_{K'}(a_{K'}(i))[1]$ generate the mapping telescope $\operatorname{tel}(\mathcal{C}_{H_{K'}})$, with the telescope differential $d_{tel}(su) = u \cdot \kappa'_i - u - s(du)$. 
The vertical comparison map $\iota_{K'}$ to the homotopy colimit acts as $\iota_{K'}(u) = u \otimes e_b$ and $\iota_{K'}(su) = u \otimes \kappa'_i \otimes e_b$. The horizontal telescope map $T \colon \operatorname{tel}(\mathcal{C}_{H_{K'}}) \to \operatorname{tel}(\mathcal{C}_{H_K})$ acts as $T(u) = u \cdot c_i$ and $T(su) = s(u \cdot c_i) - u \cdot S_i$.

We can now explicitly construct the degree $-1$ chain homotopy $h_1 \colon \operatorname{tel}(\mathcal{C}_{H_{K'}}) \to \operatorname{hocolim}(F_K, w_K)[1]$ which witnesses the fact that the square \eqref{tel-square-comp} is homotopy commutative. 

For an \emph{unshifted element} $u$, $h_1(u)$ is defined as the depth-1 tree:
\begin{center}
\lintree{c_i, e_b}
\end{center}

Applying the boundary operator $d_\Delta$ gives $(u \cdot c_i) \otimes e_b - u \otimes (c_i \cdot e_b)$. Since $c_i \cdot e_b = e_b$, this yields exactly $\iota_K(T(u)) - \iota_{K'}(u)$.

For a \emph{shifted element} $su$, we define $h_1(su)$ algebraically as a signed sum of two depth-2 trees (triangulating the square) and one depth-1 tree incorporating the homotopy $S_i$:
\begin{center}
\lintree[T_1]{c_i, \kappa_i, e_b} \ $-$ \ \lintree[T_2]{\kappa'_i, c_{i+1}, e_b} \ $- $ \ $\epsilon_u$\lintree[T_3]{S_i, e_b}
\end{center} Here and afterward, $\epsilon_u:=(-1)^{|u|}.$

Applying the total chain complex differential $D_{tot} = d_\Delta + d_{int}$ yields canceling terms: the internal differential of the third tree acts on $S_i$, yielding $d_{int}(S_i) = \kappa_i \circ c_i - c_{i+1} \circ I(\kappa'_i)$. When applying the internal edge-collapses ($d_1$) to the first two trees, we generate exactly $- u \otimes (\kappa_i \circ c_i) \otimes e_b$ and $+ u \otimes (c_{i+1} \circ I(\kappa'_i)) \otimes e_b$, which cancel against the $d_{int}(S_i)$ terms.

What survives are the top and bottom edge-collapses ($d_0$ and $d_2$) of $T_1$ and $T_2$, and the top collapse ($d_0$) of $T_3$:
\begin{itemize}
    \item $d_0(T_1) = (u \cdot c_i) \otimes \kappa_i \otimes e_b = \iota_K(s(u \cdot c_i))$.
    \item $-\epsilon_ud_0(T_3) = -(u \cdot S_i) \otimes e_b = -\iota_K(u \cdot S_i)$. Together, these form $\iota_K(T(su))$.
    \item $- d_2(T_2) = - u \otimes I(\kappa'_i) \otimes e_b = -\iota_{K'}(su)$.
    \item $d_2(T_1) = u \otimes c_i \otimes e_b = h_1(u)$.
    \item $- d_0(T_2) = - (u \cdot I(\kappa'_i)) \otimes c_{i+1} \otimes e_b = - h_1(u \cdot I(\kappa'_i))$.
\end{itemize}
Grouping the last two terms yields $-h_1(u \cdot I(\kappa'_i) - u) = -h_1(d_{tel}(su))$ (ignoring the natural boundaries associated with $du$). This verifies the strict chain homotopy equation $D_{tot} h_1(su) + h_1(d_{tel}(su)) = \iota_K T(su) - \iota_{K'} su$.

Now we come to the prismatic diagram corresponding to nested compact subsets $K'' \supset K' \supset K$. We similarly insert a first homotopy layer replacing the telescope part of the prism by a diagram with strict Hamiltonian inequalities, shifting the three acceleration sequences by zero, one, and two indices, respectively. Using the constructions of \cite[Section~3.4]{Varolgunes2021}, we choose this layer coherently over all faces, including the homotopy in the top triangle. Its vertical arrows are quasi-isomorphisms by the sandwiching argument of \cite[Section~3.3.2]{Varolgunes2021}. The construction below applies to the strict replacement, and concatenation with the first layer gives the required filling of the original prism. 

We use the continuation maps $c'_i \colon I(a_{K''}(i)) \to a_{K'}(i)$ and $c_i \colon I(a_{K'}(i)) \to a_K(i)$ as above. We complete this section by defining the degree $-2$ chain homotopy $h_2$ filling the prism.

For an \emph{unshifted element} $u \in F_{K''}(a_{K''}(i))$, $h_2(u)$ is defined as the depth-2 tree:
\begin{center}
\lintree{c'_i, c_i, e_b}
\end{center}
Applying the boundary $d_\Delta$ on this tree strips away one node at a time. The top collapse yields $h_1^{K'\to K}(u \cdot c'_i)$, the middle collapse yields $h_1^{K''\to K}(u)$ using the composed continuation map $c_i \circ c'_i$, and the bottom collapse yields $h_1^{K''\to K'}(u)$. These exactly map to the $1$-dimensional chain homotopies $h_1$ evaluated on the three bounding square faces of the prism. 

The prism has three square faces, corresponding to three square homotopies $S_i$, $S'_i$, and $S''_i$ (where $S''_i$ is the square homotopy for the composed map $c_i \circ c'_i$). Note that these $S$'s are defined using 1-cubes of continuation maps (which define $1$-chains). The three-dimensional homotopy $V_i$ filling the prism is then chosen using a 2-chain of continuation maps.

The homotopy map $h_2$ on a \emph{shifted element} $su \in F_{K''}(a_{K''}(i))[1]$ is a signed sum of six trees. This includes three depth-3 trees, two depth-2 trees incorporating the square homotopies $S_i$ and $S'_i$, and one depth-1 tree for the prism filling homotopy $V_i$:

\begin{center}
\lintree[T_1]{c'_i, c_i, \kappa_i, e_b} \quad $-$ \quad \lintree[T_2]{c'_i, \kappa'_i, c_{i+1}, e_b}

\medskip
$+$ \quad \lintree[T_3]{\kappa''_i, c'_{i+1}, c_{i+1}, e_b} \quad $-$ \quad $ \epsilon_u$\lintree[T_4]{c'_i, S_i, e_b}

\medskip
$- $ \quad $\epsilon_u$ \lintree[T_5]{S'_i, c_{i+1}, e_b} \quad $-$ \quad \lintree[T_6]{V_i, e_b}
\end{center}

Applying $D_{tot}$ follows a similar pattern. The internal differentials of the shorter trees cancel all $d_1$ and $d_2$ internal edge-collapses of the taller trees. Crucially, the 2-cube filling homotopy $V_i$ has an internal differential involving the three 1-cube square homotopies corresponding to the square faces of the prism: $d_{int}(V_i) = -S_i \circ c'_i - c_{i+1} \circ S'_i + S''_i$. The first two terms cancel the remaining internal edge-collapses of $T_4$ and $T_5$, while the third term provides the depth-1 tree component $\epsilon_u u \otimes S''_i \otimes e_b$ needed to form the third boundary square face $-h_1^{K'' \to K}(su)$. What remains are exactly the boundary face evaluations, so that the prism commutes up to homotopy. (The bottom triangle of the prism, corresponding to unshifted generators mapping between the identical sequence indices, strictly commutes due to the functoriality of restriction maps). 

Applying the $T$-adic completion yields the required degree $-1$ and degree $-2$ chain homotopies, completing the proof of Proposition \ref{prop-tel-compare-main}. 
\qed

\section{Mayer-Vietoris property}\label{sec:mayer-vietoris}

\subsection{The statement}\label{sec-statement-descent}
We first introduce a definition that will allow us to apply our main descent result (Theorem \ref{thm-descent}) to a broad class of examples. A toy example of this definition occurs when each $K_i$ is a codimension $0$ manifold with contact boundary, and all boundaries are pairwise disjoint. In that case, the functions $f_{m,i}$ below can be chosen to approximate the characteristic function of $K_i$:

\begin{defin}\label{def-weak-inv} A collection $\{ K_1,...,K_N \} $ of compact subsets of $M$ are \emph{weakly Poisson commuting} if there exist smooth functions $f_{m,i}:M\to \bR$ for $m=1,\ldots, N$ and $i=1,2,\ldots$ such that the Poisson brackets $\{f_{m,i},f_{m',i}\}$ vanish for each natural number $i$ and for each pair  $(m,m')$ and the following additional properties hold:
\begin{itemize}
\item each function $f_{m,i}$ is strictly negative on $K_m$, i.e.  $f_{m,i} |_{K_m}<0$,
\item the functions $f_{m,i}$ are strictly increasing in $i$, i.e.  $f_{m,i} < f_{m,i+1}$ , and
\item if $f_{m,i}(x)<0$ for all $i$, then $x\in K_m$.
\end{itemize}
If $K=K_1\cup\ldots\cup K_N$ and $K_1,\ldots, K_N$ are weakly Poisson commuting compact subsets, we call $K_1,\ldots, K_N$ a weakly involutive cover of $K$.
\end{defin}

\begin{lem}\label{lem-weak-Pois}
If $K_1,...,K_N$ are weakly Poisson commuting, then the collection of compact subsets consisting of all unions of their intersections is also weakly Poisson commuting.\end{lem}
\begin{proof}
    Unions and intersections can be realized by smoothings of $\min$ and $\max$ respectively. Observe now that if three functions $(f,g,h)$ pairwise Poisson commute, and $\phi:\bR^2\to\bR$ is smooth, then $\{\phi(f,g),h\}=0$. 
\end{proof}

If $J$ is a subset of $\{1, \ldots, N\}$, we shall write $K_J$ for the intersection $ \bigcap_{m\in J}K_m$.  
Our goal is to prove Theorem \ref{thm-descent}.

Let us first define the target of the functor, which is the \v{C}ech $A_\infty$-category associated to a cover $K=K_1\cup\ldots \cup K_N$. We start by defining the morphisms. For any two objects $(\mathbb{L}_0,\mathbb{L}_1)$ of $\Fuk_M(K)$, we define the morphism chain complex for $(\mathbb{L}_0,\mathbb{L}_1)$ by $$\bigoplus_{p\geq 0}\bigoplus_{|J|=p+1}\Fuk^*_M(K_J)(\mathbb{L}_0,\mathbb{L}_1)[p],$$where $[-]$ denotes the shift operator on chain complexes and $J$ ranges through all subsets of $\{1,\ldots ,N\}$, and with differential the sum of the \v{C}ech differential and the internal (Floer) differential. Here, we are using the ordered version of the  \v{C}ech complex, so that the outer direct sum is finite. 

In order to describe the $A_\infty$-structure it is better to first consider a more abstract approach. Let $\mathbb{N}^{inj}$ denote the category with objects $0,1,2,\ldots $ and morphisms from $p$ to $q$ the set of injective order-preserving maps $\{n\in \mathbb{Z}\mid 1\leq n\leq p+1\} \to \{n\in \mathbb{Z}\mid 1\leq n\leq q+1\}$. A contravariant (respectively covariant) functor $\mathbb{N}^{inj}\to Ch(\Lambda_0)$ is called a semi-simplicial (respectively semi-cosimplicial) chain complex. We shall use these notions for more general target categories than chain complexes.

Let $({\Fuk}^\bullet_M,d_{i})$ be the semi-cosimplicial $A_\infty$-category on the objects of the Fukaya category of $M$, which on morphisms is given by $$p\mapsto \bigoplus_{|J|=p+1}\Fuk^*_M(K_J)(\mathbb{L}_0,\mathbb{L}_1),$$ with coface maps $d_i$ obtained from the restriction maps in the standard way. 

Let us define the  semi-simplicial differential graded associative algebra
$$
p\mapsto NC^*(\Delta^p):= \bigoplus_{F\subset \Delta^p}\Lambda_{\geq 0}\cdot \delta_F$$ where the direct sum is over all faces of $\Delta^p$ (which are in one to one correspondence with all the morphisms in $\mathbb{N}^{inj}$ with target $n$) and $\delta_F$ is in degree equal to the dimension of $F.$ We think of the elements of $NC^*(\Delta^p)$ as assignments of scalars to each face of $\Delta^p$. We equip $NC^*(\Delta^p)$ with the differential that sends $\delta_F$ to the alternating sum of $\delta_{F'}$, where $F'$ is a face containing $F$ of dimension $1$ higher. The face maps of $NC^*(\Delta^\bullet)$ are those induced by the face inclusions $\Delta^p\to \Delta^{q}.$

We can define the totalization $A_\infty$-category $Tot(\Fuk^\bullet_M)$ as the category with the same objects as $\Fuk_M(K)$, with morphism spaces $$Tot(\Fuk^\bullet_M)(\mathbb{L}_0,\mathbb{L}_1):=\operatorname{eq} \left(\prod_{p=0}^\infty \Fuk^p_M(\mathbb{L}_0,\mathbb{L}_1)\otimes NC^*(\Delta^p)\rightrightarrows \prod_{r\to q} \Fuk^q_M(\mathbb{L}_0,\mathbb{L}_1)\otimes NC^*(\Delta^r)\right).$$ Here the second product is over all maps in $\mathbb{N}^{inj}$ and the symbol $\operatorname{eq}$ refers to the equalizer, i.e. the kernel of the difference of the two natural maps.  

\begin{prop}\label{prop-Tot-Cech}
If $(K_1, \ldots, K_N)$ is a weakly involutive cover of $K$, then for every pair of objects $(\mathbb{L}_0,\mathbb{L}_1)$ there is an isomorphism of chain complexes
$$Tot(\Fuk^\bullet_M)(\mathbb{L}_0,\mathbb{L}_1)\to \Cech(\Fuk_M; K_1,\ldots,K_N)(\mathbb{L}_0,\mathbb{L}_1)$$ which intertwines the canonical maps received from  $\Fuk_M^*(K)$.\qed
\end{prop}

Finally, $Tot(\Fuk^\bullet_M)$ can be equipped with an $A_\infty$ structure. The key point is that if $C$ is a dga, then $D\mapsto D\otimes C$ defines a lax  monoidal functor $Ch(\Lambda_0)\to Ch(\Lambda_0)$ with the natural transformation $$(D_1\otimes C)\otimes (D_2\otimes C)\to (D_1\otimes D_2)\otimes C$$given by the product structure of $C.$ This implies that the category with the same objects as $\Fuk_M(K)$ and morphism spaces for $(\mathbb{L}_0,\mathbb{L}_1)$ equal to the chain complex $\Fuk^p_M(\mathbb{L}_0,\mathbb{L}_1)\otimes NC^*(\Delta^p)$ has an induced $A_\infty$-structure. The same is true for $\prod_{p=0}^\infty \Fuk^p_M(\mathbb{L}_0,\mathbb{L}_1)$ by simply taking the product structure and it is elementary to check that the structure descends to the equalizers.

\begin{rem}
    This discussion works equally well on any kind of non-symmetric operadic structure. There is nothing special about $A_\infty$-categories.
\end{rem}

We therefore obtain an $A_\infty$-structure on $\Cech(\Fuk_M; K_1,\ldots,K_N)$ using Proposition \ref{prop-Tot-Cech}. This structure can of course be described explicitly, see for example \cite[Equation (14.24)]{Bott82} for the binary product operation.

\subsection{Reducing to the construction of a special diagram}\label{subsec-reduce-special}

The proof of Theorem \ref{thm-descent} reduces to proving the following result, where we recall that $K=\bigcup_i K_i$:
\begin{thm}\label{thm-descent-linear}
  For any pair of objects $(\mathbb{L}_0,\mathbb{L}_1),$ the chain map on morphisms $$\Fuk^*_{M}(K)(\mathbb{L}_0,\mathbb{L}_1)\to \Cech(\Fuk_{M}; K_1,\ldots,K_N)(\mathbb{L}_0,\mathbb{L}_1)$$ is a quasi-isomorphism. 
\end{thm}
Let us fix the objects $(\mathbb{L}_0,\mathbb{L}_1)$ from now on.

\begin{lemma}\label{lem-descent-2-to-N}
    It suffices to prove Theorem \ref{thm-descent-linear} for $N=2.$
\end{lemma}
\begin{proof}
This follows from Lemma \ref{lem-weak-Pois} and the standard inclusion exclusion argument that reduces the case of a cover consisting of $N$ subsets to the case of a cover with $N-1$ subsets by taking the union of two of the sets.
\end{proof}

To proceed with the proof of Theorem \ref{thm-descent-linear}, we need to use explicit models for morphisms in the Fukaya categories with support. To this end, recall that an acceleration datum for a subset $K\subset M$ consists of a sequence $H_1< H_2<\ldots$ of admissible Hamiltonians $H_i: [0,1]\times M\to \bR$, for $(\mathbb{L}_0,\mathbb{L}_1)$, satisfying $H_i|_{[0,1]\times K}<0$  and so that $$
H_i(t,x)\xrightarrow[i\to\infty]{}\begin{cases}
0,& x\in K,\\
+\infty,& x\notin K
\end{cases}
$$  for every $(t,x)\in [0,1] \times M$, together with a choice of monotone interpolations between $H_i$ and $H_{i+1}$ for all $i\geq 1$. Let us denote the corresponding diagram of Floer chain complexes by $\mathcal{C}_H:\mathbb{N}\to Ch(\Lambda_0)$. As shown in Proposition~\ref{prop-tel-quasi},  there are quasi-isomorphisms
\begin{equation}\label{tel}
\widehat{\operatorname{tel}}(\mathcal{C}_H)\to \Fuk^*_{M}(K)(\mathbb{L}_0,\mathbb{L}_1).
\end{equation} 

We now continue with our proof of Theorem \ref{thm-descent-linear}, which will provide an alternative to the Mayer-Vietoris argument from \cite{Varolgunes2021}. The first step, for a nested inclusion of compact subsets, is to ensure that the maps from the telescope are functorial in a homotopy coherent way. Here is the specific statement that we need:

\begin{prop}\label{prop-tel-compare} If $K_1$ and $K_2$ are compact subsets of $M$, there is a homotopy commutative diagram:
   \[
    \begin{tikzcd}
\widehat{\operatorname{tel}}(\mathcal{C}_{K_1\cup K_2}) \arrow{rr}{}\arrow{d}{} && \operatorname{cocone}\left(\widehat{\operatorname{tel}}(\mathcal{C}_{K_1})\oplus\widehat{\operatorname{tel}}(\mathcal{C}_{K_2})\to \widehat{\operatorname{tel}}(\mathcal{C}_{K_1\cap K_2})\right)\arrow{d}{} \\
\Fuk^*_{M}(K_1\cup K_2)(\mathbb{L}_0,\mathbb{L}_1)  \arrow{rr}{} && \Cech(\Fuk_{M}; K_1,K_2)(\mathbb{L}_0,\mathbb{L}_1)
\end{tikzcd}
\]where the top row constructed using appropriate acceleration data as in \cite[Section 3.4]{Varolgunes2021} and the vertical arrows are quasi-isomorphisms.\qed
\end{prop}

This reduces the desired result to the statement that the upper horizontal arrow is an equivalence. The diagram of Figure~\ref{fig-cubic-diagram}, in which each horizontal line corresponds to one of the mapping telescopes, encodes the arrows between the underlying Floer complexes.
\begin{figure}[ht]
\centering
\begin{tikzpicture}
  \matrix (m) [matrix of math nodes, row sep=1.5em,
    column sep=1em]{
    & \ldots & & CF(H_i^{K_1\cup K_2}) & & CF(H_{i+1}^{K_1\cup K_2}) & & \ldots   \\
    \ldots & & CF(H_i^{K_1})  & & CF(H_{i+1}^{K_1})  & & \ldots & \\
    &\ldots & & CF(H_i^{K_2})  & & CF(H_{i+1}^{K_2}) & & \ldots \\
    \ldots & & CF(H_i^{K_1\cap K_2})  & & CF(H_{i+1}^{K_1\cap K_2}) & &\ldots  & \\};
  \path[-stealth]
    (m-1-4) edge (m-1-6) edge (m-2-3)
            edge [densely dotted] (m-3-4) edge (m-2-5)
            edge [densely dotted] (m-4-3) 
            edge [densely dotted] (m-3-6)
            edge [densely dotted] (m-4-5)
    (m-1-6) edge [densely dotted] (m-3-6) edge (m-2-5) edge [densely dotted] (m-4-5)
    
    (m-2-3) edge [-,line width=6pt,draw=white] (m-2-5)
            edge (m-2-5) edge (m-4-3) edge (m-4-5)
    (m-3-4) edge [densely dotted] (m-3-6)
            edge [densely dotted] (m-4-3)
            edge [densely dotted] (m-4-5)
    (m-4-3) edge (m-4-5)
    (m-3-6) edge (m-4-5)
    (m-2-5) edge [-,line width=6pt,draw=white] (m-4-5)
            edge (m-4-5)
    (m-1-2) edge (m-1-4)
    (m-2-1) edge (m-2-3) 
    (m-3-2) edge [densely dotted](m-3-4) 
    (m-4-1) edge (m-4-3) 
    (m-1-6) edge (m-1-8) 
    (m-2-5) edge [-,line width=6pt,draw=white] (m-2-7) edge (m-2-7)
    (m-3-6) edge (m-3-8) 
    (m-4-5) edge (m-4-7);             
  \end{tikzpicture}
\caption{The diagram of Floer complexes determined by acceleration data for $K_1\cup K_2$, $K_1$, $K_2$, and $K_1\cap K_2$.}\label{fig-cubic-diagram}
\end{figure}
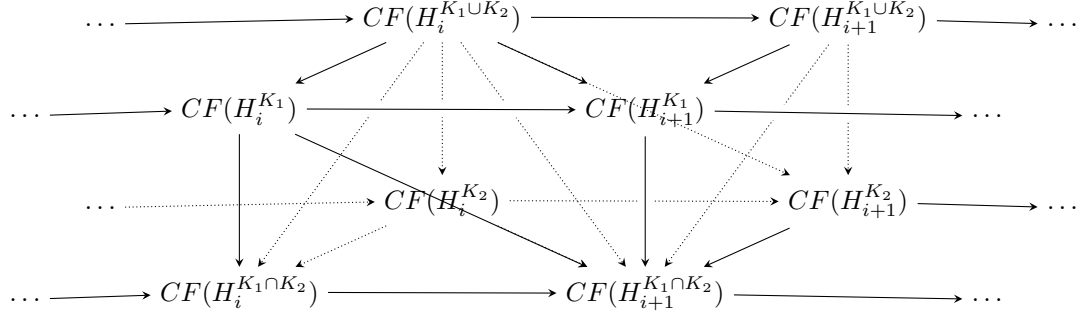

\noindent In this diagram, the Hamiltonians $H_i^K$ (for $K \in \{K_1,K_2,K_1 \cup K_2, K_1 \cap K_2\}$)  are the acceleration data, and we can visualise it as a diagram in $3$-space, consisting of a collection of cubes placed side by side. The edges of the cubes are chain maps. The diagonals along the square faces are chain homotopies and the big diagonal inside the cube is a null-homotopy of the sum (with signs) of 6 terms obtained by composing a square diagonal and an edge to go from the beginning to the end of the big diagonal (see \cite[Figure 8]{Varolgunes2021}). We have suppressed the Lagrangian branes and almost complex structures from the notation.

Having reduced the proof of the Mayer-Vietoris property to the claim that the top horizontal arrow in Proposition \ref{prop-tel-compare} is an equivalence, we now consider a pair $K_1,K_2$ of weakly Poisson commuting compact subsets, and show that the acceleration data can be chosen so that, in fact, each slice of the diagram above, namely
\begin{align}
\xymatrix{ 
CF(H_n^{K_1\cup K_2})\ar[r]\ar[d]\ar[dr]& CF(H_n^{K_1})\ar[d]\\ CF(H_n^{K_2})\ar[r] &CF(H_n^{K_1\cap K_2})}
\end{align}
is acyclic. A key input is the following lemma, which replaces each pair of Hamiltonians in the acceleration data for the pair $K_1,K_2$ by a nearby pair of Hamiltonians for which the minimum and maximum are admissible smooth functions in the sense of Definition \ref{def:admissible_Hamiltonian}, and whose $1$-chords satisfy a separation condition.

\begin{lemma}\label{lem-main-separated-admissible}
Let $f_1,f_2:M\to\mathbb{R}$ be smooth functions satisfying $\{f_1,f_2\}=0$. For every $\epsilon>0$ there exist time-dependent Hamiltonians $g_1,g_2:[0,1]\times M\to\mathbb{R}$ such that, setting
\[
U:=\{g_1<g_2\},\qquad V:=\{g_2<g_1\},\qquad W:=\{g_1=g_2\}
\]
as subsets of $[0,1]\times M$, the following properties hold:
\begin{enumerate}
\item\emph{(Approximation from below)} for $m=1,2$, the functions $f_m$ and $g_m$ satisfy $g_m\leq f_m$ and $\|g_m-f_m\|_{C^0([0,1]\times M)}<\epsilon$,
\item\emph{(Disjointness)} the closures of the subsets $U$ and $V$ are disjoint: $\overline{U}\cap\overline{V}=\varnothing$,
\item\emph{(Smooth Matching)} the complements $([0,1]\times M)\setminus U$ and $([0,1]\times M)\setminus V$ are each the closure of an open set,
\item\emph{(Separation)} no $1$-chord of $X_{g_1}$, $X_{g_2}$, $X_{\min(g_1,g_2)}$, or $X_{\max(g_1,g_2)}$ from $\mathbb{L}_0$ to $\mathbb{L}_1$ has a graph\footnote{The graph of $\gamma:[0,1]\to M$ is the image of $\mathrm{id}\times\gamma:[0,1]\to[0,1]\times M$.} that intersects both $U$ and $V$, and
\item\emph{(Admissibility)} $g_1$, $g_2$, $\min(g_1,g_2)$, and $\max(g_1,g_2)$ are admissible (and in particular non-degenerate) Hamiltonians for $(\mathbb{L}_0,\mathbb{L}_1)$.
\end{enumerate}

\end{lemma}

The motivation for introducing these conditions is the following lemma. 

\begin{lemma}[{\cite[Proposition~4.1.1]{Varolgunes2021}}]\label{lem-var-711}
    Let $g_1,g_2:[0,1]\times M\to\mathbb{R}$ be admissible Hamiltonians satisfying the disjointness, smooth matching, separation, and admissibility conclusions of Lemma~\ref{lem-main-separated-admissible}. Then, for any regular $2$-cube Floer datum (i.e. continuation maps between the Floer groups at the corners and a homotopy for the square),  extending $g_1,g_2,\min(g_1,g_2),\max(g_1,g_2)$ at the corners, the square
    \begin{align}\label{eq-acyclic-square}
    \xymatrix{
    CF(\min(g_1,g_2))\ar[r]\ar[d]\ar[dr]& CF(g_1)\ar[d]\\
    CF(g_2)\ar[r] &CF(\max(g_1,g_2))
    }
    \end{align}
    is acyclic.
    \end{lemma}

We now prove Lemma~\ref{lem-main-separated-admissible} in detail. The proof proceeds in two stages: first tangentialization (Section~\ref{sec-tangentialization}) yields a pair of approximating autonomous functions $(\tilde f_1,\tilde f_2)$ approximating $(f_1, f_2)$ from below, and satisfying both Disjointness, and Separation; Smooth Matching is achieved later. Then the perturbation lemmas of Section~\ref{subsec-perturb} are applied: we perturb jointly on the interior of the overlap to achieve non-degeneracy there, then perturb each function separately (downward, each in the region where the other dominates) so that the region of equality becomes the complement of two disjoint open sets, yielding Smooth Matching and Admissibility.

\subsection{Tangentialization}\label{sec-tangentialization}

\begin{defin}\label{def-tangentialization}
A pair $(\tilde f_1,\tilde f_2)$ approximating $(f_1,f_2)$ from below is a \emph{tangentialization} of $(f_1,f_2)$ if:
\begin{itemize}
    \item $\min(\tilde f_1,\tilde f_2)$ and $\max(\tilde f_1,\tilde f_2)$ are smooth;
    \item (Disjointness) the dominance regions $\{\tilde f_1<\tilde f_2\}$ and $\{\tilde f_2<\tilde f_1\}$ are open with disjoint closures;
    \item (Separation) the functions $\tilde f_1$ and $\tilde f_2$ Poisson commute.
\end{itemize}
\end{defin}
To see how separation, in our earlier sense, follows from the last condition, observe that the Hamiltonian flow of $\tilde{f}_i$ preserves the level sets of both $\tilde{f}_i$ and $\tilde{f}_j$, hence its chords cannot intersect both regions.

\begin{figure}[h!]
\centering
\begin{tikzpicture}[scale=0.9, baseline=(g.base)]
  \begin{scope}[local bounding box=g]
    \draw[->] (-2.8,0) -- (2.8,0) node[right] {$t$};
    \draw[->] (0,-1.25) -- (0,1.5) node[above] {$g(t)$};
    \draw[thick,domain=-2.5:-1.0,samples=20] plot (\x, \x);
    \draw[thick] (-0.5,0) -- (0.5,0);
    \draw[thick,domain=1.0:2.5,samples=20] plot (\x, \x);
    \draw[thick,domain=-1.0:-0.5,samples=25] plot (\x, {-1.0+1.0*((\x+1.0)/0.5)*((\x+1.0)/0.5)});
    \draw[thick,domain=0.5:1.0,samples=25] plot (\x, {1.0*((\x-0.5)/0.5)*((\x-0.5)/0.5)});
    \draw[dashed] (-1.0,-1.0) -- (-1.0,0) (1.0,0) -- (1.0,1.0);
    \draw[dashed] (-0.5,0) -- (-0.5,-0.7) (0.5,0) -- (0.5,-0.7);
    \node[below] at (-1.0,-0.05) {$-g_+$};
    \node[below] at (1.0,-0.05) {$g_+$};
    \node[below] at (-0.5,-0.75) {$-g_-$};
    \node[below] at (0.5,-0.75) {$g_-$};
  \end{scope}
\end{tikzpicture}
\caption{The auxiliary function $g$ used in tangentialization.  It agrees with the identity for $|t|\geq g_+$ and vanishes on $[-g_-,g_-]$, with $g'(t)>0$ for $|t|>g_-$.}
\label{fig-g}
\end{figure}
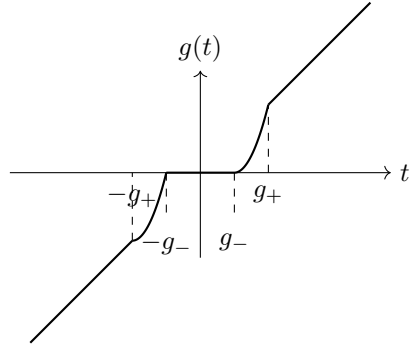

Let us denote the coordinates on $\mathbb{R}^2$ by $x$ and $y.$
Consider a smooth function $g:\mathbb{R}\to \mathbb{R}$ that satisfies\begin{enumerate}
        \item $g(t)=t$ for $|t|\geq g_+$ for some $g_+\in \mathbb{R}_{>0}$,
        \item $g(t)=0$ for $|t|\leq g_-$ for some $g_-\in \mathbb{R}_{>0}$ with $g_- < g_+$, and 
        \item $g'(t)>0$ for $|t|> g_-$.
    \end{enumerate}
%
%

The reader might want to think about the following constructions in the coordinates $(u,v)$ with $u=x-y$ and $v=x+y$. We will continue using the original coordinate system as that will be most convenient for later use.

We define smooth functions $x_g,y_g:\mathbb{R}^2\to\mathbb{R}$ by
\[
x_g(x,y):=\frac{x+y+g(x-y)}{2},\qquad
y_g(x,y):=\frac{x+y-g(x-y)}{2}.
\]

\begin{lemma}\label{lem-tang-equality-region}
Let $M$ be a smooth manifold and $F=(f_1,f_2):M\to\mathbb{R}^2$ a smooth map.  Fix $g$ as above and, for some $\delta\ge 0$, set
\[
\tilde f_1:=x_g\circ F-\delta,\qquad
\tilde f_2:=y_g\circ F-\delta.
\]
Then:
\begin{enumerate}
    \item The equality locus
    \[
    E:=\{\tilde f_1=\tilde f_2\}
    \]
    is the closed subset $F^{-1}(\{|x-y|\le g_-\})$.  Its complement decomposes as a disjoint union
    \[
    M\setminus E \;=\; U\sqcup V,\qquad
    U:=\{\tilde f_1<\tilde f_2\},\; V:=\{\tilde f_2<\tilde f_1\},
    \]
    and $U,V$ are open with disjoint closures.
    \item The functions $\min(\tilde f_1,\tilde f_2)$ and $\max(\tilde f_1,\tilde f_2)$ are smooth on $M$.
    \item For every $\epsilon>0$ one can choose $g_+$ and $\delta$ so that $\tilde f_i\le f_i$ and $\|\tilde f_i-f_i\|_{C^0}<\epsilon$ for $i=1,2$.
    \end{enumerate}
    Moreover, if $M$ is symplectic, and $(f_1,f_2)$ Poisson commute, then so do $(\tilde f_1,\tilde f_2)$, and the pair $(\tilde f_1,\tilde f_2)$ is a tangentialization of $(f_1,f_2)$ in the sense of Definition~\ref{def-tangentialization}.
\end{lemma}

\begin{proof}
Write $u=x-y$ and $v=x+y$, so that $x_g=(v+g(u))/2$ and $y_g=(v-g(u))/2$.  Then
\[
x_g-y_g=g(u),
\]
so $x_g=y_g$ if and only if $g(u)=0$, i.e.\ $|u|\le g_-$, which gives item~(1) after composing with $F$.  Since $\tilde f_1$ and $\tilde f_2$ are smooth and agree on $E$, the functions $\min(\tilde f_1,\tilde f_2)$ and $\max(\tilde f_1,\tilde f_2)$ are obtained by gluing together smooth functions that coincide on the closure of an open region, hence are smooth, proving item~(2).  For item~(3), observe that when $g_+$ is small the map $(x,y)\mapsto(x_g,y_g)$ is $C^0$-close to the identity: indeed $g(u)=u$ for $|u|\ge g_+$, while $|g(u)|\le g_+$ for $|u|\le g_+$, so $|x_g-x|=|g(u)-u|/2\le g_+$ and likewise $|y_g-y|\le g_+$.  Composing with $F$ and subtracting $\delta>g_+$ therefore gives the result. Finally, we have
\[
u=f_1-f_2,\qquad v=f_1+f_2,
\]
so $\{u,v\}=2\{f_1,f_2\}=0$, and therefore any two functions of $(u,v)$ Poisson commute; in particular $\{\tilde f_1,\tilde f_2\}=0$. 
\end{proof}

\begin{figure}[htbp]
    \centering
    \includegraphics{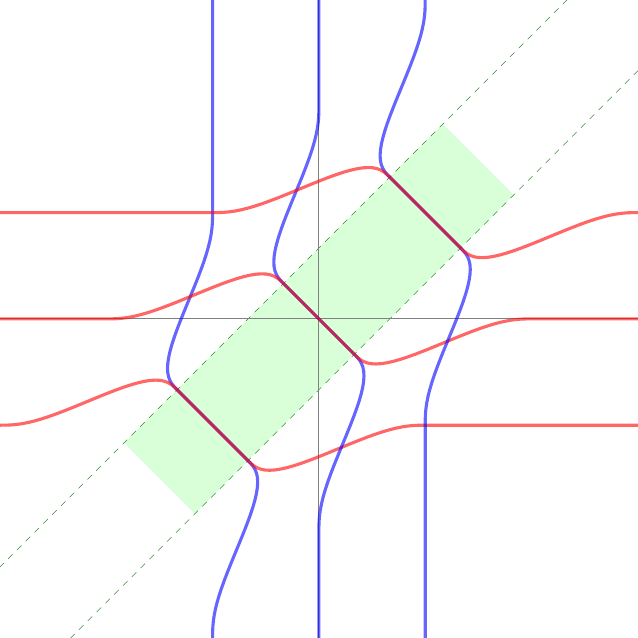}
    \caption{The blue and red curves represent the level curves of $x_g$ and $y_g$. Curves that are tangentially overlapping in the green region take the same value.}
\end{figure}

\subsection{Perturbation lemmas }\label{subsec-perturb}

Let $\mathbb{L}_0,\mathbb{L}_1$ be Lagrangian branes with underlying Lagrangian submanifolds $L_0,L_1$, and let $H:[0,1]\times M\to \mathbb{R}$ be a Hamiltonian. Recall that a $1$-chord $\gamma:[0,1]\to M$ of $H$ from $\mathbb{L}_0$ to $\mathbb{L}_1$ is \emph{non-degenerate} if the image of $T_{\gamma(0)}L_0$ under the time-$1$ map of $X_H$ is transverse to $T_{\gamma(1)}L_1$.

Admissibility (Definition \ref{def:admissible_Hamiltonian}) requires in addition that the image of $\gamma$ avoid the bulk $B$ and that the endpoints avoid the irregular loci $W_2(L_i)$ and $Z(L_i;J_i)$ associated to the branes $\mathbb{L}_i$.
We begin by recalling a local submersivity property for Hamiltonian flows from \cite{Varolgunes2021}.

\begin{lemma}[{\cite[Lemma 4.2.2]{Varolgunes2021}}] \label{lem:varolgunes}
Let $H:[0,T]\times M\to \mathbb{R}$ be a Hamiltonian and $x:[0,T]\to M$ a flow line of $X_H$. If there exists a time $\tau \in (0,T)$ such that $x(\tau) \in U$ for an open set $U \subset M$, then the time-$T$ evaluation map $h \mapsto \phi^T_{H+h}(x(0))$, where $h$ varies in the space of smooth functions supported in $ (0,T) \times U$, is a submersion at $h=0$. Moreover, every vector in $T_{x(T)}M$ is obtained from a path $s\mapsto h_s$ with $h_s\geq0$. \qed
\end{lemma}
This local property yields the global surjectivity of the evaluation maps on the universal moduli space of chords.

For a sufficiently large integer $k$, let $\mathcal{P}_\delta$ denote the space of perturbations 
\[
\mathcal{P}_{\delta} := \Bigl\{h\in C^k([0,1]\times M)\ \Big|\ \operatorname{supp}(h)\subset [0,1]\times U,\ \|h\|_{C^k}<\delta\Bigr\}.
\]
We define the universal set of $1$-chords as $\mathcal{Y} := \{ (h, x_0) \in \mathcal{P}_\delta \times L_0 \mid \phi^1_{H+h}(x_0) \in L_1 \}$, writing $x_0$ also for the chord $t \mapsto \phi^t_{H+h}(x_0)$ that a pair $(h,x_0)$ determines, and its subset
\begin{equation}
  \label{eq:universal_chords_through_U}
  \mathcal{Y}_U := \{ (h, x_0) \in \mathcal{Y} \mid \textrm{the image of } x_0 \textrm{ meets } U \},
\end{equation}
which is a smooth manifold by the implicit function theorem.
\begin{cor}\label{cor:univ-submersion}
The endpoint evaluation maps $Ev_0: \mathcal{Y}_U \to L_0$ and $Ev_1: \mathcal{Y}_U \to L_1$ are submersions. Furthermore, the total interior evaluation map
\[
Ev: \mathcal{Y}_U \times (0,1) \to M, \qquad (h,x_0,t) \mapsto \phi^t_{H+h}(x_0)
\]
is a submersion on the subset consisting of triples $(h,x_0,t_0)$ such that the chord $x_0$ meets $U$ before and after $t_0$.
\end{cor}

\begin{proof}
  Submersivity of the endpoint maps follows immediately from Lemma \ref{lem:varolgunes}. Indeed, variations of $h$ supported in $U$ ensure that, for any pair $(h,x_0)$ with the property that the time-$1$ chord of $H+h$ starting at $x_0$ passes through $U$, the evaluation map $\mathcal{P}_\delta \to M$ given by $h \mapsto \phi^1_{H+h}(x_0)$ is a submersion onto $M$. Restricting the domain to the inverse image of $L_1$ yields the claim that the time-$1$ map is a submersion to $L_1$, and letting $x_0$ vary over points in $L_0$ yields the same for the time-$0$ map. To establish surjectivity of the interior evaluation map at a time $t_0 \in (0,1)$ such that there are times $t_- < t_0 < t_+$ with the property that $x_0(t_\pm) \in U$, we apply the same result to the intervals $[0,t_0]$ and $[t_0,1]$.

\end{proof}

It is not clear, in general, how to ensure the condition that a chord intersect $U$ both before and after $t_0$; there is, however, a situation in which this is straightforward, namely when $U$ is invariant under the Hamiltonian flow of $H$. In our applications, we shall start with an autonomous function, and $U$ will be defined in terms of the level sets of Poisson commuting functions.

\begin{lemma}\label{lem-pert-nondeg}
Let $M$ be a symplectic manifold and $\mathbb{L}_0,\mathbb{L}_1$ Lagrangian branes. Let $H: [0,1]\times M\to \mathbb{R}$ be a Hamiltonian, $k$ a positive integer, $\delta>0$, $\tilde{g}$ a Riemannian metric on $M$, and $U\subset M$ the interior of a codimension $0$ submanifold with boundary which is invariant under the Hamiltonian flow of $H$. There exists a Hamiltonian $H':[0,1]\times M\to \mathbb{R}$ such that
\begin{itemize}
\item $\operatorname{supp}(H'-H)\subset [0,1]\times \bar{U}$,
\item $\|H'-H\|_{C^k([0,1]\times M)}<\delta$, relative to the metric $g=\tilde{g}+dt^2$,
\item $H'\leq H$, with equality precisely on the complement of $[0,1]\times U$,
\item every $1$-chord of $X_{H'}$ from $\mathbb{L}_0$ to $\mathbb{L}_1$ which intersects $U$ is non-degenerate and admissible (all such chords in fact lie in $U$ by the invariance of $U$).
\end{itemize}
\end{lemma}

\begin{proof}
We first shift the Hamiltonian strictly downwards to create room for a generic perturbation. Choose a non-negative smooth bump function $\rho: M \to \mathbb{R}$ vanishing precisely on the complement of $U$, and regard it as a time-independent function on $[0,1]\times M$. Let $\epsilon \in (0, \delta)$ be sufficiently small so that $\|\epsilon \rho\|_{C^k([0,1]\times M)} < \delta/2$. Define the shifted Hamiltonian $H_{\mathrm{shift}} := H - \epsilon \rho$. Note that $H_{\mathrm{shift}} \leq H$, and that the complement of $U$ is invariant under the Hamiltonian flow of $H$, and hence of $H_{\mathrm{shift}}$.

Let $\mathcal{P}^k_{\rho}$ denote the open subset of the Banach space $C^k([0,1]\times M)$ consisting of functions $b$ satisfying
\[
\|b\|_{C^k([0,1]\times M)} < \epsilon/2,
\qquad
\|\rho b\|_{C^k([0,1]\times M)} < \delta/2.
\]
We use $\rho b$ as the perturbation of $H_{\mathrm{shift}}$. Since $\rho>0$ on $U$, every local variation compactly supported in $[0,1]\times U$ can be written as $\rho\,\dot b$. We define the universal space of $C^k$ $1$-chords contained in $U$ associated to the family $H_{\mathrm{shift}}+\rho b$, as $\mathcal{Y}^k_U$. As in Corollary \ref{cor:univ-submersion}, the time-$1$ evaluation map is a submersion, ensuring that $\mathcal{Y}^k_U$ is a smooth submanifold of $\mathcal{P}^k_{\rho}\times L_0$. We may choose a generic $b \in \mathcal{P}^k_{\rho}$ so that all $1$-chords of $H' := H_{\mathrm{shift}} + \rho b = H - \epsilon \rho + \rho b$ lying in $U$ are non-degenerate. Since $\rho b$ and $\rho$ vanish outside $U$, the set $U$ remains invariant under the flow of $H'$. Thus every chord intersecting $U$ is entirely contained in $U$, so that an appropriate application of Sard-Smale ensures admissibility.

\end{proof}
While the above result is stated for $C^k$ functions, there are standard techniques to extend it to the smooth setting, see e.g.~\cite{Floer1995b, McDuff2012, Wendl}, and we shall from now assume that the resulting perturbation is in fact smooth (and $C^k$-close for any $k$).

Note that, in the above proof, the invariance of $U$ is only used because it ensures that a chord of the Hamiltonian flow of $H$ (or a perturbation thereof) which meets $U$ is entirely contained in it, and hence that it intersects $U$ both before and after any possible intersection with another submanifold of $M$. The flexibility afforded by the more general setting is used below:
\begin{cor}\label{cor-pert-global}
  Let $M$ be a compact symplectic manifold and $\mathbb{L}_0,\mathbb{L}_1$ Lagrangian branes. Let $H: [0,1]\times M\to \mathbb{R}$ be a Hamiltonian, and $U\subset M$ the interior of a codimension $0$ submanifold with boundary which is invariant under the Hamiltonian flow of $H$, and so that all chords which lie in $U$ are non-degenerate and admissible. Let $V\subset M$ be the interior of a codimension $0$ submanifold with boundary such that $U\cup V=M$.

  Then, for every positive integer $k$ and every $\delta>0$, there exists an admissible Hamiltonian $H':[0,1]\times M\to \mathbb{R}$ with $\operatorname{supp}(H'-H)\subset [0,1]\times \bar{V}$, with $H'\leq H$, and with $\|H'-H\|_{C^k([0,1]\times M)}<\delta$.
\end{cor}
\begin{proof}
  We consider Hamiltonians $H'$ with $\operatorname{supp}(H'-H)\subset [0,1]\times \bar{V}$ which are sufficiently close to $H$. For such Hamiltonians, we consider: (i) chords contained in $V$, (ii) chords which cross the boundary of $\bar{V}$, and (iii) chords contained in the complement of $\bar{V}$.

  The argument from Lemma \ref{lem-pert-nondeg} applies verbatim to chords contained in $V$, and chords contained in the complement of $\bar{V}$ are exactly the same as the corresponding chords of $H$, so are admissible by assumption. By choosing the perturbation to be sufficiently small, we can ensure that the chords that cross the boundary are small perturbations of those of $H$, hence also admissible.
\end{proof}

\subsection{Proof of Lemma~\ref{lem-main-separated-admissible}}\label{subsec-proof-ham-prep}

\begin{proof}
\textbf{Step~1: Tangentialization.}

Apply the tangentialization procedure of Section~\ref{sec-tangentialization} to the pair $(f_1,f_2):M\to\mathbb{R}^2$, choosing parameters $g_+<\epsilon/4$ and $\delta=\epsilon/4$, to obtain
\[
\tilde f_1 := x_{g}\circ(f_1,f_2)-\delta, \qquad \tilde f_2 := y_{g}\circ(f_1,f_2)-\delta.
\]
By Lemma~\ref{lem-tang-equality-region}, after possibly shrinking $g_+$ further we may assume $\tilde f_m\leq f_m$ and $\|\tilde f_m-f_m\|_{C^0}<\epsilon/2$. The dominance regions
\[
\widetilde U:=\{\tilde f_1<\tilde f_2\},\qquad \widetilde V:=\{\tilde f_2<\tilde f_1\}
\]
are open with $\overline{\widetilde U}\cap\overline{\widetilde V}=\varnothing$ (Disjointness). Thus we have \emph{Approximation from below}, \emph{Disjointness}, and by Poisson commutativity \emph{Separation} for $(\tilde f_1,\tilde f_2)$. 

\medskip
\textbf{Step~2: Perturbation.}
We perturb in the equality region so that all relevant chords become non-degenerate and admissible, and Smooth Matching is achieved.

\emph{2a. Joint perturbation on the interior of the overlap.} The equality locus is $E:=\{\tilde f_1=\tilde f_2\}$; let $E^\circ$ denote its interior (relative to $[0,1]\times M$).

Relying on Lemma~\ref{lem-pert-nondeg} we perturb $\tilde f_1$ and $\tilde f_2$ by the \emph{same} function $h$ supported in $E^\circ$, so that they still agree on $E$ and every chord intersecting it (hence lying in it) is non-degenerate and admissible. Use the notation $(\tilde g_1,\tilde g_2)$ for this perturbed pair. We still have that $\tilde g_1=\tilde g_2$ on $E$. Moreover, the perturbation may be taken so that it is arbitrarily $C^k$-small and $\tilde{g}_m\leq f_m$.

\emph{2b.  Separate perturbations in the dominance regions.} 
Next, choose disjoint open sets $U\supset\overline{\widetilde U}$ and $V\supset\overline{\widetilde V}$, large enough that the complement $W:=([0,1]\times M)\setminus(U\cup V)$ is contained in $E^\circ$. We apply Corollary \ref{cor-pert-global} twice, once to $\tilde g_1$ and once to $\tilde g_2$, in each case perturbing in $U\cup V$ and taking for the invariant region of the Corollary the interior $E^\circ$ of the equality locus, which is invariant because $\{\tilde f_1,\tilde f_2\}=0$; the hypothesis on chords is then supplied by Step 2a. Both functions are pushed downward, and we arrange that the drop of $\tilde g_1$ dominates that of $\tilde g_2$ on $U$, and conversely on $V$, both perturbations vanishing on $W$.

We obtain the final pair $g_1,g_2$ with
\[
g_2>g_1 \text{ on } U,\qquad g_1>g_2 \text{ on } V,\qquad g_1=g_2 \text{ on } W = ([0,1]\times M)\setminus(U\cup V)
\]
and so that every $1$-chord of $X_{g_1}$, $X_{g_2}$, $X_{\min(g_1,g_2)}$, and $X_{\max(g_1,g_2)}$ is non-degenerate and admissible. We have $\|g_m-\tilde{f}_m\|_{C^k}<\epsilon$ for arbitrarily small $\epsilon$ and $g_m\le f_m$. We now have a pair satisfying admissibility and smooth matching. We turn to verify separation.

\medskip
\textbf{Step~3: Separation.}

Since $\{f_1,f_2\}=0$, the function $f_2-f_1$ is preserved along flow lines of $X_{f_1}$, and $f_1-f_2$ along flow lines of $X_{f_2}$. So no chord of $X_{f_1}$ or $X_{f_2}$ can have graph intersecting both $\{f_1<f_2\}$ and $\{f_2<f_1\}$: the sign of $f_2-f_1$ along the chord is constant. The same holds after tangentialization for $\tilde f_1,\tilde f_2$ and the dominance regions $\widetilde U,\widetilde V$, because $\{\tilde f_1,\tilde f_2\}=0$ and the difference $\tilde f_1-\tilde f_2$ has definite sign on each side of the equality locus.

We now show that this separation property persists under the perturbations of Step 2. Set $q:=f_1-f_2$ and choose the tangentialization parameter $g_-$ so that $\pm g_-$ are regular
values of $q$. Choose $0<a<g_-$ with $\pm a$ also regular values, and
in Step~2b take
\[
U=[0,1]\times\{q<-a\},\qquad
V=[0,1]\times\{q>a\}.
\]
These sets contain $\widetilde U$ and $\widetilde V$, respectively,
have disjoint closures, and their complement $W$ is contained in
$E^\circ$. The regular-value choices ensure the smooth-boundary
hypotheses required in the perturbation lemmas. Since
$\{q,\widetilde f_m\}=0$, we may choose the total perturbations
sufficiently $C^1$-small that
\[
\max_{m=1,2}
\bigl\|\{q,g_m-\widetilde f_m\}\bigr\|_{C^0}<2a.
\]
The same bound holds for the Poisson brackets of $q$ with
$\min(g_1,g_2)$ and $\max(g_1,g_2)$, since these functions are smooth
and their differentials agree pointwise with those of one of
$g_1,g_2$. Consequently, along any time-one trajectory of any of
these four Hamiltonians, the value of $q$ changes by less than $2a$
between any two times. Such a trajectory cannot meet both $U$ and
$V$, which proves Separation.
\end{proof}

\subsection{Filling the 2-cube with regular Floer data}\label{subsec-fill-2-cube}

Having produced admissible non-degenerate Hamiltonians $g_1,g_2,\min(g_1,g_2),\max(g_1,g_2)$ at the four corners of the square \eqref{eq-acyclic-square}, we need to fill in the edges and faces with regular Floer data. 

\begin{lemma}\label{lem-fill-2cube}
Let $g_1,g_2:[0,1]\times M\to\mathbb{R}$ be admissible Hamiltonians satisfying the conclusions of Lemma~\ref{lem-main-separated-admissible}. There exist regular continuation data on each of the four edges
\[
\min(g_1,g_2)\leq g_1,\quad \min(g_1,g_2)\leq g_2,\quad g_1\leq \max(g_1,g_2),\quad g_2\leq \max(g_1,g_2),
\]
and a regular homotopy of continuation data filling the square
\begin{align}\label{eq-2cube}
\xymatrix{
CF(\min(g_1,g_2))\ar[r]\ar[d]\ar[dr] & CF(g_1)\ar[d] \\
CF(g_2)\ar[r] & CF(\max(g_1,g_2)).
}
\end{align}
\end{lemma}

The non-trivial point is regularity in the region where $H^-=H^+$. We explain the issue and carry out the construction in the following two subsections, which contain the proof of Lemma~\ref{lem-fill-2cube}.

One difference with \cite{Varolgunes2021} is that here we rely on standard transversality techniques instead of virtual techniques. 

In the discussion below, we write (as before) $W$ for the region where the Hamiltonians agree, and $U$ and $V$ for where they dominate. On the region $W$ where $g_1=g_2$ (and hence all four corner Hamiltonians agree), any monotone interpolation $H_s$ satisfying $H_s=g_1$ for $s\ll 0$ and $H_s=g_2$ for $s\gg 0$ must have $\partial_s H_s=0$ on $W$. A constant-in-$s$ solution $u(s,t)=\gamma(t)$ (where $\gamma$ is a chord of $g_1=g_2$ in $W$) is then a Floer trajectory for $H_s$ for \emph{any} almost complex structure. The linearized Floer operator at such a solution need not be surjective: varying $J$ only contributes a term involving $\partial_s u=0$, which vanishes. This is the phenomenon observed by Seidel for $H_{s,t}(z)=(2\pi s+0.5)|z|^2$ on $\mathbb{C}$ (see \cite[page~14]{Seidel2008b}), where regularity cannot be achieved without modifying the Hamiltonian.
The resolution is to require that $\partial_s H_s$ vanishes \emph{to second order} at $\gamma(t)$ for every such constant chord $\gamma$. The key fact is that, for a solution $u(s,t)=\gamma(t)$ which is $s$-independent, the topological energy of the family of deformations of $u$ determined by an element of the kernel of the linearized operator vanishes identically, and differentiating this twice forces that element to vanish; the linearized operator is therefore injective (see Lemma~\ref{lem-const-reg} below). Since $u$ is asymptotic to $\gamma$ at both ends, this operator has index $0$, so it is surjective as well.

\begin{lemma}\label{lem-const-reg}
      Let $(H^-, J^-)$ and $ (H^+, J^+)$ be admissible Floer data for $(\mathbb{L}_0,\mathbb{L}_1)$ on $M$. Let $\gamma: [0,1]\to M$ be a smooth map that is a $1$-chord of both $H^-$ and $H^+$. Let us choose a continuation datum $(\{H_{s,t}\}, \{J_{s,t}\})$ such that \begin{itemize}
          \item $(H_{s,t}, J_{s,t})=(H^-_t, J^-_t)$ for $s\ll0$, $(H_{s,t}, J_{s,t})=(H^+_t, J^+_t)$  for $s\gg0$
          \item Defining $H: \mathbb{R}\times [0,1]\times M\to \mathbb{R}$ as $H(s,t,x)=H_{s,t}(x)$, we have $\partial_sH\big|_{\{(s,t)\}\times M}: M\to \mathbb{R}$ vanishes to second order at $\gamma(t)$ for all $(s,t)\in \mathbb{R}\times [0,1].$
      \end{itemize} Then, the map $u: \mathbb{R}\times [0,1]\to M$ defined by $u(s,t)=\gamma(t)$ is a regular solution of the Floer equation \begin{align*}
     \partial_s u + J_{s,t}(u)(\partial_t u- X_{H_{s,t}}(u))=0.
\end{align*}
\end{lemma}
\begin{proof}
    Let us denote the linearization of the Floer operator at $u$ by
    $D_u$. We know that $D_u$ is a Fredholm operator of index $0$.
    Therefore, it suffices to show that $D_u$ is injective.

    Let $D_u(\eta)=0$. By linear elliptic regularity, $\eta$ is a smooth
    section. It also has the exponential decay property
    \cite[Lemma 2.11]{Salamon1999}. Choose a smooth family of auxiliary
    Riemannian metrics $g_t$ on $M$ such that $L_i$ is totally geodesic
    for $g_i$, for $i=0,1$, and set $g_{s,t}=g_t$. For $\epsilon\geq 0$
    sufficiently small, define
    \[
    u_\epsilon(s,t):=
    \exp^{g_{s,t}}_{u(s,t)}(\epsilon\eta(s,t)).
    \]
    Since $\eta(s,i)\in T_{\gamma(i)}L_i$, these maps satisfy
    $u_\epsilon(s,i)\in L_i$. We note that
    \begin{align*}
        \lim_{s\to-\infty}u_\epsilon(s,t)&=\gamma(t),\\
        \lim_{s\to+\infty}u_\epsilon(s,t)&=\gamma(t).
    \end{align*}
    In fact, this convergence is exponential.

    In the discussion below, we denote by
    $\widetilde v:\mathbb R\times[0,1]\to
    \mathbb R\times[0,1]\times M$
    the graph of any map $v:\mathbb R\times[0,1]\to M$.
    By exponential decay, the topological energy functional
    \[
    \rho(\epsilon):=E(u_\epsilon)
    =\int_{\mathbb R\times[0,1]}u_\epsilon^*\omega
    +\int_{\mathbb R\times[0,1]}
    d(\widetilde{u_\epsilon}^{\,*}H)\wedge dt
    \]
    is well defined for sufficiently small $\epsilon\geq 0$ and smooth
    at $\epsilon=0$. By Stokes's theorem for the map
    \[
    [0,\epsilon]\times\mathbb R\times[0,1]
    \longrightarrow \mathbb R\times[0,1]\times M
    \]
    obtained by looking at the graphs of $u_\delta$ with
    $0\leq\delta\leq\epsilon$, using the Lagrangian boundary conditions
    and fixed asymptotics, we have $\rho(\epsilon)=\rho(0)=0$.
    By analyzing the derivatives of $\rho$ at $0$, we will conclude
    that $\eta$ must be identically zero.

    A standard computation shows that we can rewrite
    \[
    \rho(\epsilon)=
    \underbrace{
    \int_{\mathbb R\times[0,1]}
    \omega\left(
        \frac{\partial u_\epsilon}{\partial s},
        \frac{\partial u_\epsilon}{\partial t}
        -X_{H_{s,t}}(u_\epsilon)
    \right)\,ds\,dt}_{\rho_1(\epsilon)}
    +
    \underbrace{
    \int_{\mathbb R\times[0,1]}
    \partial_sH(s,t,u_\epsilon(s,t))\,ds\,dt}_{\rho_2(\epsilon)}.
    \]

    We first show that the second-order vanishing condition implies
    $\rho_2'(0)=0$ and $\rho_2''(0)=0$. Let us start with the first
    derivative. We differentiate under the integral sign and show
    that the integrand is zero. To see this, take any smooth
    $f:M\to\mathbb R$ and Riemannian metric $g$ on $M$. Let $\nabla$
    be its Levi-Civita connection and let $\tau:[0,\delta)\to M$
    be a geodesic. Then
    \begin{align*}
        \frac{d}{d\epsilon}(f\circ\tau)
        &=
        g\bigl(\operatorname{grad}_{\tau(\epsilon)}f,
               \tau'(\epsilon)\bigr),\\
        \frac{d^2}{d\epsilon^2}(f\circ\tau)
        &=
        \frac{d}{d\epsilon}
        g\bigl(\operatorname{grad}_{\tau(\epsilon)}f,
               \tau'(\epsilon)\bigr)\\
        &=
        g\bigl(
            \nabla_{\tau'(\epsilon)}
            \operatorname{grad}_{\tau(\epsilon)}f,
            \tau'(\epsilon)
        \bigr).
    \end{align*}
    Here $\nabla_{\tau'(\epsilon)}$ denotes covariant differentiation
    along $\tau$, evaluated at $\epsilon$. If the $2$-jet of $f$
    vanishes at $\tau(0)$, both expressions vanish at $\epsilon=0$.
    We apply this with
    $f=\partial_sH|_{\{(s,t)\}\times M}$, $g=g_{s,t}$, and $\tau$
    the geodesic starting at $u(s,t)=\gamma(t)$ in the direction
    $\eta(s,t)$, to obtain the desired result.

    Before analyzing $\rho_1$, it will be useful to recall the meaning
    of $D_u(\eta)=0$. Fix $(s,t)\in\mathbb R\times[0,1]$ and consider
    again the geodesic $\tau=\tau(\epsilon)$ starting from $\gamma(t)$
    in the direction $\eta(s,t)$. Define a vector field $V$ along
    $\tau$ by
    \[
    \partial_su_\epsilon(s,t)
    +J_{s,t}(u_\epsilon(s,t))
    \bigl(
        \partial_tu_\epsilon(s,t)
        -X_{H_{s,t}}(u_\epsilon(s,t))
    \bigr).
    \]
    By definition, $D_u(\eta)_{s,t}$ equals the value at $\epsilon=0$
    of $\nabla_{\tau'(\epsilon)}V$. Thus $D_u(\eta)=0$ is equivalent
    to
    \[
    \left.\nabla_{\tau'(\epsilon)}V\right|_{\epsilon=0}=0
    \]
    for every $s,t$.

    Let us start by showing that $\rho_1'(0)=0$. We again differentiate
    under the integral sign and analyze the integrand for fixed
    $s,t$:
    \begin{equation}\label{eq-covariant-linearization}
    \begin{aligned}
        \frac{d}{d\epsilon}
        &\left(
        \omega\left(
            \frac{\partial u_\epsilon}{\partial s},
            \frac{\partial u_\epsilon}{\partial t}
            -X_{H_{s,t}}(u_\epsilon)
        \right)\right)\\
        &=
        (\nabla_{\tau'(\epsilon)}\omega)
        \left(
            \frac{\partial u_\epsilon}{\partial s},
            \frac{\partial u_\epsilon}{\partial t}
            -X_{H_{s,t}}(u_\epsilon)
        \right)\\
        &\quad+
        \omega\left(
            \nabla_{\tau'(\epsilon)}
            \frac{\partial u_\epsilon}{\partial s},
            \frac{\partial u_\epsilon}{\partial t}
            -X_{H_{s,t}}(u_\epsilon)
        \right)\\
        &\quad+
        \omega\left(
            \frac{\partial u_\epsilon}{\partial s},
            \nabla_{\tau'(\epsilon)}
            \left(
                \frac{\partial u_\epsilon}{\partial t}
                -X_{H_{s,t}}(u_\epsilon)
            \right)
        \right).
    \end{aligned}
    \end{equation}
    All terms on the right vanish at $\epsilon=0$, because
    $\partial_su=0$ and $\partial_tu-X_{H_{s,t}}(u)=0$.
    Here we did not need to use the fact that $\eta$ solves the
    linearized equation.

    Now consider $\rho_1''(0)$. Differentiating the right-hand side
    of \eqref{eq-covariant-linearization} and using
    $\partial_su=0$ and $\partial_tu-X_{H_{s,t}}(u)=0$ again, the
    only surviving term at $\epsilon=0$ is twice
    \[
    \omega\left(
        \nabla_{\tau'(\epsilon)}
        \frac{\partial u_\epsilon}{\partial s},
        \nabla_{\tau'(\epsilon)}
        \left(
            \frac{\partial u_\epsilon}{\partial t}
            -X_{H_{s,t}}(u_\epsilon)
        \right)
    \right).
    \]

    Finally,, using $D_u(\eta)=0$, we can rewrite this expression
    at $\epsilon=0$ as
    \[
    \omega\left(
        \left.
        \nabla_{\tau'(\epsilon)}
        \frac{\partial u_\epsilon}{\partial s}
        \right|_0,
        \left.
        \nabla_{\tau'(\epsilon)}
        \left(
            J_{s,t}\frac{\partial u_\epsilon}{\partial s}
        \right)
        \right|_0
    \right).
    \]
    By the Leibniz rule and $\partial_su=0$, this equals
    \[
    \omega\left(
        \left.
        \nabla_{\tau'(\epsilon)}
        \frac{\partial u_\epsilon}{\partial s}
        \right|_0,
        J_{s,t}
        \left.
        \nabla_{\tau'(\epsilon)}
        \frac{\partial u_\epsilon}{\partial s}
        \right|_0
    \right).
    \]
    Since $J_{s,t}$ is $\omega$-tame, this expression is nonnegative
    and vanishes precisely when its first argument vanishes.
    Together with $\rho''(0)=0$ and $\rho_2''(0)=0$, this implies
    \[
    \left.
    \nabla_{\tau'(\epsilon)}
    \frac{\partial u_\epsilon}{\partial s}
    \right|_0=0
    \]
    everywhere.

    At each fixed $t$, choose local coordinates near $\gamma(t)$.
    Since $\partial_su=0$ and
    $\left.\partial_\epsilon u_\epsilon\right|_0=\eta$, we have
    \[
    \left(
        \left.
        \nabla_{\tau'(\epsilon)}
        \frac{\partial u_\epsilon}{\partial s}
        \right|_0
    \right)^a
    =
    \left.
    \frac{\partial^2u_\epsilon^a}
         {\partial\epsilon\,\partial s}
    \right|_0
    +\Gamma^a_{bc}(g_t)\eta^b
      \frac{\partial u^c}{\partial s}
    =
    \frac{\partial\eta^a}{\partial s}.
    \]
    Thus $\partial_s\eta=0$. Since $\eta$ converges to $0$ as
    $s\to+\infty$, it follows that $\eta=0$, as desired.
\end{proof}

The following proposition guarantees that the regular Floer data $(H^-, J^-)$ and $(H^+, J^+)$ can be connected by a continuation datum under certain reasonable conditions.
First we introduce a definition.

\begin{defin}\label{def-acceptable-cont}
Consider a pair of  Floer data  $(H^-, J^-)$ and $(H^+, J^+)$  for $(\mathbb{L}_0,\mathbb{L}_1)$. An   \emph{acceptable continuation datum} consists of a smooth $H: \mathbb{R}\times [0,1] \times M\to \mathbb{R}$ such that
    \begin{itemize}
        \item $H$ is non-decreasing with respect to $s$, that is $\partial_s H\geqslant 0$.
        \item $H_{s,t}=H^-_t$ for $s\leq s_-$ and $H_{s,t}=H^+_t$  for $s\geq s_+$, for a pair $s_-<s_+$ of real numbers.
         \item If $\gamma: [0,1]\to M$ is a $1$-chord of both $H^-$ and $H^+$, then
         \begin{itemize} 
         \item if 
         $H^+(t,\gamma(t))=H^-(t,\gamma(t)),$ for all $t\in [0,1]$ then $\partial_sH\big|_{\{(s,t)\}\times M}: M\to \mathbb{R}$ vanishes to second order at $\gamma(t)$ for all $(s,t)\in \mathbb{R}\times [0,1].$
        \item If for some $t\in [0,1]$, we have $H^+(t,\gamma(t))>H^-(t,\gamma(t)),$ then there is an $s\in (s_-,s_+)$ such that $\gamma$ is not a $1$-chord of  $H_s.$
        \end{itemize}
    \end{itemize}
\end{defin}
The  space of acceptable continuation data is not convex. More precisely, the first two conditions and the fist half of the third are convex, while the second half of the third is not. However, given a pair of acceptable continuation data  we can interpolate between them so that the first two and a half conditions are satisfied. We can then achieve all the conditions by perturbing relative to the endpoint continuation maps by a perturbation supported away from the common periodic orbits for which  $H^+(t,\gamma(t))=H^-(t,\gamma(t)),$ for all $t\in [0,1]$ and in a region of $\bR\times[0,1]\times M$ where $\partial_sH>0$ and small enough to maintain this monotonicity condition. This will be useful for producing parametric data in later constructions. More generally, the same argument applies to a finite-dimensional parameter space with acceptable data prescribed on a collar of its boundary: one first extends the family in the larger convex space obtained by omitting the final common-chord-breaking condition, and then makes an arbitrarily small perturbation relative to the boundary to restore that condition.

\begin{prop}\label{prop-reg-monotone-cont}
      Let $(H^-, J^-)$ and $(H^+, J^+)$ be admissible Floer data for $(\mathbb{L}_0,\mathbb{L}_1)$ on $M$, with $H^- \leq H^+$ everywhere, and assume that $H^-\big|_t$ and $H^+\big|_t$ agree up to second order at all the points $(t,p)$ for which $H^-(t,p)= H^+(t,p)$. 
    \begin{enumerate}
        \item  For any pair $s_-<s_+$ there exists an acceptable continuation datum  $H: \mathbb{R}\times [0,1]\times M\to \mathbb{R}$. 
        \item  Given any such acceptable continuation datum, for every $\delta >0$, there exists an almost complex structure $J_{s,t}$ such that $(H_{s,t}, J_{s,t})$ forms a regular datum and such that $ J_{s,t}=J^-_t$ for $s\leq s_--\delta$ and  $ J_{s,t}= J^+_t$  for $s\geq s_++\delta$.
    \end{enumerate}
\end{prop}
\begin{proof} The first item is easy. We first construct a $G: \mathbb{R}\times [0,1]\times M\to \mathbb{R}$ satisfying the first two bullet points and the first half of the third. Let $\rho:\mathbb{R}\to [0,1]$ be a smooth non-decreasing function that is equal to $0$ for $s\leq s_-$ and to $1$ for $s\geq s_+.$ Then, we let $$G(s,t,x):= (1-\rho(s))H^-(t,x)+\rho(s)H^+(t,x).$$  The first two bullet points are clear. For the third one, we compute that $$(\partial_sG)(s,t,x):= \rho'(s)(H^+(t,x)-H^-(t,x)).$$ For $\gamma: [0,1]\to M$ with $H^+(t,\gamma(t))=H^-(t,\gamma(t)),$ for some $t\in [0,1]$, we are given that $H^-\big|_t$ and $H^+\big|_t$ agree up to second order at $\gamma(t).$ This implies the first half of the third bullet point.

We deal with the second half of the third bullet point. If there is a $\gamma: [0,1]\to M$ that is a $1$-chord of all $G_s$  such that  $H^+(t,\gamma(t))>H^-(t,\gamma(t))$ for some $t\in [0,1],$ then for some $s\in (s_-,s_+),$ we have $\partial_sG(s,t,\gamma(t))>0.$ We now perturb $G$ near $(s,t,\gamma(t))$ by a smooth $C^1$-small function $h$ supported inside a neighborhood where $\partial_sG>0$ and and which satisfies $d(h\big|_{s,t})_{\gamma(t)}\neq 0.$ Since there are only finitely many $1$-chords of $H^-$ that are all disjoint from each other, we can do this operation until the last bullet point is also satisfied.

The second item follows by a standard argument, following \cite[Proposition 4.2.18]{Schwarz1995} in the case of Hamiltonian orbits, for solutions which are not $s$-independent, and by  Lemma \ref{lem-const-reg} otherwise.

\end{proof}

We now deal with homotopies. Consider four admissible Floer data assigned to the vertices of a square, denoted $(H^{ij}, J^{ij})$ for $i,j \in \{-,+\}$. We assume that the four directed edges of the square are equipped with acceptable continuation data.

This configuration naturally induces two  broken continuation data interpolating from  $(H^{--}, J^{--})$ to  $(H^{++}, J^{++})$:
\begin{itemize}
\item $\mathcal{D}^0$: The broken data formed by composing along the path $(--) \to (-+) \to (++)$.
\item $\mathcal{D}^1$: The broken data formed by composing along the path $(--) \to (+-) \to (++)$.
\end{itemize}

\begin{defin}\label{def-acceptable-htpy}
An \emph{acceptable homotopy} between $\mathcal{D}^0$ and $\mathcal{D}^1$ is a smooth 1-parameter family of acceptable continuation data $(H^\tau, J^\tau)$  from $(H^{--}, J^{--})$ to $(H^{++}, J^{++})$ parametrized by $\tau \in [0,1]$, whose limit at the endpoints $\tau = 0$ and $\tau = 1$, is given by the broken data $\mathcal{D}^0$ and $\mathcal{D}^1$, and which is obtained by gluing near the ends.
\end{defin}

Note that the values $s_-$ and $s_+$ for an acceptable continuation map cannot be constant in the parameter $\tau$, but we suppress the dependence below. With this terminology, we  formulate the following proposition concerning existence and regularity of these homotopies under our second-order vanishing conditions. We omit the proof. 

\begin{prop}\label{prop-reg-monotone-htpy}
Let $(H^{ij}, J^{ij})$ for $i,j \in \{-,+\}$ be admissible Floer data for $(\mathbb{L}_0,\mathbb{L}_1)$ on $M$. Assume that for any two adjacent vertices $a, b \in \{--, -+, +-, ++\}$ (i.e., those differing by a single sign), the Hamiltonians $H^a\big|_t$ and $H^b\big|_t$ agree up to second order at all points $(t,p)$ where $H^a(t,p) = H^b(t,p)$.

Fix acceptable continuation data along the four edges of the square, and let $\mathcal{D}^0$ and $\mathcal{D}^1$ be the resulting broken continuation data.
\begin{enumerate}
    \item \textbf{Existence of Homotopy:} There exists a smooth 1-parameter family of Hamiltonians $H^\tau$ ($\tau \in [0,1]$) forming an acceptable homotopy between  $\mathcal{D}^0$ and $\mathcal{D}^1$.

    \item \textbf{Regularity:} Given any such acceptable homotopy of Hamiltonians $H^\tau$, and for any $\delta > 0$, there exists a family of almost complex structures $J^\tau$ such that the family $(H^\tau, J^\tau)$ forms a regular acceptable homotopy. Furthermore, for every $\tau \in (0,1)$, the almost complex structure satisfies $J^\tau_{s,t} = J^{--}_t$ for $s \leq s_- - \delta$ and $J^\tau_{s,t} = J^{++}_t$ for $s \geq s_+ + \delta$.
\end{enumerate}

 \qed
\end{prop}

\begin{proof}[Proof of Lemma~\ref{lem-fill-2cube}]
For each edge, the pair of corner Hamiltonians satisfies the hypotheses of Proposition~\ref{prop-reg-monotone-cont}: the two Hamiltonians agree on $W$ and there (since $g_1=g_2=\min(g_1,g_2)=\max(g_1,g_2)$ on $W$) they agree to second order, as both are the restriction of the same smooth function. This yields the second-order vanishing condition of Proposition~\ref{prop-reg-monotone-cont}, which guarantees that chords in $W$ are automatically regular as in Lemma~\ref{lem-const-reg}. The proposition then produces a regular monotone continuation datum on each edge.

For the square homotopy, apply Proposition~\ref{prop-reg-monotone-htpy} to the four corner Hamiltonians; the second-order agreement condition is satisfied for every pair that agrees somewhere by the same argument.
\end{proof}

\subsection{Acyclicity and conclusion}\label{subsec-conclusion}

We now assemble the ingredients established in the previous sections to prove the main descent theorem.

\begin{proof}[Proof of Theorem \ref{thm-descent}] 
As mentioned in Section \ref{sec-statement-descent}, it suffices to prove Theorem \ref{thm-descent-linear}. By Lemma \ref{lem-descent-2-to-N}, it suffices to prove the case $N=2$. That is, we need to show that for $K_1, K_2$ weakly Poisson commuting subsets of $M$, the canonical map
 \[
  \begin{tikzcd}
\Fuk^*_{M}(K_1\cup K_2)(\mathbb{L}_0,\mathbb{L}_1)  \arrow{rr}{} && \Cech(\Fuk_{M}; K_1,K_2)(\mathbb{L}_0,\mathbb{L}_1)
\end{tikzcd}
\]
is a quasi-isomorphism.

By Proposition \ref{prop-tel-compare} it suffices to prove  that the map
\[
\widehat{\operatorname{tel}}(\mathcal{C}_{K_1\cup K_2}) \to \operatorname{cocone}\left(\widehat{\operatorname{tel}}(\mathcal{C}_{K_1})\oplus\widehat{\operatorname{tel}}(\mathcal{C}_{K_2})\to \widehat{\operatorname{tel}}(\mathcal{C}_{K_1\cap K_2})\right) 
\]
is a quasi-isomorphism. We now prove this statement.

Let $K_1,K_2$ be weakly Poisson commuting compact subsets with defining sequences $f_{m,i}:M\to\mathbb{R}$ as in Definition~\ref{def-weak-inv}. 

\textbf{Step 1: Construction of the adjusted functions $g_{m,i}$.}
Postcomposing the defining functions with appropriate functions $\bR\to\bR$, we may assume the functions $f_{m,i}$ form acceleration data for $K_m$ in the sense that the sequence $f_{m,i}$ converges to $0$ on $K_m$ and to $\infty$ on $M\setminus K_m$. 

For each $i$, we apply Lemma~\ref{lem-main-separated-admissible} to the pair $(f_{1,i},f_{2,i})$ with bounds $\epsilon_i>0$ inductively chosen to be small enough so that the resulting functions satisfy $g_{m,i}<g_{m,i+1}$ and the acceleration datum remains cofinal. Denote the output of Lemma~\ref{lem-main-separated-admissible} by $g_{1,i},g_{2,i}:[0,1]\times M\to\mathbb{R}$. 

By the conclusions of Lemma~\ref{lem-main-separated-admissible}, these Hamiltonians satisfy Approximation from below, Disjointness, Smooth Matching, Separation, and Admissibility.

\textbf{Step 2:  Construction of a cubic diagram.}
For each index $i$, we define Hamiltonians corresponding to the four compact subsets $K_1\cup K_2$, $K_1$, $K_2$, and $K_1\cap K_2$ as follows:
\begin{align*}
H_i^{K_1\cup K_2} &:= \min(g_{1,i},g_{2,i}), \\
H_i^{K_1} &:= g_{1,i}, \\
H_i^{K_2} &:= g_{2,i}, \\
H_i^{K_1\cap K_2} &:= \max(g_{1,i},g_{2,i}).
\end{align*}

These Hamiltonians are admissible and non-degenerate by Lemma~\ref{lem-main-separated-admissible}, and they form the vertices of a square diagram for each fixed $i$. We then construct the cubic diagram of Figure~\ref{fig-cubic-diagram}, with the index $i$ running over the acceleration datum.


For each of the four subsets $K\in\{K_1\cup K_2, K_1, K_2, K_1\cap K_2\}$, the horizontal arrows $CF(H_i^{K})\to CF(H_{i+1}^{K})$ are the continuation maps arising from the inequalities $H_i^{K}\leq H_{i+1}^{K}$ (which follow from $g_{m,i}<g_{m,i+1}$ and the monotonicity of $\min$ and $\max$). These form the four horizontal rays of the diagram.
The vertical slices are constructed in Lemma~\ref{lem-fill-2cube}.

\textbf{Step 3: Extension to three dimensions.}
The next step is to produce a map from $CF(H_i^{K_1 \cup K_2})$ to $CF(H^{K_1 \cap K_2}_{i+1})$ corresponding to the inner diagonal in Figure \ref{fig-cubic-diagram}. We shall not elaborate on the details of the construction, but the basic idea is as follows: there are 6 ways of moving along the edges of a cube from $(0,0,0)$ to $(1,1,1)$, which correspond to six different chain maps from $CF(H_i^{K_1 \cup K_2})$ to $CF(H^{K_1 \cap K_2}_{i+1})$ obtained as a composition of continuation maps.  These six paths form the vertices of a hexagon whose edges correspond to a choice of a face. Since any face contains either $(0,0,0)$ or $(1,1,1)$, and there is a unique path from the face to whichever one of these two vertices does not lie on it, we obtain a chain homotopy associated to each such edge. The sum of these chain homotopies is in fact a chain map (of degree $-1$), as all its boundaries cancel. The goal is to construct a null-homotopy for it.

The procedure is completely similar to the one we followed for the square: each vertex of the hexagon arises from a broken continuation (with three components) which is acceptable and regular. Each edge is given by a family of broken curves given by concatenating a regular and acceptable datum with a regular $1$-parametric family of acceptable data. By gluing, we obtain a family of regular and acceptable data, parametrized by a neighbourhood of the boundary, which we can the extend to the remainder of the hexagon by relative extension argument following Definition \ref{def-acceptable-cont}. Finally, perturb the almost-complex structures generically, relative
to the boundary collar, to make the two-parameter family regular.

\textbf{Step 4: Acyclicity.}

Lemma~\ref{lem-main-separated-admissible} and Lemma~\ref{lem-var-711} imply that  for each fixed $i$, the square
\begin{align}\label{eq-slice-i}
\xymatrix{
CF(H_i^{K_1\cup K_2})\ar[r]\ar[d]\ar[dr] & CF(H_i^{K_1})\ar[d] \\
CF(H_i^{K_2})\ar[r] & CF(H_i^{K_1\cap K_2})
}
\end{align}
equipped with the regular Floer data from Lemma~\ref{lem-fill-2cube} is acyclic.

By the usual filtration argument, acyclicity of each slice implies acyclicity of the full telescope. The acyclicity of the completion follows from the fact that we are working with free chain complexes (see \cite[Corollary 2.3.6]{Varolgunes2021}).

\end{proof}
\subsection{A remark}

\begin{rem}
    We would like to alert the reader to one point: \cite[Lemma 3.2.12]{Varolgunes2021} is only true as stated if the manifold with corners $X$ in that statement is at most $2$-dimensional, although the rest of the argument in that section is accurate. 
    
    Let us first explain the issue in a simpler situation. Given a non-negative smooth function on the boundary of $[0,1)^n$ (in the sense that it is the restriction of a smooth function on a neighbourhood of this subset of $\bR^n$), can one extend it to a non-negative smooth function on the entire $[0,1)^n$? The answer is yes for $n\leq 2.$ For $n=1$, this is trivial. For $n=2$, we do convex linear combinations along the antidiagonal lines, which one can check results in a smooth (and non-negative) extension. Here is a counterexample for $n=3$.

Call the Cartesian coordinates $x,y,z$. On the non-negative quadrants of the $xy, yz, xz$ coordinate planes, consider the functions $(x-y)^2, (y-z)^2, (x-z)^2$, respectively. This defines a smooth function on the boundary of  $[0,1)^3$. Let $F$ be any smooth extension to $[0,1)^3$. Then, we can compute the directional derivative of $F$ at the origin in the positive diagonal direction. This gives $0$. We then also compute the second derivative of $F$ in the same direction at the origin. This gives the negative number $-6$. Hence $F$ must dip down and be negative somewhere along the diagonal.
    
    The proof of \cite[Lemma 3.2.12]{Varolgunes2021} is wrong due to a misunderstanding of the Whitney extension theorem (for which the third author apologizes). A correct proof for $n\leq 2$ is not difficult. It can be done with the same idea as the simpler extension result above. We choose independent extensions near the corners and edges of $X$ (using the extension procedure described above) along with an arbitrary valid monotone datum in the interior of $X$. These are then patched together using a non-negative partition of unity indexed by the faces (including the top dimensional one) of $X$. All the necessary properties are satisfied.

    For $n\geq 3,$ the statement needs to be modified. First, we require strict inequality $H_0<H_1$ (in fact, $\int H_0(x,t)dt< \int H_1(x,t)dt$ is enough) instead of $H_0\leq H_1$. Then, we change the monotonicity requirement for continuation map data to what we are introducing in this paper (see Section~\ref{sec:monotonicity-inequality}), namely $$\int \min_{x\in M}\frac{\partial H}{\partial s}(x,s,t)dsdt> 0$$ instead of $\frac{\partial H}{\partial s}\geq 0.$ With these modifications \cite[Lemma 3.2.12]{Varolgunes2021} becomes correct for $n\geq 3$ as well (with a much simpler proof).

    The only place where a modification needs to be made in \cite{Varolgunes2021} based on this correction is that in \cite[Section 3.4]{Varolgunes2021}, an acceleration datum for $C_I$ should have strict inequalities and the inequalities $H_k^{C_I}\leq  H_k^{C_{I'}}$ should be replaced with $H_k^{C_I}<  H_k^{C_{I'}}$. The proof of the Mayer-Vietoris property only needs \cite[Lemma 3.2.12]{Varolgunes2021} as written (and not the slightly modified version) for $n\leq 1.$ We define $SH_M(K_1,K_2)$ using strict inequality diagrams but then we show by the sandwiching argument that it can be computed using the diagrams that are constructed in \cite[Section 4]{Varolgunes2021}. These diagrams relax $H_k^{C_I}<  H_k^{C_{I'}}$ to $H_k^{C_I}\leq  H_k^{C_{I'}}$ but the only use of  \cite[Lemma 3.2.12]{Varolgunes2021} as written is for $n=1$ as we are considering two subsets and all other inequalities are chosen to be strict. In particular, the $n=2$ case of \cite[Lemma 3.2.12]{Varolgunes2021} as written is not actually needed anywhere in the proof of the Mayer-Vietoris property.

    Finally, we note that the definition of $SH_M(K)$ in \cite[Section 3.3.2]{Varolgunes2021} allows for acceleration data of the form $H_1\leq H_2\leq\ldots$, and the sandwiching argument in that section is accurate. On the other hand it is not clear if there is anything to be gained by allowing equality in the acceleration data in this way, the strict inequalities seem conceptually clearer, thus avoiding the need to introduce a different definition for the purpose of proving the Mayer-Vietoris property.
\end{rem}


\appendix

\section{Trees and punctured discs}
\label{sec:trees-discs}

\subsection{Conventions about ribbon trees}
A ribbon tree (or planar tree) is a tree equipped with a proper embedding in the plane, considered up to isotopy. Such a tree inherits a cyclic ordering of the edges adjacent to each vertex from the counter-clockwise orientation of the plane. In our setting, these cyclic orderings encode the order in which marked points appear on the boundary of a disc.

\begin{figure}[h]
  \centering
  \begin{tikzpicture}
    \begin{scope}
    \node[circle, fill=black, inner sep=0, minimum size = 6] (v1) at (2.5,-1) {};
    \foreach \i in {1,...,4}
    \draw (\i,0) -- (v1);
      \node[label=above:{$e_{n}$}] (en) at (1,0) {};
      \node[label=above:{$e_{n-1}$}] (en-1) at (2,0) {};
    \end{scope}
      \begin{scope}[shift={(5,0)}]
       \node (e...) at (0,0) {$\cdots$};

       \begin{scope}[shift={(0,-1)}]
         \node[circle, fill=black, inner sep=0, minimum size = 6, label=right:{$v_{root}$}] (v) at (0,-2.5) {};
         \draw (v) --  +(-90:1);
         \node[label=below:{$e_{0}$}] (e0) at (0,-3.5) {};
         \foreach \i in {110, 100, ..., 70} \draw[dashed] (v)-- +(\i:0.5);
        \end{scope}
       \end{scope}

    \begin{scope}[shift={(-1,0)}]
    \node[label=above:{$e_{2}$}] (en-1) at (7,0) {};
    \node[label=above:{$e_{1}$}] (en) at (8,0) {};
    \node[circle, fill=black, inner sep=0, minimum size = 6] (v'') at (7.5,-2) {};
    \draw (en)--(v'') -- (en-1);

    \end{scope}

\draw[dashed] (v1)--(v)--(v'');    
    
  \end{tikzpicture}
  \caption{A stable ribbon tree with $n$ inputs and one output. The planar embedding determines a cyclic ordering of the edges adjacent to any given vertex and the choice of output external edge upgrades it to a linear order.}
  \label{fig:stable_tree}
\end{figure}
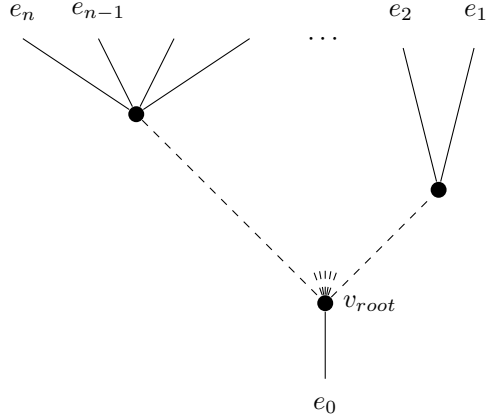

In what follows we will allow trees $T$ to have edges that have one endpoint. We call such edges \emph{external edges} and denote them by $E_{ext}(T)$ and call the edges with two endpoints \emph{internal edges}, which we denote by $E_{int}(T)$. We write $E(v)$ for the set of edges adjacent to a vertex $v$, and denote the \emph{valency} of $v$ by $\upsilon(v):=|E(v)|$. Each vertex must have valency at least two.

\begin{defin}\label{defRibbonTree}
  For a natural number $k\geq 1$, a \emph{pre-stable ribbon tree $T$ with $k$ inputs and $1$ output} consists of a ribbon tree $T$ with $k+1$ external edges, one of which is designated as an \emph{output}. We refer to the other external edges as \emph{inputs}. Given a pre-stable ribbon tree $T$, we obtain a canonical bijection $\{0,\dots,k\}\to E_{ext}(T)$ that sends $0$ to the output and is increasing with respect to the counter-clockwise order induced on the inputs.

A pre-stable ribbon tree $T$ is called \emph{stable} if the valency of each vertex is at least three. 
\end{defin}

For each vertex $v$, we denote by $e_{out}(v)\in E(v)$ the unique outgoing edge, that is the edge lying on the minimal path connecting $v$ to the output edge. The edges adjacent to $v$ that are not outgoing are called incoming edges:
\begin{equation}
E_{in}(v) := E(v)\setminus\{e_{out}(v)\}.
\end{equation}

Each internal edge is the outgoing edge of exactly one vertex denoted by $v^-(e)$. The other vertex of $e$ is denoted by $v^+(e)$.

\begin{figure}[h]
  \centering
  \begin{tikzpicture}
                 \node[circle, fill=black, inner sep=0, minimum size = 6, label=right:{$v^-(e)$}] (v+e) at (-2,-2) {};
              \foreach \i in {110,100,80,70} \draw[dashed] (v+e)-- +(\i:0.5);

    \node[circle, fill=black, inner sep=0, minimum size = 6, label=right:{$v^+(e)$}] (v-e) at (-2,-3) {};
    \draw (v+e)--(v-e);
    \node[label=right:{$e$}] (e) at (-2,-2.5) {};
    \foreach \i in {110, 100, 80,70} \draw[dashed] (v-e)-- +(\i:0.5);
   \draw (v-e)-- +(-90:0.5);
    \draw[dashed] (v-e)-- +(-90:1);

    \begin{scope}[shift={(0,-.5)}]
     \node[label=above:{$E_{in}(v)$}] (Einv) at (2,-1.5) {};
    \node[circle, fill=black, inner sep=0, minimum size = 6, label=right:{$v$}] (v) at (2,-2.5) {};
    \node[label=right:{$e_{out}(v)$}] (eout) at (2,-3) {};
     \draw (v)-- +(-90:0.5);
     \draw[dashed] (v)-- +(-90:1);
     \foreach \i in {110, 100, 80,70}
     {\draw (v)-- +(\i:0.5);
       \draw[dashed] (v)-- +(\i:1);
     };
\end{scope}

  \end{tikzpicture}
  \caption{Conventions for labelling edges and vertices}
  \label{fig:edges_vertices_labels}
\end{figure}
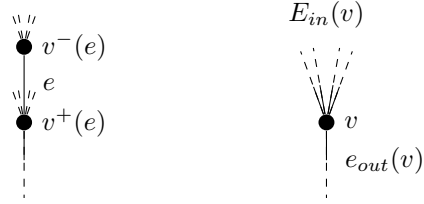

Below, for a pre-stable tree $T$, we abbreviate $v\in T$ to mean $v\in V(T)$.
\subsection{Pre-stable punctured discs with strip-like ends}
We now introduce our conventions for discs with boundary marked points:

\begin{defin}\label{def-strip-like}
  A \emph{punctured disc} is a Riemann surface $\Sigma$ which is the complement of finitely many points $P$ on the boundary of the closed unit disc, which we denote $\overline{\Sigma}$. The elements of $P$ are referred to as the boundary marked points or punctures of $\Sigma$.

  A \emph{choice of strip-like ends on $\Sigma$} is a decomposition $P=P_{in}\sqcup P_{out}$ into inputs and outputs, and a choice for each $p\in P$ of a  biholomorphic embedding
\begin{align}
  \epsilon_p^{-}:(-\infty,0]\times[0,1] & \to\Sigma  \qquad \textrm{if } p \in P_{in}\\
  \epsilon_p^{+}:[0,\infty)\times[0,1]& \to\Sigma \qquad \textrm{if } p \in P_{out}
\end{align}
 onto a punctured neighborhood of $p$ in $\overline{\Sigma}$, mapping any point $(s,t)$ with $t=0$ or $t=1$ to $\partial \overline{\Sigma}$. We require that:
\begin{itemize}
  \item The images of the strip-like ends are pairwise disjoint
  
  \item In the case of one input and one output, $\epsilon^{\pm}$ both extend to a biholomorphism from $\bR\times[0,1]$ to $\Sigma$
\end{itemize}

A \emph{punctured disc with $k_{in}$ inputs and $k_{out}$ outputs} is a punctured disc with $|P_{in}| = k_{in}$ and $|P_{out}| = k_{out}$, together with a choice of strip-like ends (we only use the case $k_{out} = 1$ in this paper).
\end{defin}

Next, we discuss the notion of pre-stability, which can be formulated abstractly in terms of nodal Riemann surfaces, but which we prefer to describe in terms of trees labelled by punctured discs:

\begin{defin}\label{def-pre-stable-disc}
  A \emph{pre-stable disc $\Sigma$ with $k$ inputs and one output} consists of a pre-stable ribbon tree $T$ with $k$ inputs and one output, and, for each vertex $v\in T$, a punctured disc $\Sigma_v$ with $|E_{in}(v)|$ inputs and $1$ output, together with a bijection between the edges adjacent to each vertex $v$ and the boundary marked points of the corresponding disc:
  \begin{align}
    E(v)\to & P_{\Sigma_v} \\
    e\mapsto & p_e.
  \end{align}
We require this bijection to:
\begin{itemize}
  \item Map the outgoing edge to the output marked point
  \item Preserve the cyclic ordering of edges at $v$ and marked points on $\partial\Sigma_v$
\end{itemize}
We denote the inverse bijection by $p\mapsto e_p$.

A \emph{stable disc with $k$ inputs and 1 output} is a pre-stable disc whose underlying ribbon tree $T$ is stable.
\end{defin}

We refer to the surfaces $\Sigma_v$ as the \emph{components} of $\Sigma$. For an internal edge $e\in E_{int}(T)$ we denote by $p^{\pm}_e$ the corresponding boundary marked points in $\Sigma_{v^{\pm}(e)}$.

\begin{defin}\label{dfGluing0}
Consider a pre-stable disc $\Sigma$ with $k$ inputs and one output, with underlying tree $T$, and fix an internal edge $e$ in $T$.
  Given a parameter $r\in (0,1)$ we define the pre-stable curve $\Gamma_{r,e}(\Sigma)$ to be given by the following data:
  \begin{enumerate}
  \item The tree $T/e$ is obtained from $T$ by removing $e$ from the edge set and quotienting the vertex set by $v^{-}(e)\sim v^{+}(e)$. We denote by $v_{new}\in T/e$ the vertex corresponding to the equivalence class of $v^{-}(e)\sim v^{+}(e)$. The ribbon structure and the output external edge of $T/e$ are induced from $T$.
    \item The surface $\Sigma_{v_{new}}$ which is defined by gluing
\begin{equation}
\left(\Sigma_{v^{+}(e)}\setminus \epsilon^-_{p_e^+}((-\infty,\ln r)\times [0,1])\right)\sqcup_\sim \left(\Sigma_{v^{-}(e)}\setminus \epsilon^+_{p_e^-}((-\ln r,\infty)\times [0,1])\right),
\end{equation}
along their boundaries by the relation
\begin{equation} \label{eq:gluing_along_boundary}
{\epsilon^-_{p_e^+}(\ln r,t)\sim \epsilon^+_{p_e^-}(-\ln r,t)}.
\end{equation}
The marked points of $\Sigma_{v_{new}}$ have strip-like ends which are induced from those on $\Sigma_{v^-(e)}$ and $\Sigma_{v^+(e)}$. 
\item The surfaces $\Sigma_v$ for $v \neq v_{new} \in V(T/e)$ are given by the corresponding component of $\Sigma$.
  \end{enumerate}
Let us also declare that $\Gamma_{0,e}(\Sigma)=\Sigma$, i.e. gluing with $0$ gluing parameter does not change the curve.
\end{defin}

\begin{defin}\label{def-gluing-general} Given $\vec{r}\in [0,1)^{E_{int}(T)}$, we denote by $\Gamma_{\vec{r}}(\Sigma)$ the gluing of $\Sigma$ with gluing parameter given by $\vec{r}$. That is, we pick an order on the edges of $T$ and we inductively glue each edge $e$ whose associated parameter $r_e$ is non-zero. The result does not depend on the order in which gluing is performed. For a real number $r\in [0,1)$ we write
\begin{equation}
\Gamma_r(\Sigma):=\Gamma_{(r,\dots,r)}(\Sigma).
\end{equation}
That is, we glue all the edges with the same gluing parameter $r$.\end{defin}

\subsection{Moduli spaces of discs and stabilization} \label{sec:stab-moduli-discs}

For the next definition, we consider a pair $\Sigma_1 = (T_1,\{\Sigma_v\}_{v\in T_1})$ and $\Sigma_2 = (T_2,\{\Sigma_v\}_{v\in T_2})$ of pre-stable discs, each with $k$ inputs and $1$ output.

\begin{defin}\label{def-isom-disc}
An \emph{isomorphism} $\psi=(\phi,\{\psi_v\})$ from $\Sigma_1$ to $\Sigma_2$ consists of an isomorphism $\phi:T_1\to T_2$ of pre-stable ribbon trees with $k$ inputs and $1$ output, together with, for each $v\in V(T_1)$, a biholomorphism $\psi_v:\Sigma_v\to\Sigma_{\phi(v)}$ such that:
\begin{itemize}
  \item the induced map $P_{\Sigma_v}\to P_{\Sigma_{\phi(v)}}$ on boundary marked points agrees with the bijection $E(v)\to E(\phi(v))$ induced by $\phi$;
  \item $\psi_v$ intertwines the strip-like end at each marked point of $\Sigma_v$ with the strip-like end at the corresponding marked point of $\Sigma_{\phi(v)}$.
\end{itemize}
We say that $\Sigma_1$ and $\Sigma_2$ are \emph{isomorphic} if there exists an isomorphism between them.
\end{defin}

Denote by $\cD_{k+1}$ the set of isomorphism classes of pre-stable discs with $k$ inputs and one output. For a pre-stable ribbon tree $T$ with $k$ inputs and one output, denote by $\cD^T_{k+1}\subset \cD_{k+1}$ the subset consisting of discs whose underlying tree is $T$. If $T$ is a stable ribbon tree, then $\cD^T_{k+1}$ is equipped with its natural topology induced by the $C^\infty$ topology on the space of complex structures on a given topological type, and on the space of holomorphic embeddings. 

\begin{rem}
Note that we consider $\cD_{k+1}$ merely as a set, as there is not natural Hausdorff topology on $\cD_{k+1}$. In addition tothe difficulties caused by unstable components, even if we consider a sequence of stable discs converging (in a reasonable topology) to a stable disc, there will be an ambiguity in the choice of strip-like ends on the limiting disc whenever breaking is involved.
\end{rem}

Denote by $\overline{\cR}_{k+1}$ the standard Deligne-Mumford-Stasheff moduli space of stable nodal disks with $k+1$ unlabeled boundary marked points, one of which is designated as output. Note that the marked points automatically inherit a linear order. This is a real manifold with corners that compactifies the moduli space $\cR_{k+1}$ of discs with $k+1$ boundary marked points.

The stabilization map
\begin{equation}
\fF:\cD_{k+1}\to \overline{\cR}_{k+1}
\end{equation}
takes a pre-stable disc with $k$ inputs and one output and produces a stable nodal disc with $k+1$ boundary marked points by:
\begin{itemize}
\item compactifying each disc component by adding its marked points,
\item identifying the corresponding points as specified by the tree structure,
\item declaring the output of the disk corresponding to the root vertex as the output of the nodal disc, and
\item contracting any unstable disc components (those with fewer than 3 special points) and declaring that the output maps to the output under the contraction maps.
\end{itemize}

It will be convenient to set $\overline{\cR}_{1}$ to be a point. In this way, the space $\overline{\cR}_{k+1}$ become the underlying spaces of an operad, whose unit is the singleton element of $\overline{\cR}_{1}$; we shall generalize this to multiple colors in Definition \ref{def:mult-forg-map}.


\section{Monotonicity condition}
\label{sec:monotonicity-inequality}

In the proof of \cite[Proposition 2.20]{Abouzaid2023} we erroneously rely on smooth approximation to claim that the inclusion of smooth chains on the space of 0-multimorphisms into the space of continuous chains defines a homotopy equivalence. This is incorrect. The issue is that the monotonicity requirement (2.1) of Definition 2.1 $dH_x\wedge\alpha\geq 0$ from \cite{Abouzaid2023} is defined by a non-strict inequality. In general, a continuous cube cannot be approximated by smooth cubes while preserving this monotonicity property. Note that we cannot impose a strict inequality. Indeed, unless the underlying Riemann surface $\Sigma$ has only one input, $\alpha$ will have to vanish at some points. 

To fix this,  we had to slightly relax our notion of  monotonicity. We replace the monotonicity condition (2.1) of Definition 2.1 from \cite{Abouzaid2023} with an integrated version
\begin{equation}
    \int_{\Sigma_v} \min_{x\in M} \left(d_{\Sigma_v} (H_v(\cdot, x)) \wedge \alpha_v\right) > 0,
\end{equation} as already written in Condition \eqref{eq:strict-monotonicity} of Definition \ref{def:pre-multimorphism-H}. It is worth explaining this condition a little bit. Define $\Sigma:=\Sigma_v$, $H:=H_v$ and $\alpha:=\alpha_v.$
Here we can alternatively think of $dH_x:=d_{\Sigma_v} (H_v(\cdot, x))$ as the pullback of $dH,$ which is a one form on $\Sigma\times M$, to $\Sigma$ embedded as the submanifold $\Sigma\times\{x\}$.
To make sense of the min in the integrand, note that the complex structure on $\Sigma$ induces an orientation. Let $\omega$ be a positive area form with respect to this orientation. Then, $dH_x \wedge \alpha=f\mid_{\{x\}\times \Sigma} \omega$ for some unique $f:\Sigma\times M\to\bR$. At every $p\in \Sigma,$ we define the two-form
$$\min_{x\in M} \left(d_{\Sigma_v} (H_v(\cdot, x)) \wedge \alpha_v\right)= (\min_xf(p,x))\omega.$$

The definition is readily seen to be independent of the choice of $\omega$. Note that the integrand is continuous and, by the other assumptions in the definition, it is compactly supported. So, the integral makes sense. 



For this to be a valid replacement we need to verify that it is a convex condition and that it guarantees topological energy is bounded from below by geometric energy. 

We verify convexity: Suppose $H_1,H_2$ satisfy the inequality \eqref{eq:strict-monotonicity}. Then for any $t\in[0,1]$ we have 
\begin{align*}
    \int_{\Sigma}\min_{x\in M}\{d(tH_1+(1-t)H_2)_x\wedge\alpha\}&\geq t\int_{\Sigma}\min_{x\in M}\{d(H_1)_x\wedge\alpha\}+(1-t)\int_{\Sigma}\min_{x\in M}\{d(H_2)_x\wedge\alpha\}\\
    &>0
\end{align*}

We now verify that the inequality \eqref{eq:strict-monotonicity} guarantees that topological energy is bounded from below by geometric energy. For this note that for any Floer solution $u$ with inputs inhomogeneous term given by $\alpha\otimes H$ and  we have 
\begin{align*}
E^{top}(u)-E^{geom}(u)&= \int_{\Sigma}d_\Sigma(u^*H)\wedge\alpha\\
&\geq\int_{\Sigma}\min_{x\in M}\left(d_{\Sigma} (H(\cdot, x)) \wedge \alpha\right)
\\&>0.
\end{align*}

\bibliographystyle{alpha}
\bibliography{large-bib}

\end{document}